\documentclass[12pt,a4paper]{amsart}
\usepackage{color,graphicx,fullpage,amssymb,tikz}
\usepackage{url}
\usepackage[colorlinks=true]{hyperref}
\usepackage{doi}
\usepackage[skip=10pt]{caption}
\newtheorem{theorem}{Theorem}[section]
\newtheorem{lemma}[theorem]{Lemma}
\newtheorem{proposition}[theorem]{Proposition}
\newtheorem{corollary}[theorem]{Corollary}

\theoremstyle{definition}
\newtheorem{definition}[theorem]{Definition}
\newtheorem{remark}[theorem]{Remark}
\newtheorem{example}[theorem]{Example}

\newcommand{\Q}{\mathbb{Q}}
\newcommand{\Qp}{\mathbb{Q}_p}
\newcommand{\Cp}{\mathbb{C}_p}
\newcommand{\Zp}{\mathbb{Z}_p}
\newcommand{\N}{\mathbb{N}}
\newcommand{\R}{\mathbb{R}}
\newcommand{\Z}{\mathbb{Z}}
\newcommand{\C}{\mathbb{C}}
\newcommand{\F}{\mathbb{F}}
\newcommand{\dd}{\mathrm{d}}
\newcommand{\ii}{\mathrm{i}}
\newcommand{\e}{\mathrm{e}}
\newcommand{\f}{\mathrm{f}}
\newcommand{\h}{\mathrm{h}}

\newcommand{\letpprime}{Let $p$ be a prime number}

\newcommand{\M}{\mathcal{M}}

\newcommand{\bDSq}{\overline{\DSq}}
\newcommand{\prodvec}[2]{u_{#1}^T\Omega_0u_{#2}}
\newcommand{\prodvecprime}[2]{(u'_{#1})^T\Omega_0u'_{#2}}
\newcommand{\prodvecsecond}[2]{(u''_{#1})^T\Omega_0u''_{#2}}
\renewcommand{\le}{\leqslant}
\renewcommand{\ge}{\geqslant}

\DeclareMathOperator{\ord}{ord}
\DeclareMathOperator{\digit}{digit}
\DeclareMathOperator{\Sq}{Sq}
\DeclareMathOperator{\DSq}{DSq}

\newenvironment{euproof}[1][Proof]{\begin{proof}[#1]\begin{enumerate}\renewcommand{\theenumi}{\alph{enumi}}}{\qedhere\end{enumerate}\end{proof}}

\numberwithin{equation}{section}

\title{$p$-adic symplectic geometry of integrable systems and Weierstrass-Williamson theory II}
\author[Luis Crespo, \'Alvaro Pelayo]{Luis Crespo\,\,\,\,\,\, \'Alvaro Pelayo}

\address{Luis Crespo,
	Departamento de Matem\'{a}ticas, Estad\'{i}stica y Computaci\'{o}n, Universidad de Cantabria, Av.~de Los Castros 48, 39005 Santander, Spain}
\email{luis.cresporuiz@unican.es}

\address{\'Alvaro Pelayo,
	Facultad de Ciencias Matem\'aticas,
	Universidad Complutense de Madrid, 28040 Madrid, Spain, and Real Academia de Ciencias Exactas, F\'isicas y Naturales, Spain}
\email{alvpel01@ucm.es}

\begin{document}
	
\maketitle

\begin{center}
	\emph{In memory of Professor Vladimir Voevodsky (1966--2017).}
\end{center}

\begin{abstract}
	This paper is a sequel to \cite{CrePel-integrable}, in which we shall give proofs of several results stated in \cite{CrePel-integrable} (Theorems D--L) which, for brevity and clarity, we postponed to this sequel paper. These results were the following: for any prime number $p$, first we show that every $2$-by-$2$ symmetric matrix with coefficients in $\Qp$ can be reduced to a canonical form, and we give the exact numbers of families of normal forms with one parameter and of isolated normal forms, which depend on $p$. Then we make the same analysis for $4$-by-$4$ matrices. We also prove that, for higher size, the number of families of normal forms of matrices, even in the non-degenerate case, grows almost exponentially with the size. The paper can be read independently of \cite{CrePel-integrable} as we recall the statements of \cite{CrePel-integrable} that we shall prove here. The statements and proofs of the present paper are of an algebraic and arithmetical nature, and rely mainly on Galois theory of $p$-adic extension fields.
	
\end{abstract}

\section{Introduction}

$p$-adic symplectic geometry is a new subject which has the goal of developing the ideas of symplectic geometry, but taking coefficients in the non-Archimedean field of $p$-adic numbers instead of the field of real numbers. The development of $p$-adic symplectic geometry was proposed by \'A. Pelayo, V. Voevodsky and M. Warren \cite[Section 7]{PVW}, who also had in mind its implementation in the setting of homotopy type theory and Voevodsky's Univalent Foundations \cite{APW,PelWar2}.

The present paper, as its title indicates, is a sequel to \cite{CrePel-integrable}. Our goal in \cite{CrePel-integrable} was to state the classification of local normal forms of $p$-adic integrable systems on symplectic $4$-manifolds up to local symplectomorphisms, and of matrices up to congruence via a symplectic matrix. Due to the very nature of the $p$-adic numbers, which form a non-Archimedean field, the proofs of the classification results concerning matrices are quite elaborate, and in order not to cloud the presentation of the results in \cite{CrePel-integrable}, we felt that these proofs were better organized in a separate paper, which is the present one.

In a sense, this second paper is more fundamental than the first one, as it contains the technical heart of the classification theorems, which involves ideas and techniques developed for the first time in this context of symplectic geometry of $p$-adic matrices and $p$-adic integrable systems. The statements and proofs are of an algebraic and arithmetical nature, and rely mainly on Galois theory of $p$-adic extension fields. Hence both papers should be equally useful, with the first one containing full statements of all the classification results, and this one containing proofs with techniques and ideas which had not been used before in this setting.

In order to make the present paper as self-contained as possible we are going to review the statements of the results which we prove in it, and which come from \cite{CrePel-integrable}. For clarity when we state the results for which we provide proofs in this sequel, we will use the same label as in \cite{CrePel-integrable}, that is, we will write ``\cite[Theorem D]{CrePel-integrable}'' after the formal statement label. As announced in \cite{CrePel-integrable}, in the present paper we prove \cite[Theorems D--L]{CrePel-integrable}, which we state here, respectively, as:
\begin{itemize}
	\item Theorem \ref{thm:williamson} (\cite[Theorem D]{CrePel-integrable}): gives a classification of $2$-by-$2$ matrices over $\Qp$ up to congruence via a symplectic matrix.
	\item Theorem \ref{thm:num-forms2} (\cite[Theorem E]{CrePel-integrable}): gives the number of normal forms of $2$-by-$2$ matrices over $\Qp$ up to congruence via a symplectic matrix.
	\item Theorem \ref{thm:williamson4} (\cite[Theorem F]{CrePel-integrable}): gives a classification of $4$-by-$4$ matrices over $\Qp$ up to congruence via a symplectic matrix in the non-degenerate case, when $\Omega_0^{-1}M$ has no multiple eigenvalues, where $\Omega_0$ is the matrix of the standard symplectic form.
	\item Theorem \ref{thm:williamson4-deg} (\cite[Theorem G]{CrePel-integrable}): gives a classification of $4$-by-$4$ matrices over $\Qp$ up to congruence via a symplectic matrix in the degenerate case, when $\Omega_0^{-1}M$ has some multiple eigenvalue.
	\item Theorem \ref{thm:num-forms1} (\cite[Theorem H]{CrePel-integrable}): gives the number of normal forms of $4$-by-$4$ matrices over $\Qp$ up to congruence via a symplectic matrix.
	\item Theorem \ref{thm:williamson-real} (\cite[Theorem K]{CrePel-integrable}): gives a classification of $2n$-by-$2n$ matrices over $\R$ up to congruence via a symplectic matrix in the case when $\Omega_0^{-1}M$ is diagonalizable.
	\item Theorem \ref{thm:williamson-real2} (\cite[Theorem L]{CrePel-integrable}): gives a classification of $2$-by-$2$ matrices over $\Qp$ up to congruence via a symplectic matrix in the general case (that is, the classical Weierstrass-Williamson classification).
	\item Theorem \ref{thm:num-forms} (\cite[Theorem I]{CrePel-integrable}): gives the asymptotic behavior of the number of normal forms of $2n$-by-$2n$ matrices over $\Qp$ up to congruence via a symplectic matrix.
	\item Theorem \ref{thm:num-forms-lower-bound} (\cite[Theorem J]{CrePel-integrable}): gives an explicit lower bound of the number of normal forms of $2n$-by-$2n$ matrices over $\Qp$ up to congruence via a symplectic matrix when $n\le 10$.
\end{itemize}

We recommend readers to consult \cite{CrePel-integrable} for further motivations and references, as well as for statements of the results about the classifications of local normal forms of integrable systems: the results we prove in the present paper were stepping stones needed to prove the results about integrable systems in \cite{CrePel-integrable}. Here we will prove the more algebraic results about matrix normal form classifications, which were essential to deduce in \cite{CrePel-integrable} the results about integrable systems.

The paper can be read independently from \cite{CrePel-integrable} as we recall all statements we prove. This paper would be the most interesting of the two for those who are interested in $p$-adic matrices and not in $p$-adic integrable systems. In addition, the present paper contains a number of intermediate results and examples which are of interest independently of \cite{CrePel-integrable}.

\subsection*{Structure of the paper}

Section \ref{sec:algclosed} continues the study of the classification problem in an algebraically closed field, started in \cite[Section 3]{CrePel-integrable}. Section \ref{sec:dim2} solves the classification problem over $\Qp$, for $2$-by-$2$ matrices, hence proving \cite[Theorems D and E]{CrePel-integrable}. For $4$-by-$4$ matrices there are first some general results in Section \ref{sec:dim4-gen}, then Section \ref{sec:dim4} proves \cite[Theorems F, G and H]{CrePel-integrable}. In Section \ref{sec:real} we use the strategy introduced in previous sections to give a new proof of the general case of the real Weierstrass-Williamson classification, that is, \cite[Theorems K and L]{CrePel-integrable}. The proof is given in Section \ref{sec:real-general}, then there is an example in Section \ref{sec:real-example} and some comments about the positive-definite case in Section \ref{sec:real-posdef}. Section \ref{sec:num-forms} discusses the $p$-adic classification in higher dimensions and proves \cite[Theorems I and J]{CrePel-integrable}. Section \ref{sec:JC} solves the classification of the $p$-adic Jaynes-Cummings system. Finally, Section \ref{sec:examples} gives some examples of the classification results in this paper.

\medskip
\textbf{Acknowledgments.}
We thank Pilar Bayer and Enrique Arrondo for discussions and suggestions which have improved the paper. We are really grateful to an anonymous referee for helpful comments which have improved the paper.

The second author thanks the Dean of the School of Mathematical Sciences Antonio Br\'u and the Chair of the Department of Algebra, Geometry and Topology at the Complutense University of Madrid, Rutwig Campoamor, for their support and excellent resources he is being provided with to carry out the FBBVA project. He also thanks Tobias Colding and the Massachusetts Institute of Technology MIT for the hospitality in October/November 2025 when he was at MIT as a Visiting Professor; the discussions with MIT faculty influenced the organization of some of the ideas of the present paper as well as opened possibilities for future research.

Part of this paper was written during a visit of the second author to the Universidad de Cantabria and the Universidad Internacional Men\'endez Pelayo in the summer of 2024 and he thanks both institutions for their hospitality. Another part of this paper was written while the first author was visiting the Complutense University of Madrid in September-October 2024 and he thanks this institution.

\medskip
\textbf{Funding.}
The first author is funded by grant PID2022-137283NB-C21 of MCIN/AEI/ 10.13039/501100011033 / FEDER, UE. The second author is funded by a FBBVA (Bank Bilbao Vizcaya Argentaria Foundation) Grant for Scientific Research Projects with title \textit{From Integrability to Randomness in Symplectic and Quantum Geometry}.

\section{Algebraic and arithmetical preliminaries: symplectic classification of matrices over algebraically closed fields}\label{sec:algclosed}

In this section we carry out the first step of our strategy to obtain the real and $p$-adic Weierstrass-Williamson classification, which consists on making such a classification over an algebraically closed field. In order to do this, we first find some linear bases of the symplectic space with good properties and then we use these bases to find a general form of the symplectic matrix we need. In later sections, we will apply this general form in order to deduce the matrix classification over the $p$-adic field (Sections \ref{sec:dim2}, \ref{sec:dim4-gen} and \ref{sec:dim4}) and over the real field (Section \ref{sec:real}). The results in this section are also of independent interest.

Let $F$ be a field with multiplicative identity element $1$. Let $\Omega_0\in\M_{2n}(F)$ be a block-diagonal matrix of size $2n$ with all blocks equal to
\[\begin{pmatrix}
	0 & 1 \\
	-1 & 0
\end{pmatrix},\]
i.e.
\[\Omega_0=\begin{pmatrix}
	0 & 1 &   &   &   &   &   \\
	-1& 0 &   &   &   &   &   \\
	&   & 0 & 1 &   &   &   \\
	&   &-1 & 0 &   &   &   \\
	&   &   &   & \ddots &   &   \\
	&   &   &   &   & 0 & 1 \\
	&   &   &   &   &-1 & 0
\end{pmatrix}\]

A matrix $S\in\M_{2n}(F)$ is \emph{symplectic} if $S^T\Omega_0S=\Omega_0$. Weierstrass \cite{Weierstrass} proved in 1858 that any positive definite symmetric matrix over $\R$ can be diagonalized by congruence via a symplectic matrix, that is, for any positive definite symmetric matrix $M\in\M_{2n}(\R)$ there is a symplectic matrix $S\in\M_{2n}(\R)$ such that $S^TMS$ is diagonal. Williamson \cite{Williamson} generalized this result in 1936 to all symmetric matrices, proving that, for any symmetric matrix $M\in\M_{2n}(\R)$, there is a symplectic matrix $S\in\M_{2n}(\R)$ such that $S^TMS$ is diagonal by blocks. If the eigenvalues of $\Omega_0^{-1}M$ are pairwise distinct, the blocks have the following forms:
\[\begin{pmatrix}
	r_i & 0 \\
	0 & r_i
\end{pmatrix},
\begin{pmatrix}
	0 & r_i \\
	r_i & 0
\end{pmatrix}
\text{ or }
\begin{pmatrix}
	0 & r_{i+1} & 0 & r_i \\
	r_{i+1} & 0 & -r_i & 0 \\
	0 & -r_i & 0 & r_{i+1} \\
	r_i & 0 & r_{i+1} & 0
\end{pmatrix},\]
for some $r_i\in\R,1\le i\le n$, which are called \textit{elliptic block}, \textit{hyperbolic block} and \textit{focus-focus block}. In Section \ref{sec:real} we recover Williamson's result following an essentially different strategy. In Williamson's proof, all matrix transformations take place in the base field. Our strategy starts by lifting the problem to an algebraically closed field, where the problem becomes easier to solve. Then we observe that, if the field $F$ is not algebraically closed (such as $\R$ or $\Qp$) and the size of the matrix is at most $4$, we only need to use an extension of degree $2$ or $4$. In the real case this means $\C$, and the problem is almost solved. In the $p$-adic case there are many more different cases.

The first step of our strategy, which implies working in an algebraically closed field, is started in \cite[Section 3]{CrePel-integrable}. There, we derive two partial results. One of them gives a basis with good properties if the eigenvalues of $\Omega_0^{-1}M$ are pairwise distinct.

\begin{lemma}[{\cite[Lemma 3.4]{CrePel-integrable}}]\label{lemma:eig}
	Let $n$ be a positive integer. Let $F$ be an algebraically closed field with characteristic different from $2$. Let $\Omega_0$ be the matrix of the standard symplectic form on $F^{2n}$ and let $M\in\M_{2n}(F)$ be a symmetric matrix such that the eigenvalues of $\Omega_0^{-1}M$ are pairwise distinct. Then there exists a basis $\{u_1,v_1,\ldots,u_n,v_n\}$ of $F^{2n}$ such that
	\begin{itemize}
		\item $u_i$ and $v_i$ are eigenvectors of $\Omega_0^{-1}M$ with eigenvalues of opposite sign, and
		\item the standard symplectic form is represented in this basis by a block-diagonal matrix with blocks of size two.
	\end{itemize}
\end{lemma}

Our other result is more general, and separates the part of a matrix with nonzero eigenvalues from the part with zero eigenvalues, giving a basis for the first.

\begin{lemma}[{\cite[Lemma 3.6]{CrePel-integrable}}]\label{lemma:eig2}
	Let $n$ be a positive integer. Let $F$ be an algebraically closed field with characteristic different from $2$ and let $M\in\M_{2n}(F)$ be a symmetric matrix. Let $\Omega_0$ be the matrix of the standard symplectic form on $F^{2n}$. Let $A=\Omega_0^{-1}M$. Then, the number of nonzero eigenvalues of $A$ is even, that is, $2m$ for some integer $m$ with $0\le m\le n$, and there exists a set \[\Big\{u_1,v_1,\ldots,u_m,v_m\Big\}\subset F^{2n}\] which satisfies the following properties:
	\begin{itemize}
		\item $Au_i=\lambda_iu_i+\mu_iu_{i-1}$ and $Av_i=-\lambda_iv_i+\mu_{i+1}v_{i+1}$ for $1\le i\le m$, where $\lambda_i\in F$, $\mu_i=0$ or $1$, $\mu_1=\mu_{m+1}=0$, and $\mu_i=1$ only if $\lambda_i=\lambda_{i-1}$. (That is to say, the vectors are a ``Jordan basis''.)
		\item The vectors can be completed to a symplectic basis: given $w_1,w_2$ in the set, $w_1^T\Omega_0w_2=1$ if $w_1=u_i,w_2=v_i$ for some $i$ with $1\le i\le m$, and otherwise $w_1^T\Omega_0w_2=0$.
	\end{itemize}
\end{lemma}

Choosing a symplectic basis for the zero eigenvalues to complete the partial basis in a way that we obtain a complete ``Jordan basis'' is more complicated, but it can also be done, as the following theorem shows.

\begin{definition}
	Let $t$ be a positive integer. A tuple $K=(k_1,k_2,\ldots,k_t)$ of natural numbers is called \textit{good} if the numbers are in non-increasing order, that is, $k_i\ge k_{i+1}$ for every $1\le i\le t-1$, and each odd number appears an even number of times in $K$. We define $f_K:\{1,\ldots,t\}\to\{1,\ldots,t\}$ as follows:
	\begin{itemize}
		\item If $k_i$ is even, $f_K(i)=i$.
		\item If $k_i$ is odd, let $i_0$ be the first index such that $k_{i_0}=k_i$. Then, $f_K(i)=i+(-1)^{i-i_0}$.
	\end{itemize}
	If $K$ is good, $f_K$ is an involution that fixes the indices of even elements and pairs those of odd elements.
\end{definition}

\begin{theorem}\label{thm:zeros}
	Let $n$ be a positive integer. Let $F$ be a field with characteristic different from $2$ and let $M\in\M_{2n}(F)$ be a symmetric matrix. Let $\Omega_0$ be the matrix of the standard symplectic form on $F^{2n}$. Let $A=\Omega_0^{-1}M$. Suppose that the eigenvalues of $A$ are all zero. Then, there is an integer $t\ge 1$, a good tuple $K=(k_1,\ldots,k_t)$ and a basis of $F^{2n}$ of the form
	\begin{equation}\label{eq:basis}
		\Big\{u_{11},u_{12},\ldots,u_{1k_1},\ldots,u_{t1},u_{t2},\ldots,u_{tk_t}\Big\}
	\end{equation}
	such that the following two conditions hold:
	\begin{itemize}
		\item For every $1\le i\le t,1\le j\le k_t$, $Au_{ij}=u_{i,j-1}$, where $u_{i0}=0$.
		\item Given $u_{ij},u_{i'j'}$ in this basis, $u_{ij}^T\Omega_0u_{i'j'}\ne 0$ if and only if $i'=f_K(i)$ and $j+j'=k_i+1$.
	\end{itemize}
\end{theorem}

We refer to Table \ref{table:change}, left, for an example of such a basis.

In order to prove this theorem, we need several intermediate results.

\begin{definition}
	Let $t$ be a positive integer. Given a tuple $K=(k_1,\ldots,k_t)$ such that $k_i\ge k_{i+1}$ for $1\le i\le t-1$, we call $R_K$ the set of tuples of the form $(\ell,i,j)$ where $1\le \ell\le t,\ell\le i\le t,1\le j\le k_i$.
	
	Given any subset $R\subset R_K$, we say that a basis is \emph{$R$-acceptable} if, for each $(\ell,i,j)\in R$, one of these alternatives holds:
	\begin{itemize}
		\item $j=1$, $i=f_K(\ell)$, $\ell=f_K(i)$, and $\prodvec{\ell k_\ell}{ij}\ne 0$;
		\item $j\ne 1$ or $i\ne f_K(\ell)$, and $\prodvec{\ell k_\ell}{ij}=0$.
	\end{itemize}
	(Note that, if there exists $(\ell,i,j)\in R$ with $j=1$, $i=f_K(\ell)$ but $\ell\ne f_K(i)$, no basis is $R$-acceptable.)
	
	Given $(\ell,i,j),(\ell',i',j')\in R_K$, we say that $(\ell,i,j)<(\ell',i',j')$ if
	\begin{itemize}
		\item $\ell<\ell'$,
		\item $\ell=\ell'$ and $j<j'$, or
		\item $\ell=\ell'$, $j=j'$ and $i<i'$.
	\end{itemize}
	Note that $<$ is a total order in $R_K$.
	
	For each $s$ integer with $0\le s\le |R_K|$, we call $R_{K,s}$ the set of the first $s$ elements of $R_K$ according to the order $<$.
\end{definition}

Let $J$ be the Jordan form of $A$. Since all eigenvalues of $A$ are zero, $J$ has all entries equal to zero except some in the first diagonal above the main diagonal, which are $1$. That is, $J$ is a block-diagonal matrix. Let $k_1,\ldots,k_t$ be the sizes of the blocks, with $k_i\ge k_{i+1}$, and $K=(k_1,\ldots,k_t)$. For this $K$, any Jordan basis satisfies the first condition.

\begin{lemma}\label{lemma:shifting}
	If $\{u_{ij}\}$ is a Jordan basis for $A$ in the form \eqref{eq:basis}, $u_{ij}^T\Omega_0u_{i'j'}=-u_{i,j-1}^T\Omega_0u_{i',j'+1}$.
	That is, the product $u_{ij}^T\Omega_0u_{i'j'}$ only depends on $j+j'$, except for the sign.
\end{lemma}

\begin{proof}
	Using that $M$ is symmetric and $\Omega_0$ is antisymmetric,
	\begin{align*}
		u_{ij}^T\Omega_0u_{i'j'} & =u_{ij}^T\Omega_0 Au_{i',j'+1} \\
		& =u_{ij}^TMu_{i',j'+1} \\
		& =-u_{ij}^TA^T\Omega_0u_{i',j'+1} \\
		& =-u_{i,j-1}^T\Omega_0u_{i',j'+1}.\qedhere
	\end{align*}
\end{proof}

\begin{corollary}\label{cor:zero}
	Let us assume the conditions of Lemma \ref{lemma:shifting}. If $j+j'\le\max\{k_i,k_{i'}\}$, $\prodvec{ij}{i'j'}=0$.
\end{corollary}

\begin{proof}
	Without loss of generality suppose that $k_i\le k_{i'}$. Applying the previous lemma $j$ times, we have
	\[\prodvec{ij}{i'j'}=(-1)^j\prodvec{i0}{i',j+j'}=0.\qedhere\]
\end{proof}

\begin{corollary}\label{cor:nonzero}
	If a basis $\{u_{ij}\}$ of the form \eqref{eq:basis} is $R$-acceptable for some $R$ which contains $(\ell,f_K(\ell),1)$ or $(f_K(\ell),\ell,1)$, then $\prodvec{\ell k_\ell}{f_K(\ell),1}\ne 0$.
\end{corollary}

\begin{proof}
	If $(\ell,f_K(\ell),1)\in R$, then the conclusion holds by definition.
	
	If $(f_K(\ell),\ell,1)\in R$, then $f_K(\ell)\le \ell$, which implies $k_\ell=k_{f_K(\ell)}$, and $\prodvec{\ell k_\ell}{f_K(\ell),1}=(-1)^{k_\ell-1}\prodvec{\ell1}{f_K(\ell),k_\ell}\ne 0$.
\end{proof}

We will need this operation to change a basis:

\begin{definition}
	Let $B$ be a basis of $F^{2n}$ of the form \eqref{eq:basis}, $1\le \ell\le t,1\le i\le t,1\le j\le k_i$ and $c\in F$. We call $B[\ell,i,j,c]$ the basis
	\[\Big\{u'_{11},u'_{12},\ldots,u'_{1k_1},\ldots,u'_{t1},u'_{t2},\ldots,u'_{tk_t}\Big\}\]
	where
	\[u'_{i'j'}=\begin{cases}
		u_{i'j'} & \text{if } i'\ne i, \\
		u_{ij'} & \text{if } i'=i,j'<j, \\
		u_{ij'}+cu_{\ell,j'-j+1} & \text{if } i'=i,j'\ge j.
	\end{cases}\]
\end{definition}

We refer to Table \ref{table:change} for an example.

\begin{table}
	\begin{tabular}{|c|c|c|c|c|c|}\cline{1-1}
		$u_{16}$ \\ \cline{1-3}
		$u_{15}$ & $u_{25}$ & $u_{35}$ \\ \cline{1-4}
		$u_{14}$ & $u_{24}$ & $u_{34}$ & $u_{44}$ \\ \hline
		$u_{13}$ & $u_{23}$ & $u_{33}$ & $u_{43}$ & $u_{53}$ & $u_{63}$ \\ \hline
		$u_{12}$ & $u_{22}$ & $u_{32}$ & $u_{42}$ & $u_{52}$ & $u_{62}$ \\ \hline
		$u_{11}$ & $u_{21}$ & $u_{31}$ & $u_{41}$ & $u_{51}$ & $u_{61}$ \\ \hline
	\end{tabular}
	$\quad\rightarrow\quad$
	\begin{tabular}{|c|c|c|c|c|c|}\cline{1-1}
		$u_{16}$ \\ \cline{1-3}
		$u_{15}$ & $u_{25}$ & $u_{35}$ \\ \cline{1-4}
		$u_{14}$ & $u_{24}$ & $u_{34}$ & $u_{44}+cu_{23}$ \\ \hline
		$u_{13}$ & $u_{23}$ & $u_{33}$ & $u_{43}+cu_{22}$ & $u_{53}$ & $u_{63}$ \\ \hline
		$u_{12}$ & $u_{22}$ & $u_{32}$ & $u_{42}+cu_{21}$ & $u_{52}$ & $u_{62}$ \\ \hline
		$u_{11}$ & $u_{21}$ & $u_{31}$ & $u_{41}$ & $u_{51}$ & $u_{61}$ \\ \hline
	\end{tabular}
	\caption{Left: a basis $B$ of $F^{26}$ in the form \eqref{eq:basis} for $K=(6,5,5,4,3,3)$. Right: the basis $B[2,4,2,c]$, for $c\in F$.}
	\label{table:change}
\end{table}

\begin{lemma}\label{lemma:modbasis}
	\begin{enumerate}
		\item If $B$ is a Jordan basis for $A$, $B[\ell,i,j,c]$ is also a Jordan basis.
		\item The only products of the form $\prodvec{\ell'k_{\ell'}}{i'j'}$ which may be different in $B$ and $B[\ell,i,j,c]$ are those with $i'=i\ne \ell'$ and $j'\ge j$, which vary in $c\prodvec{\ell'k_{\ell'}}{\ell,j'-j+1}$, those with $\ell'=i\ne i'$, which vary in $c\prodvec{\ell,k_i-j+1}{i'j'}$, and those with $\ell'=i'=i$, which vary in $c\prodvec{ik_i}{\ell,j'-j+1}+c\prodvec{\ell,k_i-j+1}{ij'}+c^2\prodvec{\ell,k_i-j+1}{\ell,j'-j+1}$ (understanding $u_{ij}=0$ if $j\le 0$ in the last equality).
	\end{enumerate}
\end{lemma}

\begin{proof}
	\begin{enumerate}
		\item We have that $Au_{i'j'}=u_{i',j'-1}$ and we need to see that $Au'_{i'j'}=u'_{i',j'-1}$ for all indices $1\le i'\le t,1\le j'\le k_{i'}$. If $i'\ne i$, the conclusion follows because the vectors do not change. If $i'=i$ and $j'<j$, the same happens. Otherwise, $i'=i$ and $j'\ge j$, and
		\[Au'_{ij'}=A(u_{ij'}+cu_{\ell,j'-j+1})=u_{i,j'-1}+cu_{\ell,j'-j}=u'_{i,j'-1}\]
		where the last equality also holds if $j'=j$ and $u_{\ell,j'-j}=0$.
		\item This follows from the definition of the new basis: to change the product we need to change any of the two vectors.\qedhere
	\end{enumerate}
\end{proof}

\begin{lemma}\label{lemma:main}
	For every $s$ with $0\le s\le |R_K|$, there is a Jordan basis for $A$ which is $R_{K,s}$-acceptable.
\end{lemma}

\begin{proof}
	We prove this by induction in $s$. For $s=0$, $R_{K,0}$ is empty and the problem reduces to the existence of a Jordan basis.
	
	Supposing it is true for $s$, we prove it for $s+1$. Let $\{(\ell,i,j)\}=R_{K,s+1}\setminus R_{K,s}$, that is, $(\ell,i,j)$ is the $(s+1)$-th element of $R_K$ according to the order $<$. Let $B=\{u_{ij}\}$ be the $R_{K,s}$-acceptable basis given by the inductive hypothesis.
	
	There are several cases to consider.
	\begin{enumerate}
		\item $i=\ell$ and $k_\ell-j$ is even. By Lemma \ref{lemma:shifting}, we have that \[\prodvec{\ell k_\ell}{\ell j}=\prodvec{\ell,(k_\ell+j)/2}{\ell,(k_\ell+j)/2}=0.\] Either $j>1$ and we want a zero, or $j=1$ with $k_\ell$ odd and we also want a zero; in any case $B$ itself is $R_{K,s+1}$-acceptable.
		
		\item $i=f_K(\ell),j=1$. This means that $i$ is $\ell$ if $k_\ell$ is even and $\ell+1$ if $k_\ell$ is odd (it cannot be $\ell-1$ because $(\ell,i,j)\in R_K$). If $\prodvec{\ell k_\ell}{i1}\ne 0$, $B$ is $R_{K,s+1}$-acceptable and we are done. Otherwise, since $B$ is a basis, there are $1\le i_1\le t,1\le j_1\le k_{i_1}$ such that $\prodvec{\ell1}{i_1j_1}\ne 0.$ By Corollary \ref{cor:zero}, $j_1+1>\max\{k_\ell,k_{i_1}\}$ or equivalently $j_1\ge\max\{k_\ell,k_{i_1}\}$. But this implies $k_{i_1}\ge j_1\ge\max\{k_\ell,k_{i_1}\}\ge k_\ell$, so $j_1\ge k_{i_1}$, which implies $j_1=k_{i_1}$ and $\prodvec{\ell1}{i_1k_{i_1}}\ne 0.$
		
		By Lemma \ref{lemma:shifting}, $\prodvec{\ell1}{ik_i}=(-1)^{k_i-1}\prodvec{\ell k_\ell}{i1}=0$, so $i\ne i_1$. If $i_1<i$, either $i_1<\ell$ or $i_1=\ell$ with $i=\ell+1$. In any case $(i_1,\ell,1)<(\ell,i,1)$, and since $B$ is $R_{K,s}$-acceptable, $\prodvec{\ell1}{i_1k_{i_1}}=0$, a contradiction. Hence, $i_1>i\ge \ell$. As the sequence $K$ is non-increasing and $k_{i_1}\ge k_\ell$, we must have $k_{i_1}=k_i=k_\ell$, and again by Lemma \ref{lemma:shifting}, $\prodvec{\ell k_\ell}{i_11}\ne 0$. If $i=\ell+1$, $k_i=k_\ell$ together with $f_K(\ell)=i$ implies that $f_K(i)=\ell$.
		
		Let $B'=\{u'_{ij}\}=B[i_1,i,1,1]$. This basis is $R_{K,s}$-acceptable: by Lemma \ref{lemma:modbasis}, the conditions in $R_{K,s}$ that may break with this change are $(\ell',i,j')$ with $\ell'<\ell$ (the only condition which may be in $R_{K,s}$ with $\ell'=\ell$ is $(\ell,\ell,1)$). For these tuples,
		\[\prodvecprime{\ell'k_{\ell'}}{ij'}=\prodvec{\ell'k_{\ell'}}{ij'}+\prodvec{\ell'k_{\ell'}}{i_1j'}=0\]
		because $(\ell',i,j')$ and $(\ell',i_1,j')$ are in $R_{K,s}$. Moreover, if $i=\ell+1$,
		\[\prodvecprime{\ell k_\ell}{i1}=\prodvec{\ell k_\ell}{i1}+\prodvec{\ell k_\ell}{i_11}=\prodvec{\ell k_\ell}{i_11}\ne 0,\]
		so $B'$ is $R_{K,s+1}$-acceptable, as we wanted.
		
		If $i=\ell$, we also define $B''=\{u''_{ij}\}=B[i_1,i,1,-1]$. Analogously to what we said for $B'$, $B''$ is $R_{K,s}$-acceptable. Now we have
		\begin{align*}
			\prodvecprime{\ell k_\ell}{i1} & =\prodvec{\ell k_\ell}{i1}+\prodvec{\ell k_\ell}{i_11}+\prodvec{i_1k_\ell}{i1}+\prodvec{i_1k_\ell}{i_11} \\
			& =2\prodvec{\ell k_\ell}{i_11}+\prodvec{i_1k_\ell}{i_11}
		\end{align*}
		and
		\begin{align*}
			\prodvecsecond{\ell k_\ell}{i1} & =\prodvec{\ell k_\ell}{i1}-\prodvec{\ell k_\ell}{i_11}-\prodvec{i_1k_\ell}{i1}+\prodvec{i_1k_\ell}{i_11} \\
			& =-2\prodvec{\ell k_\ell}{i_11}+\prodvec{i_1k_\ell}{i_11}
		\end{align*}
		where the second equality in each line is due to Lemma \ref{lemma:shifting}. If both results were zero at the same time, that would imply $\prodvec{\ell k_\ell}{i_11}=0$, so one of them must be nonzero, and one of $B'$ and $B''$ must be $R_{K,s+1}$-acceptable.
		
		\item $i=\ell,j>1$ and $k_\ell-j$ is odd. We set
		\[c=-\frac{\prodvec{\ell k_\ell}{\ell j}}{2\prodvec{\ell k_\ell}{f_K(\ell),1}}\]
		and $B'=B[f_K(\ell),\ell,j,c]$ (the denominator is not zero by Corollary \ref{cor:nonzero}, because $(\ell,f_K(\ell),1)\in R_{K,s}$ or $(f_K(\ell),\ell,1)\in R_{K,s}$). By Lemma \ref{lemma:modbasis}, the conditions that can break are $(\ell,i',j')$ with $\ell\le i'\le t$ and $1\le j'\le j-1$, and $(\ell',\ell,j')$ with $1\le \ell'\le \ell-1$ and $j\le j'\le k_i$. For the first type, we see that
		\[\prodvecprime{\ell k_\ell}{i'j'}=\prodvec{\ell k_\ell}{i'j'}+c\prodvec{f_K(\ell),k_\ell-j+1}{i'j'}\]
		The second term is $0$ because $k_\ell-j+1+j'\le k_\ell=k_{f_K(\ell)}$ (we are using here that $f_K(f_K(\ell))=\ell$ because the basis is acceptable). For the second type,
		\[\prodvecprime{\ell'k_{\ell'}}{\ell j'}=\prodvec{\ell'k_{\ell'}}{\ell j'}+c\prodvec{\ell'k_{\ell'}}{f_K(\ell),j'-j+1}\]
		Both terms are zero because $(\ell',\ell,j')$ and $(\ell',f_K(\ell),j'-j+1)$ are in $R_{K,s}$, $f_K(\ell')\ne f_K(\ell)$ and $j'>1$. It is only left to show that the new condition is satisfied:
		\begin{align*}
			\prodvecprime{\ell k_\ell}{\ell j} & =\prodvec{\ell k_\ell}{\ell j}+c\prodvec{\ell k_\ell}{f_K(\ell),1} \\
			& \ \ \ +c\prodvec{f_K(\ell),k_\ell-j+1}{\ell j}+c^2\prodvec{f_K(\ell),k_\ell-j+1}{f_K(\ell),1}
		\end{align*}
		The last term is zero because $k_\ell-j+1+1\le k_\ell$. The second and the third are equal because, by Lemma \ref{lemma:shifting},
		\begin{align*}
			\prodvec{f_K(\ell),k_\ell-j+1}{\ell j} & =-\prodvec{\ell j}{f_K(\ell),k_\ell-j+1} \\
			& =(-1)^{k_\ell-j+1}\prodvec{\ell k_\ell}{f_K(\ell),1} \\
			& =\prodvec{\ell k_\ell}{f_K(\ell),1}.
		\end{align*}
		Hence
		\[\prodvecprime{\ell k_\ell}{\ell j}=\prodvec{\ell k_\ell}{\ell j}+2c\prodvec{\ell k_\ell}{f_K(\ell),1}=0\]
		as we wanted.
		
		\item $i>\ell$ and either $j>1$ or $i\ne f_K(\ell)$. We set
		\[c=-\frac{\prodvec{\ell k_\ell}{ij}}{\prodvec{\ell k_\ell}{f_K(\ell),1}}\]
		and $B'=B[f_k(\ell),i,j,c]$ (the denominator is not zero by Corollary \ref{cor:nonzero}, because $(\ell,f_K(\ell),1)\in R_{K,s}$ or $(f_K(\ell),\ell,1)\in R_{K,s}$). By Lemma \ref{lemma:modbasis}, the conditions that can break are $(\ell',i,j')$ with $1\le \ell'\le \ell-1$ and $j\le j'\le k_i$. In these cases we have
		\[\prodvecprime{\ell'k_{\ell'}}{ij'}=\prodvec{\ell'k_{\ell'}}{ij'}+c\prodvec{\ell'k_{\ell'}}{f_K(\ell),j'-j+1}\]
		Both terms are zero because $(\ell',i,j')$ and $(\ell',f_K(\ell),j'-j+1)$ are in $R_{K,s}$, $f_K(\ell')\ne f_K(\ell)$, and either $j'>1$ or $i\ne f_K(\ell)$. It is left to show the new condition:
		\[\prodvecprime{\ell k_\ell}{ij}=\prodvec{\ell k_\ell}{ij}+c\prodvec{\ell k_\ell}{f_K(\ell),1}=0.\qedhere\]
	\end{enumerate}
\end{proof}

With these results, we are ready to prove Theorem \ref{thm:zeros}.

\begin{proof}[Proof of Theorem \ref{thm:zeros}]
	We apply Lemma \ref{lemma:main} with $s=|R_K|$ to obtain a basis $B$ which is $R_K$-acceptable. Now we show that this $B$ is the basis we want. The first condition holds because $B$ is a Jordan basis. For the second, if $j+j'\le\max\{k_i,k_{i'}\}$, the product is zero by Corollary \ref{cor:zero}, as we want. Otherwise, suppose $i\le i'$ and $k_i\ge k_{i'}$. By Lemma \ref{lemma:shifting},
	\[\prodvec{ij}{i'j'}=\prodvec{ik_i}{i',j+j'-k_i}=0,\]
	and the second condition also holds.
	
	It is left to show that $K$ is good, but this is an immediate consequence of $B$ being $R_K$-acceptable: by definition, if $f_K(\ell)=i$ but $f_K(i)\ne \ell$ for some $(\ell,i,j)\in R_K$, there would not be any $R_K$-acceptable basis.
\end{proof}

Our next result reduces congruence via a symplectic matrix to an equality of Jordan forms and finds explicitly the form of the symplectic matrix if it exists.

\begin{proposition}\label{prop:symplectomorphic}
	Let $n$ be a positive integer. Let $F$ be an algebraically closed field. Let $\Omega_0$ be the matrix of the standard symplectic form on $F^{2n}$. Let $M_1,M_2\in\M_{2n}(F)$ be symmetric matrices. We set \[A_i=\Omega_0^{-1}M_i,\] and $J_i$ the Jordan form of $A_i$. Let $\Psi_i$ be such that $\Psi_i^{-1}A_i\Psi_i=J_i$, for $i\in\{1,2\}$. Then there exists a symplectic matrix $S$ such that $S^TM_1S=M_2$ if and only if $J_1=J_2$. Moreover, in that case $S$ must have the form $\Psi_1 D\Psi_2^{-1}$, where $D$ is a matrix that commutes with $J_1=J_2$ and satisfies $D^T\Psi_1^T\Omega_0\Psi_1D=\Psi_2^T\Omega_0\Psi_2$.
\end{proposition}

\begin{proof}
	Let $A_i=\Omega_0^{-1}M_i$. Suppose first that such a $S$ exists. Then,
	\[S^{-1}A_1S=S^{-1}\Omega_0^{-1}M_1S=S^{-1}\Omega_0^{-1}(S^T)^{-1}S^TM_1S=\Omega_0^{-1}M_2=A_2\]
	hence $A_1$ and $A_2$ are equivalent, and $J_1=J_2$.
	
	Let $D=\Psi_1^{-1}S\Psi_2$. We have that $S^T\Omega_0S=\Omega_0$, which implies \[D^T\Psi_1^T\Omega_0\Psi_1D=\Psi_2^T\Omega_0\Psi_2.\] Also,
	\[J_1D=\Psi_1^{-1}A_1\Psi_1D=\Psi_1^{-1}A_1S\Psi_2=\Psi^{-1}SA_2\Psi_2=D\Psi_2^{-1}A_2\Psi_2=DJ_2=DJ_1.\]
	
	Now suppose that $J_1=J_2$, let $D$ be a matrix which satisfies the conditions and let $S=\Psi_1 D\Psi_2^{-1}$. The condition $D^T\Psi_1^T\Omega_0\Psi_1D=\Psi_2^T\Omega_0\Psi_2$ implies that $S^T\Omega_0S=\Omega_0$, that is, $S$ is symplectic. Moreover,
	\[S^{-1}A_1S=\Psi_2D^{-1}\Psi_1^{-1}A_1\Psi_1D\Psi_2^{-1}=\Psi_2D^{-1}J_1D\Psi_2^{-1}=\Psi_2J_1\Psi_2^{-1}=A_2\]
	which implies
	\[S^TM_1S=S^T\Omega_0 A_1S=S^T\Omega_0 SS^{-1}A_1S=\Omega_0 A_2=M_2\]
	as we wanted.
\end{proof}

Proposition \ref{prop:symplectomorphic} can be simplified in the case where the eigenvalues are pairwise distinct:

\begin{corollary}\label{cor:symplectomorphic}
	Let $n$ be a positive integer. Let $F$ be an algebraically closed field with characteristic different from $2$. Let $\Omega_0$ be the matrix of the standard symplectic form on $F^{2n}$. Let $M_1,M_2\in\M_{2n}(F)$ be symmetric matrices such that $\Omega_0^{-1}M_j$ has pairwise distinct eigenvalues for $j\in\{1,2\}$. For $j\in\{1,2\}$ and $i\in\{1,\ldots,n\}$, let $u_i^j$ and $v_i^j$ be those of Lemma \ref{lemma:eig} when applied to $M_j$, and let
	\[\Psi_j=\begin{pmatrix}
		u_1^j & v_1^j & \ldots & u_n^j & v_n^j
	\end{pmatrix}.\]
	Then there is a symplectic matrix $S$ such that $S^TM_1S=M_2$ if and only if $\Omega_0^{-1}M_1$ and $\Omega_0^{-1}M_2$ have the same eigenvalues. Moreover, in this case $S$ must have the form $\Psi_1 D\Psi_2^{-1}$, where $D$ is a diagonal matrix with $(d_1,\ldots,d_{2n})$ in the diagonal such that
	\begin{equation}\label{eq:relation}
		d_{2i-1}d_{2i}=\frac{(u_i^2)^T\Omega_0 v_i^2}{(u_i^1)^T\Omega_0 v_i^1}.
	\end{equation}
\end{corollary}

\begin{proof}
	Let $J_j=\Psi_j^{-1}A_j\Psi_j$ for $j\in\{1,2\}$. By Lemma \ref{lemma:eig}, $J_j$ is a diagonal matrix, and $\Psi_j^T\Omega_0\Psi_j$ has all elements zero except those of the form $(2i-1,2i)$ and $(2i,2i-1)$. Moreover, in this case matrices $J_1$ and $J_2$ have all elements in the diagonal different. They are equal if and only if $A_1$ and $A_2$ have the same eigenvalues.
	
	Now we apply Proposition \ref{prop:symplectomorphic}. A matrix $D$ that commutes with $J_1=J_2$ is necessarily diagonal, and the relation \[D^T\Psi_1^T\Omega_0\Psi_1D=\Psi_2^T\Omega_0\Psi_2\] in this case reduces to \eqref{eq:relation}.
\end{proof}

We can use Lemma \ref{lemma:eig2} and Theorem \ref{thm:zeros} to give a classification in the case where the field is algebraically closed. In particular, we can apply this to $\C$.

\begin{theorem}\label{thm:algclosed}
	Let $n$ be a positive integer. Let $F$ be an algebraically closed field with characteristic different from $2$ and let $M\in \M_{2n}(F)$ be a symmetric matrix. There exists a positive integer $k$, $r\in F$, $a\in\{0,1\}$, with $a=1$ only if $r=0$, and a symplectic matrix $S\in\M_{2n}(F)$ such that $S^TMS$ is a block-diagonal matrix whose blocks are of the form
	\[M_\h(k,r,a)=\begin{pmatrix}
		0 & r &   &   &   & & \\
		r & 0 & 1 &   &   & & \\
		& 1 & 0 & r &   & & \\
		&   & r & 0 & \ddots & & \\
		&   &   & \ddots & 0 & 1 & \\
		&   &   &   & 1 & 0 & r \\
		&   &   &   &   & r & a
	\end{pmatrix}\]
	with $2k$ rows, that is, the main diagonal is $0$ except for an $a$ at the last position, and the two diagonals above and below the main alternate $r$ and $1$. Furthermore, if there are two matrices $S$ and $S'$ which reduce $M$ to this form, then the two forms only differ in the order in which the blocks are arranged.
\end{theorem}

\begin{euproof}
	\item First we prove existence. We start applying Lemma \ref{lemma:eig2}. This gives us a partial symplectic basis \[\Big\{u_1,v_1,\ldots,u_m,v_m\Big\},\] which is also a partial Jordan basis of $A$ corresponding to the nonzero eigenvalues, and $u_i$ and $v_i$ correspond to eigenvalues with opposite sign $\lambda_i$ and $-\lambda_i$. If $\{u_1,v_1,\ldots,u_k,v_k\}$ are the vectors of a block with eigenvalues $r$ and $-r$, for $r\in F$, these same vectors taken as columns of $S$ give the matrix $M_\h(k,r,0)$.
	
	For the other part of the Jordan form, we apply Theorem \ref{thm:zeros} to the eigenspace of $0$, resulting in a good multiset $K=\{k_1,\ldots,k_t\}$ and a basis \[\Big\{u_{11},u_{12},\ldots,u_{1k_1},\ldots,u_{t1},u_{t2},\ldots,u_{tk_t}\Big\}.\] This is not necessarily a symplectic basis as such, but it allows us to construct one with the required properties:
	\begin{itemize}
		\item If $k_i=2\ell_i$ is even, we have that \[u_{ij}^T\Omega_0u_{i,k_i+1-j}=(-1)^{j+1}c_i\] for some $c_i\in F$. Let \[c_i=b_i^2\] for $b_i\in F$. After dividing all the chain by $b_i$, we can assume that $c_i=(-1)^{\ell_i}$. Now \[\Big\{u_{i1},(-1)^{\ell_i}u_{i,2\ell_i},-u_{i2},(-1)^{\ell_i}u_{i,2\ell_i-1},\ldots,(-1)^{\ell_i-1}u_{i\ell_i},(-1)^{\ell_i}u_{i,\ell_i+1}\Big\}\] is a partial symplectic basis which gives the form $M_\h(\ell_i,0,1)$.
		\item If $k_i$ is odd and $f_K(i)=i+1$, we have instead \[u_{ij}^T\Omega_0u_{i+1,k_i+1-j}=(-1)^{j+1}c_i\] for some $c_i\in F$. After dividing the elements of the second chain by $c_i$, we can assume that $c_i=1$. Now \[\Big\{u_{i1},u_{i+1,k_i},-u_{i2},u_{i+1,k_i-1},\ldots,(-1)^{k_i-1}u_{ik_i},u_{i+1,1}\Big\}\] is a partial symplectic basis which gives the form $M_\h(k_i,0,0)$.
	\end{itemize}
	\item Uniqueness follows from Proposition \ref{prop:symplectomorphic}, because two matrices in normal form have the same Jordan form if and only if they differ in the order of the blocks.
\end{euproof}

In the cases of greatest interest to us the matrices in the statement of Theorem \ref{thm:algclosed} have coefficients in $\Qp$. We can take $F=\Cp$, but the resulting matrix $S$ will have the entries in $\Cp$ and not necessarily in $\Qp$, and we want a symplectomorphism of $(\Qp)^n$, which is given by a symplectic matrix with entries in $\Qp$. To avoid this, we need to manipulate adequately the symplectic basis, which translates to adjusting the matrix $D$ in Proposition \ref{prop:symplectomorphic}. This problem also happens in the real case, but, as we will see, the matrix $S$ can always be adapted to have the entries in $\R$ instead of $\C$.

\section{$p$-adic matrix classification in dimension $2$}\label{sec:dim2}

In this section we prove \cite[Theorem D]{CrePel-integrable} and \cite[Theorem E]{CrePel-integrable}, stated here as Theorems \ref{thm:williamson} and \ref{thm:num-forms2}, that is, the $p$-adic version of the Weierstrass-Williamson matrix classification in dimension $2$. The strategy consists of using the theorems in Section \ref{sec:algclosed} to solve the problem in the extension of $\Qp$ which contains the eigenvalues of $\Omega_0^{-1}M$. This may be $\Qp$ itself, and we are done, or an extension field of degree $2$, in which case we will have a general form for the matrix $S$ we need and we must give particular values for its parameters in such a way that its entries are in $\Qp$.

In dimension $2$, the eigenvalues of $\Omega_0^{-1}M$ are the roots of a degree $2$ polynomial with coefficients in $F$, so they have the form $\lambda$ and $-\lambda$. The case where the eigenvalues are in the base field is the easiest one.
\begin{proposition}\label{prop:hyperbolic}
	Let $F$ be a field of characteristic different from $2$. Let $\Omega_0$ be the matrix of the standard symplectic form on $F^2$. Let $M\in\M_2(F)$ symmetric and invertible such that the eigenvalues of $\Omega_0^{-1}M$ are in $F$. Then there is a symplectic matrix $S\in\M_2(F)$ and $a\in F$ such that
	\[S^TMS=\begin{pmatrix}
		0 & a \\
		a & 0
	\end{pmatrix}.\]
\end{proposition}

\begin{proof}
	If the eigenvalues $\pm\lambda$ of $\Omega_0^{-1}M$ are in $F$, its eigenvectors are also in $F$, and the matrix $S$ whose columns are the eigenvectors is in $F$. We can rescale one of the eigenvectors so that the determinant of $S$ is $1$, which for dimension $2$ is equivalent to be symplectic. If we take $a=\lambda$,
	\[S^TMS=(S^T\Omega_0S)(S^{-1}\Omega_0^{-1}MS)=\Omega_0\begin{pmatrix}
		-a & 0 \\
		0 & a
	\end{pmatrix}
	=\begin{pmatrix}
		0 & a \\
		a & 0
	\end{pmatrix},\]
	as we wanted.
\end{proof}

If the matrix is not invertible, the case of the null matrix is already covered by the previous result, with $a=0$. The other case is solved in the next proposition.

\begin{proposition}\label{prop:hyp-degenerate}
	Let $F$ be a field of characteristic different from $2$. Let
	\[M=\begin{pmatrix}
		a & b \\
		b & c
	\end{pmatrix}\in\M_2(F)\]
	where $b^2=ac$. Let $d\in F$. There exists a symplectic matrix $S$ such that
	\[S^TMS=\begin{pmatrix}
		d & 0 \\
		0 & 0
	\end{pmatrix}\]
	if and only if $ad$ is a square.
\end{proposition}

\begin{proof}
	The second column of $S$ must be $(kb,-ka)$ for some $k\in F$. Let the first column be $(x,y)$, for $x,y\in F$. Then
	\[d=\begin{pmatrix}
		x & y
	\end{pmatrix}
	\begin{pmatrix}
		a & b \\
		b & c
	\end{pmatrix}
	\begin{pmatrix}
		x \\ y
	\end{pmatrix}=ax^2+cy^2+2bxy\]
	and
	\[ad=a^2x^2+acy^2+2abxy=a^2x^2+b^2y^2+2abxy\]
	is a square. Conversely, if $ad$ is a square, we can take
	\[S=\begin{pmatrix}
		-kd & kb \\
		0 & -ka
	\end{pmatrix}\]
	where $k$ is chosen so that $k^2ad=1$.
\end{proof}

For the cases where $\lambda\notin F$, we need some definitions.
\begin{definition}
	Given an Abelian group $G$, we call $\Sq(G)$ the subgroup formed by the squares in $G$.
\end{definition} 
\begin{definition}
	Given a field $F$ with additive identity $0$ and $c\in F$, we call
	\[\DSq(F,c)=\Big\{x^2+cy^2: x,y\in F\Big\}\setminus\{0\}\]
	and
	\[\bDSq(F,c)=\DSq(F,c)/\Sq(F^*).\]
\end{definition} 

\begin{lemma}\label{lemma:DSq-group}
	Let $F$ be a field and $c\in F$. Then $\DSq(F,c)$ is a group with respect to multiplication in $F$.
\end{lemma}

\begin{proof}
	We just need to see that the product of two elements of $\DSq(F,c)$ is in $\DSq(F,c)$ and the inverse of an element of $\DSq(F,c)$ is in $\DSq(F,c)$:
	\[(x_1^2+cy_1^2)(x_2^2+cy_2^2)=(x_1x_2+c^2y_1y_2)^2+c(x_1y_2-x_2y_1)^2\]
	and
	\[\frac{1}{x^2+cy^2}=\left(\frac{x}{x^2+cy^2}\right)^2+c\left(\frac{y}{x^2+cy^2}\right)^2.\qedhere\]
\end{proof}

Lemma \ref{lemma:DSq-group} can also be proved using number theory: the map from $F[\sqrt{-c}]$ to $F$ sending $x+y\sqrt{-c}$ to $x^2+cy^2$ is what is known as the \emph{norm}. It is a known result that the norm is a multiplicative group morphism, hence its image, which is $\DSq(F,c)$, must be a multiplicative group.

The group $\DSq(F,c)$ can also be defined in terms of the Hilbert symbol:
\[(a,b)_p=\begin{cases}
	1 & \text{if } ax^2+by^2=1 \text{ for some } x,y\in\Qp; \\
	-1 & \text{otherwise.}
\end{cases}\]
With this definition, we have that
\[\DSq(F,c)=\Big\{a\in\Qp:(a,-c)_p=1\Big\}.\]
Lemma \ref{lemma:DSq-group} is a consequence of the multiplicativity of the Hilbert symbol.

We have that
\begin{align*}
	F^*/\DSq(F,c) & \cong(F^*/\Sq(F^*))/(\DSq(F,c)/\Sq(F^*)) \\
	& =(F^*/\Sq(F^*))/\bDSq(F,c),
\end{align*}
that is, we have a group isomorphism in the rightmost part of the commutative diagram
\[\begin{array}{ccccc}
	\DSq(F,c) & \hookrightarrow & F^* & \twoheadrightarrow & F^*/\DSq(F,c) \\
	\downarrow & & \downarrow & & \downarrow\cong \\
	\bDSq(F,c) & \hookrightarrow & F^*/\Sq(F^*) & \twoheadrightarrow & (F^*/\Sq(F^*))/\bDSq(F,c).
\end{array}\]
By definition, $\bDSq(F,c)$ is the subset of the classes in $F^*/\Sq(F^*)$ which contain elements of the form $x^2+c$, for $x\in F$. Note also that $\DSq(F,c)$, and hence $\bDSq(F,c)$, only depend on the class of $c$ modulo $\Sq(F^*)$.

Now we can give a necessary and sufficient condition for a matrix to be symplectomorphic to a normal form.
\begin{proposition}\label{prop:elliptic}
	Let $F$ be a field of characteristic different from $2$. Let $\Omega_0$ be the matrix of the standard symplectic form on $F^2$. Let $M\in\M_2(F)$ be a symmetric matrix such that the eigenvalues of $\Omega_0^{-1}M$ are of the form $\pm\lambda$ with $\lambda\notin F$ but $\lambda^2\in F$. Let $u$ be the eigenvector of value $\lambda$ in $\Omega_0^{-1}M$. Then for any $a,b\in F$ there is a symplectic matrix $S$ such that
	\[S^TMS=\begin{pmatrix}
		a & 0 \\
		0 & b
	\end{pmatrix}\]
	if and only if $\lambda^2=-ab$ and
	\[\frac{2\lambda a}{u^T\Omega_0\bar{u}}\in\DSq(F,-\lambda^2).\]
\end{proposition}

\begin{proof}
	We apply Corollary \ref{cor:symplectomorphic}. The eigenvalues of
	\[\Omega_0^{-1}\begin{pmatrix}
		a & 0 \\
		0 & b
	\end{pmatrix}
	=\begin{pmatrix}
		0 & -b \\
		a & 0
	\end{pmatrix}\]
	are $\pm\sqrt{-ab}$, so we must have $\lambda^2=-ab$.
	
	The matrix $\Psi_2$ has the form
	\[\begin{pmatrix}
		\lambda & -\lambda \\
		a & a
	\end{pmatrix}.\]
	The formula for $S$ gives that
	\[S=\Psi_1\begin{pmatrix}
		d_1 & 0 \\
		0 & d_2
	\end{pmatrix}
	\renewcommand{\arraystretch}{1.2}
	\begin{pmatrix}
		\frac{1}{2\lambda} & \frac{1}{2a} \\
		-\frac{1}{2\lambda} & \frac{1}{2a}
	\end{pmatrix}.\]
	The two columns of $\Psi_1$ are the eigenvectors of $\lambda$ and $-\lambda$. The first is $u$, and the second is the conjugate $\bar{u}$ (or, more precisely, can be taken as the conjugate).
	
	Now we have that $S$ must have entries in $F$. Let $c_1$ and $c_2$ be its columns. We get
	\[c_1=\frac{d_1u-d_2\bar{u}}{2\lambda}\in F^2;\]
	\[c_2=\frac{d_1u+d_2\bar{u}}{2a}\in F^2;\]
	\[d_1u=ac_2+\lambda c_1,d_2\bar{u}=ac_2-\lambda c_1=\overline{d_1u}\Rightarrow d_2=\bar{d}_1.\]
	
	The numbers $d_1$ and $d_2$ must also satisfy the equation (3.1) in Corollary \ref{cor:symplectomorphic}, that is
	\begin{equation}\label{eq:prod-elliptic}
		d_1d_2=\frac{(\lambda,a)\Omega_0(-\lambda,a)^T}{u^T\Omega_0\bar{u}}=\frac{2\lambda a}{u^T\Omega_0\bar{u}}.
	\end{equation}
	Taking into account that $d_2=\bar{d}_1$, we can substitute $d_1=r+s\lambda$ and $d_2=r-s\lambda$ in \ref{eq:prod-elliptic}, giving
	\[r^2-s^2\lambda^2=\frac{2\lambda a}{u^T\Omega_0\bar{u}}\]
	so this is in $\DSq(F,-\lambda^2)$.
\end{proof}

We can apply Proposition \ref{prop:elliptic} to the real elliptic case, that is, the case of a matrix in $\M_2(\R)$ whose eigenvalues are $\pm\lambda=\pm r\ii$. There, we want to achieve $a=b$. $\lambda$ is purely imaginary, so $\lambda^2=-ab=-a^2$ determines exactly $|a|=r$. $\DSq(\R,-\lambda^2)$ consists of all positive reals, so there is always a solution for $a$ (we know $|a|$ and can take the adequate sign to make $2\lambda a/u^T\Omega_0\bar{u}$ positive). This is why the real Weierstrass-Williamson classification in dimension $2$ has just two cases.

In the $p$-adic case, the two results until now are still enough to achieve a complete Weierstrass-Williamson classification in dimension $2$, but it is more complicated than the real case. First we need to compute $\DSq(\Qp,c)$ for all values of $c$. Of course, ``all values'' means ``all classes modulo $\Sq(\Qp^*)$'': this quotient is \[\Big\{1,c_0,p,c_0p\Big\},\] where $c_0$ is a quadratic non-residue modulo $p$, except if $p=2$, in which case the quotient is \[\Big\{1,-1,2,-2,3,-3,6,-6\Big\}.\]

We use the notation $\digit_i(x)$ for the digit in the $p$-adic expansion of $x$ which is $i$ positions to the left of the leading digit, that is, the digit of order $\ord(x)+i$. The value of $\DSq(\Qp,c)$ can also be deduced from known facts about the Hilbert symbol; however, this does not seem simpler than a direct proof.

\begin{proposition}\label{prop:images-extra}
	Let $p$ be a prime number such that $p\ne 2$ and $c\in\Qp^*$. Then $\DSq(\Qp,c)$ is given as follows (see Figure \ref{fig:images-extra}):
	\begin{enumerate}
		\item If $p\equiv 3\mod 4$, then $\DSq(\Qp,1)=\{u\in\Qp:\ord_p(u)\equiv 0\mod 2\}$ and $\DSq(\Qp,1)=\Qp$ otherwise;
		\item if $p\equiv 1\mod 4$, then $\DSq(\Qp,c_0)=\{u\in\Qp:\ord_p(u)\equiv 0\mod 2\}$ and $\DSq(\Qp,c_0)=\Qp$ otherwise;
		\item for any value of the prime $p$, $\DSq(\Qp,p)=\{u\in\Qp:\digit_0(u)\in\Sq(\F_p^*)\};$
		\item for any value of the prime $p$,
		$\DSq(\Qp,c_0p)=\{u\in\Qp:\ord_p(u)\equiv 0\mod 2,\digit_0(u)\in\Sq(\F_p^*)\}\cup \{u\in\Qp:\ord_p(u)\not\equiv 0\mod 2,\digit_0(u)\notin\Sq(\F_p^*)\}.$
	\end{enumerate}
\end{proposition}

\begin{proof}
	For the first point, we need to look at the possibilities modulo $\Sq(\Qp^*)$ of numbers of the form $x^2+1$. We can immediately get $1$ and $c_0$, the first by taking $x$ with high order and the second by taking it with order $0$ and an adequate leading digit. We can only get $p$ and $c_0p$ if $x^2$ can have $-1$ as a leading digit, which only happens if $p\equiv 1\mod 4$.
	
	For the second point, we need, analogously, to look at numbers of the form $x^2+c_0$. We immediately get the classes $1$ and $c_0$. We can only get $p$ and $c_0p$ if $-c_0$ is a square modulo $p$, that is, if $p\equiv 3\mod 4$.
	
	For the last two points, numbers of the form $x^2+p$ are in the classes $1$ and $p$, depending only in the order of $x$, and those of the form $x^2+c_0p$ are in $1$ and $c_0p$.
\end{proof}

\begin{figure}
	\setlength{\tabcolsep}{3mm}
	\begin{tabular}{ccccc}
		\raisebox{7mm}{$p\equiv 1\mod 4$} &
		\begin{tikzpicture}[scale=0.75]
			\filldraw[fill=blue] (0,0) circle (0.4);
			\filldraw[fill=purple] (1,0) circle (0.4);
			\filldraw[fill=green] (0,1) circle (0.4);
			\filldraw[fill=red] (1,1) circle (0.4);
		\end{tikzpicture} &
		\begin{tikzpicture}[scale=0.75]
			\draw (0,0) circle (0.4);
			\draw (1,0) circle (0.4);
			\filldraw[fill=green] (0,1) circle (0.4);
			\filldraw[fill=red] (1,1) circle (0.4);
		\end{tikzpicture} &
		\begin{tikzpicture}[scale=0.75]
			\draw (0,0) circle (0.4);
			\filldraw[fill=purple] (1,0) circle (0.4);
			\draw (0,1) circle (0.4);
			\filldraw[fill=red] (1,1) circle (0.4);
		\end{tikzpicture} &
		\begin{tikzpicture}[scale=0.75]
			\filldraw[fill=blue] (0,0) circle (0.4);
			\draw (1,0) circle (0.4);
			\draw (0,1) circle (0.4);
			\filldraw[fill=red] (1,1) circle (0.4);
		\end{tikzpicture} \\[5mm]
		\raisebox{7mm}{$p\equiv 3\mod 4$} &
		\begin{tikzpicture}[scale=0.75]
			\draw (0,0) circle (0.4);
			\draw (1,0) circle (0.4);
			\filldraw[fill=green] (0,1) circle (0.4);
			\filldraw[fill=red] (1,1) circle (0.4);
		\end{tikzpicture} &
		\begin{tikzpicture}[scale=0.75]
			\filldraw[fill=blue] (0,0) circle (0.4);
			\filldraw[fill=purple] (1,0) circle (0.4);
			\filldraw[fill=green] (0,1) circle (0.4);
			\filldraw[fill=red] (1,1) circle (0.4);
		\end{tikzpicture} &
		\begin{tikzpicture}[scale=0.75]
			\draw (0,0) circle (0.4);
			\filldraw[fill=purple] (1,0) circle (0.4);
			\draw (0,1) circle (0.4);
			\filldraw[fill=red] (1,1) circle (0.4);
		\end{tikzpicture} &
		\begin{tikzpicture}[scale=0.75]
			\filldraw[fill=blue] (0,0) circle (0.4);
			\draw (1,0) circle (0.4);
			\draw (0,1) circle (0.4);
			\filldraw[fill=red] (1,1) circle (0.4);
		\end{tikzpicture} \\
		& $1$ & $c_0$ & $p$ & $c_0p$
	\end{tabular}
	
	\includegraphics[width=0.6\linewidth]{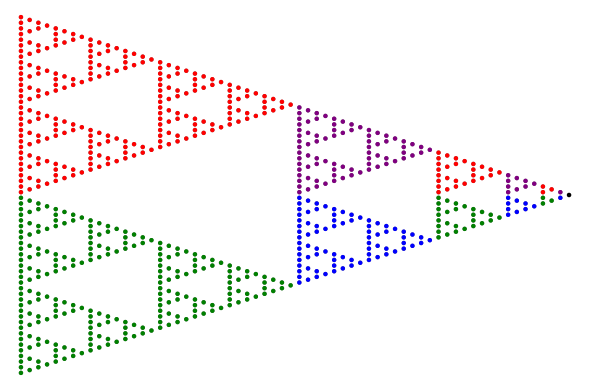}
	\caption{Top: $\DSq(\Qp,c)$ for $c\in\Qp$ and $p\ne 2$. In each group of four circles, the upper circles represent even order numbers and the lower circles odd order, and the right circles represent square leading digits and the left circles non-square digits. Bottom: these four classes depicted for $p=3$. Each circle ``contains'' the points with the same color and the black point at the right is $0$.}
	\label{fig:images-extra}
\end{figure}

\begin{proposition}\label{prop:images-extra2}
	$\DSq(\Q_2,c)$ is given as follows (see Figure \ref{fig:images-extra2}):
	\begin{enumerate}
		\item $\DSq(\Q_2,1)=\{u\in\Q_2:\digit_1(u)=0\}$;
		\item $\DSq(\Q_2,-1)=\Q_2$;
		\item $\DSq(\Q_2,2)=\{u\in\Q_2:\digit_2(u)=0\}$;
		\item $\DSq(\Q_2,-2)=\{u\in\Q_2:\digit_1(u)=\digit_2(u)\}$;
		\item $\DSq(\Q_2,3)=\{u\in\Q_2:\ord_2(u)\equiv 0\mod 2\}$;
		\item $\DSq(\Q_2,-3)=\{u\in\Q_2:\ord_2(u)+\digit_1(u)\equiv 0\mod 2\}$;
		\item $\DSq(\Q_2,6)=\{u\in\Q_2:\ord_2(u)+\digit_1(u)+\digit_2(u)\equiv 0\mod 2\}$;
		\item $\DSq(\Q_2,-6)=\{u\in\Q_2:\ord_2(u)+\digit_2(u)\equiv 0\mod 2\}$.
	\end{enumerate}
\end{proposition}

\begin{proof}
	Table \ref{table:image-extra} indicates the leading digits and the parity of the order of $a^2+cb^2$ depending on $c$ and the difference $\ord_2(b)-\ord_2(a)$. The result follows by putting together the cases in the same row of the table. Note that a case such as ``$011$ even'' covers \textit{all} $2$-adic numbers with even order and ending in $011$.
\end{proof}

\begin{table}
	\footnotesize
	\begin{center}
		\begin{tabular}{|c|c|c|c|c|c|c|}
			\hline & & \multicolumn{5}{c|}{$\ord_2(b)-\ord_2(a)$} \\ \cline{3-7}
			$c$ & $cb^2$ & $\le -2$ & $-1$ & $0$ & $1$ & $\ge 2$ \\ \hline
			$1$ & $001$ even & $001$ even & $101$ even & $01$ odd & $101$ even & $001$ even \\ \hline
			$-1$ & $111$ even & $111$ even & $011$ even & anything & $101$ even & $001$ even \\ \hline
			$2$ & $001$ odd & $001$ odd & $011$ odd & $011$ even & $001$ even & $001$ even \\ \hline
			$-2$ & $111$ odd & $111$ odd & $001$ odd & $111$ even & $001$ even & $001$ even \\ \hline
			$3$ & $011$ even & $011$ even & $111$ even & $1$ even & $101$ even & $001$ even \\ \hline
			$-3$ & $101$ even & $101$ even & $001$ even & $11$ odd & $101$ even & $001$ even \\ \hline
			$6$ & $011$ odd & $011$ odd & $101$ odd & $111$ even & $001$ even & $001$ even \\ \hline
			$-6$ & $101$ odd & $101$ odd & $111$ odd & $011$ even & $001$ even & $001$ even \\ \hline
		\end{tabular}
	\end{center}
	\caption{Leading digits and parity of the order of $a^2+cb^2$ depending on $c$ and the difference $\ord_2(b)-\ord_2(a)$. The number $a^2$ is always described as $001$ even, $cb^2$ depends exclusively on $c$, and the result of the addition of both terms depends in the offset between these digits. Note that the leading $1$'s will add up to $0$ if the offset is $0$, hence making the second digit the leading one in the cases ``$01$ odd'' and ``$11$ odd'' (in these cases the order increases in $1$), the third in the case ``$1$ even'' (the order increases in $2$), and giving any possible result when adding $001$ and $111$ at the same position.}
	\label{table:image-extra}
\end{table}

\begin{figure}
	\setlength{\tabcolsep}{3mm}
	\begin{tabular}{cccc}
		\begin{tikzpicture}
			\filldraw[fill=red] (1.2,1) circle (0.2);
			\filldraw[fill=orange] (0.8,1) circle (0.2);
			\draw (0.2,1) circle (0.2);
			\draw (-0.2,1) circle (0.2);
			\filldraw[fill=purple] (1.2,0) circle (0.2);
			\filldraw[fill=magenta] (0.8,0) circle (0.2);
			\draw (0.2,0) circle (0.2);
			\draw (-0.2,0) circle (0.2);
		\end{tikzpicture} &
		\begin{tikzpicture}
			\filldraw[fill=red] (1.2,1) circle (0.2);
			\filldraw[fill=orange] (0.8,1) circle (0.2);
			\filldraw[fill=green] (0.2,1) circle (0.2);
			\filldraw[fill=yellow] (-0.2,1) circle (0.2);
			\filldraw[fill=purple] (1.2,0) circle (0.2);
			\filldraw[fill=magenta] (0.8,0) circle (0.2);
			\filldraw[fill=blue] (0.2,0) circle (0.2);
			\filldraw[fill=cyan] (-0.2,0) circle (0.2);
		\end{tikzpicture} &
		\begin{tikzpicture}
			\filldraw[fill=red] (1.2,1) circle (0.2);
			\draw (0.8,1) circle (0.2);
			\filldraw[fill=green] (0.2,1) circle (0.2);
			\draw (-0.2,1) circle (0.2);
			\filldraw[fill=purple] (1.2,0) circle (0.2);
			\draw (0.8,0) circle (0.2);
			\filldraw[fill=blue] (0.2,0) circle (0.2);
			\draw (-0.2,0) circle (0.2);
		\end{tikzpicture} &
		\begin{tikzpicture}
			\filldraw[fill=red] (1.2,1) circle (0.2);
			\draw (0.8,1) circle (0.2);
			\draw (0.2,1) circle (0.2);
			\filldraw[fill=yellow] (-0.2,1) circle (0.2);
			\filldraw[fill=purple] (1.2,0) circle (0.2);
			\draw (0.8,0) circle (0.2);
			\draw (0.2,0) circle (0.2);
			\filldraw[fill=cyan] (-0.2,0) circle (0.2);
		\end{tikzpicture} \\
		$1$ & $-1$ & $2$ & $-2$ \\[5mm]
		\begin{tikzpicture}
			\filldraw[fill=red] (1.2,1) circle (0.2);
			\filldraw[fill=orange] (0.8,1) circle (0.2);
			\filldraw[fill=green] (0.2,1) circle (0.2);
			\filldraw[fill=yellow] (-0.2,1) circle (0.2);
			\draw (1.2,0) circle (0.2);
			\draw (0.8,0) circle (0.2);
			\draw (0.2,0) circle (0.2);
			\draw (-0.2,0) circle (0.2);
		\end{tikzpicture} &
		\begin{tikzpicture}
			\filldraw[fill=red] (1.2,1) circle (0.2);
			\filldraw[fill=orange] (0.8,1) circle (0.2);
			\draw (0.2,1) circle (0.2);
			\draw (-0.2,1) circle (0.2);
			\draw (1.2,0) circle (0.2);
			\draw (0.8,0) circle (0.2);
			\filldraw[fill=blue] (0.2,0) circle (0.2);
			\filldraw[fill=cyan] (-0.2,0) circle (0.2);
		\end{tikzpicture} &
		\begin{tikzpicture}
			\filldraw[fill=red] (1.2,1) circle (0.2);
			\draw (0.8,1) circle (0.2);
			\draw (0.2,1) circle (0.2);
			\filldraw[fill=yellow] (-0.2,1) circle (0.2);
			\draw (1.2,0) circle (0.2);
			\filldraw[fill=magenta] (0.8,0) circle (0.2);
			\filldraw[fill=blue] (0.2,0) circle (0.2);
			\draw (-0.2,0) circle (0.2);
		\end{tikzpicture} &
		\begin{tikzpicture}
			\filldraw[fill=red] (1.2,1) circle (0.2);
			\draw (0.8,1) circle (0.2);
			\filldraw[fill=green] (0.2,1) circle (0.2);
			\draw (-0.2,1) circle (0.2);
			\draw (1.2,0) circle (0.2);
			\filldraw[fill=magenta] (0.8,0) circle (0.2);
			\draw (0.2,0) circle (0.2);
			\filldraw[fill=cyan] (-0.2,0) circle (0.2);
		\end{tikzpicture} \\
		$3$ & $-3$ & $6$ & $-6$
	\end{tabular}
	
	\vspace{1cm}
	\includegraphics[width=0.8\linewidth]{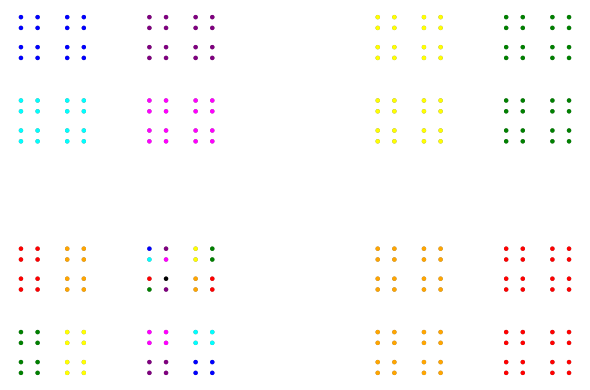}
	\caption{Top: $\DSq(\Q_2,c)$ for $c\in\Q_2$. In each group of eight circles, the upper circles represent even order numbers and the lower circles odd order, the two rightmost circles in the row represent a $0$ as second digit and the two leftmost circles a $1$, and in each pair of circles, the rightmost one has $0$ as third digit and the leftmost one has $1$. Bottom: a depiction of the eight classes. Each circle ``contains'' the points with the same color and the black point in the lower left is $0$.}
	\label{fig:images-extra2}
\end{figure}

Now we are ready to prove the matrix classification in dimension $2$ \cite[Theorems D and E]{CrePel-integrable}, using the previous results. In order to make the paper as self contained as possible, and the proof easier to follow, we recall the statement.

\begin{definition}[Non-residue sets and coefficient functions {\cite[Definition 1.1]{CrePel-integrable}}]\label{def:sets}
	\letpprime. If $p\equiv 1\mod 4$, let $c_0$ be the smallest quadratic non-residue modulo $p$. We define the \emph{non-residue sets}
	\[X_p=\begin{cases}
		\{1,c_0,p,c_0p,c_0^2p,c_0^3p,c_0p^2\} & \text{if }p\equiv 1\mod 4; \\
		\{1,-1,p,-p,p^2\} & \text{if }p\equiv 3\mod 4; \\
		\{1,-1,2,-2,3,-3,6,-6,12,-18,24\} & \text{if }p=2.
	\end{cases}\]
	\[Y_p=\begin{cases}
		\{c_0,p,c_0p\} & \text{if }p\equiv 1\mod 4; \\
		\{-1,p,-p\} & \text{if }p\equiv 3\mod 4; \\
		\{-1,2,-2,3,-3,6,-6\} & \text{if }p=2.
	\end{cases}\]
	We also define the \emph{coefficient functions} $\mathcal{C}_i^k:Y_p\times(\Qp)^4\to\Qp$ and $\mathcal{D}_i^k:Y_p\times(\Qp)^4\to\Qp$, for $k\in\{1,2\}$, $i\in\{0,1,2\}$, by
	\[\mathcal{C}_0^1(c,t_1,t_2,a,b)=\frac{ac}{2(c-b^2)},\,\,\mathcal{C}_1^1(c,t_1,t_2,a,b)=\frac{b}{b^2-c},\,\,\mathcal{C}_2^1(c,t_1,t_2,a,b)=\frac{1}{2a(c-b^2)},\]
	\[\mathcal{C}_0^2(c,t_1,t_2,a,b)=\frac{abc}{2(b^2-c)},\,\,\mathcal{C}_1^2(c,t_1,t_2,a,b)=\frac{c}{c-b^2},\,\,\mathcal{C}_2^2(c,t_1,t_2,a,b)=\frac{b}{2a(b^2-c)},\]
	\[\mathcal{D}_0^1(c,t_1,t_2,a,b)=-\frac{t_1+bt_2}{2a},\,\mathcal{D}_1^1(c,t_1,t_2,a,b)=-bt_1-ct_2,\,\mathcal{D}_2^1(c,t_1,t_2,a,b)=-\frac{ac(t_1+bt_2)}{2},\]
	\[\mathcal{D}_0^2(c,t_1,t_2,a,b)=-\frac{bt_1+ct_2}{2a},\,\mathcal{D}_1^2(c,t_1,t_2,a,b)=-c(t_1+bt_2),\,\mathcal{D}_2^2(c,t_1,t_2,a,b)=-\frac{ac(bt_1+ct_2)}{2}.\]
\end{definition}

\begin{theorem}[{\cite[Theorem D]{CrePel-integrable}}]\label{thm:williamson}
	\letpprime. Let $M\in\M_2(\Qp)$ be a symmetric matrix. Let $X_p,Y_p$ be the non-residue sets in Definition \ref{def:sets}. Then, there exists a symplectic matrix $S\in\M_2(\Qp)$ and either $c\in X_p$ and $r\in\Qp$, or $c=0$ and $r\in Y_p\cup\{1\}$, such that
	\[S^T MS=\begin{pmatrix}
		r & 0 \\
		0 & cr
	\end{pmatrix}.\]
	Furthermore, if two symplectic matrices $S$ and $S'$ reduce $M$ to the normal form of the right hand side of the equality above, then the two normal forms are equal.
\end{theorem}

\begin{euproof}
	\item First we prove existence. The characteristic polynomial of $\Omega_0^{-1}M$ has two opposite roots, which may or may not be in $\Qp$. If they are in $\Qp$ and are not $0$, Proposition \ref{prop:hyperbolic} implies that $M$ can be converted to
	\[\begin{pmatrix}
		0 & 1 \\
		1 & 0
	\end{pmatrix}\]
	except by a constant factor. But the matrix
	\[N=\begin{pmatrix}
		r & 0 \\
		0 & cr
	\end{pmatrix}\]
	in the statement, where $c=1$ if $p\equiv 1\mod 4$ and $c=-1$ otherwise, has also two eigenvalues in $\Qp$, so it can be converted to the same matrix.
	
	If the eigenvalues are $0$, either the matrix is zero, and we are in the same situation but with $r=0$, or they are not zero, and Proposition \ref{prop:hyp-degenerate} gives the same matrix but with $c=0$. In this case $r$ must be such that $ar$ is a square, where $a$ is one of the coefficients. There is one and only one $r\in Y_p\cup\{1\}$ such that this happens.
	
	Now suppose that the roots of the characteristic polynomial of $\Omega_0^{-1}M$ are $\pm\lambda$ for $\lambda\notin\Qp$. In this case, we must have $\lambda^2\in\Qp$. We have $N=S^TMS$ for a symplectic matrix $S$ and some $r$ and $c$ if the two conditions of Proposition \ref{prop:elliptic} hold for some $a$ and $b=ac$. The first condition reads
	\[\lambda^2=-a^2c\Rightarrow a^2=-\frac{\lambda^2}{c}.\]
	We must now split the proof into several cases.
	\begin{itemize}
		\item $\lambda^2$ has even order. In order for $-\lambda^2/c$ to be a square, we need $c$ of even order. We also know that $\lambda^2$ is not a square, so $-c$ must not be a square. The elements of $X_p$ which satisfy these conditions are $\{c_0,c_0p^2\}$ if $p\equiv 1\mod 4$, $\{1,p^2\}$ if $p\equiv 3\mod 4$, and $\{1,3,-3,12\}$ if $p=2$.
		
		For $p=2$, we also need that $-\lambda^2/c$ ends in $001$, so the three last digits in $c$ must agree with those in $-\lambda^2$, which will be $001$, $101$ or $011$ (not $111$, which would make $\lambda^2$ a square). This narrows further down the options to $\{1\}$, $\{-3\}$ or $\{3,12\}$, respectively.
		
		Let $C$ be the current set of options for $c$, which contains only the singleton $\{c_1\}$ or the set of two values $\{c_1,c_1p^2\}$. All of them satisfy that $-\lambda^2/c$ is a square. We still need to apply the second condition, that is, we need to choose $a$ such that \[\frac{2\lambda a}{u^T\Omega_0\bar{u}}\in\DSq(\Qp,-\lambda^2)=\DSq(\Qp,c_1).\]
		
		Let \[a_1=\sqrt{\frac{-\lambda^2}{c_1}}.\] In the two cases where $C=\{c_1\}$, by Proposition \ref{prop:images-extra2}(1) and (6), for any $x\in\Qp$, either $x$ or $-x$ is in $\DSq(\Qp,c_1)$. So, either $a=a_1$ or $a=-a_1$ satisfies the condition, and $c=c_1$ is valid.
		
		If \[C=\{c_1,c_1p^2\},\] by Proposition \ref{prop:images-extra}(1) and (2) and Proposition \ref{prop:images-extra2}(5),
		\[\DSq(\Qp,c_1)=\Big\{u\in\Qp:\ord_p(u)\equiv 0\mod 2\Big\}.\]
		This implies that either $a=a_1$ or $a=a_1/p$ satisfies the condition. Hence, only one between $c_1$ and $c_1p^2$ is a valid value of $c$.
		
		\item $\lambda^2$ has odd order. Now we need $c$ to have odd order instead of even. What happens next depends on $p$.
		\begin{itemize}
			\item If $p\equiv 1\mod 4$, the values of $c$ with odd order are $c_0^kp$ for $k=0,1,2,3$. The first condition implies that $-\lambda^2/c_0^kp$ is a square, which is true for $k=0$ and $2$ or for $k=1$ or $3$, depending on the leading digit of $-\lambda^2/p$. Let $c_1$ be the value which satisfies this between $p$ and $c_0p$. Now the two candidates for $c$ are $c_1$ and $c_1c_0^2$.
			
			We define again $a_1=\sqrt{-\lambda^2/c_1}$. In this case, by Proposition \ref{prop:images-extra}(3) and (4), for any $x\in\Qp$, either $x$ or $c_0x$ is in $\DSq(\Qp,c_1)$. Hence, either $a=a_1$ or $a=a_1/c_0$ satisfies the condition, and either $c=c_1$ or $c=c_1c_0^2$ is valid.
			
			\item If $p\equiv 3\mod 4$, the values of $c$ with odd order are $p$ and $-p$. As $-1$ is not a square, only one will make $-\lambda^2/c$ a square. For this value of $c$, we set $a_1=\sqrt{-\lambda^2/c}$. By Proposition \ref{prop:images-extra}(3) and (4), for any $x\in\Qp$, either $x$ or $-x$ is in $\DSq(\Qp,c_1)$. Hence, either $a=a_1$ or $a=-a_1$ is valid, and $c$ is valid in any case.
			
			\item If $p=2$, the values of $c$ with odd order are $2$, $-2$, $6$, $-6$, $-18$ and $24$. $-\lambda^2/c$ must end in $001$, so $c$ must agree with $-\lambda^2$ in the last three digits, that in this case can have all possible values: $001$, $101$, $011$ and $111$. The valid $c$'s in each case are, respectively,
			\[\{2\}, \{-6\}, \{6,24\}, \{-2,-18\}.\]
			Let $C$ be this set, $c_1$ the element of less absolute value (in the real sense) in $C$, and $a_1=\sqrt{-\lambda^2/c_1}$.
			
			If $C$ has only one element, by Proposition \ref{prop:images-extra2}(3) and (8), for any $x\in\Qp$, either $x$ or $-x$ is in $\DSq(\Qp,c_1)$, and again $c_1$ is valid in any case.
			
			If $C=\{-2,-18\}$, by Proposition \ref{prop:images-extra2}(4), for any $x\in\Qp$, either $x$ or $3x$ is in $\DSq(\Qp,-2)$ (because $\digit_1(x)=1-\digit_1(3x)$ and $\digit_2(x)=\digit_2(3x)$). Hence, either $a=a_1$ or $a=a_1/3$ works and $c=-2$ or $c=-18$, respectively, works.
			
			Finally, if $C=\{6,24\}$, by Proposition \ref{prop:images-extra2}(7), for any $x\in\Qp$, either $x$ or $2x$ is in $\DSq(\Qp,6)$. So either $a=a_1$ or $a=a_1/2$ works and $c=6$ or $c=24$, respectively, works.
		\end{itemize}
	\end{itemize}
	\item Now we prove uniqueness. In the case where the roots of the characteristic polynomial are in $\Qp$, the rest of $c$'s in the lists do not lead to a matrix with the eigenvalues in $\Qp$, because their opposites are not squares. In all the other cases, we have seen that there is one and only one valid value of $c$. If two matrices $S$ and $S'$ bring $M$ to normal form, $c$ must coincide, hence $r$ also coincides because the eigenvalues of the normal forms must be the same.
\end{euproof}

The following statement indicates the amount of normal forms which appear in Theorem \ref{thm:williamson}.

\begin{theorem}[{\cite[Theorem E]{CrePel-integrable}}]\label{thm:num-forms2}
	\letpprime. Let $X_p,Y_p$ be the non-residue sets in definition \ref{def:sets}. Then the following statements hold.
	\begin{enumerate}
		\item If $p\equiv 1\mod 4$, there are exactly $7$ infinite families of normal forms of $2$-by-$2$ $p$-adic matrices with one degree of freedom up to congruence via a symplectic matrix:
		\[\Big\{\Big\{r\begin{pmatrix}
			1 & 0 \\
			0 & c
		\end{pmatrix}:r\in\Qp\Big\}:c\in X_p\Big\},\]
		and exactly $4$ isolated normal forms, which correspond to $c=0$:
		\[\Big\{\begin{pmatrix}
			r & 0 \\
			0 & 0
		\end{pmatrix}:r\in Y_p\cup\{1\}\Big\}.\]
		\item If $p\equiv 3\mod 4$, there are exactly $5$ infinite families of normal forms of $2$-by-$2$ $p$-adic matrices with one degree of freedom up to congruence via a symplectic matrix, with the same formula as above, and exactly $4$ isolated normal forms.
		\item If $p=2$, there are exactly $11$ infinite families of normal forms of $2$-by-$2$ $p$-adic matrices with one degree of freedom up to congruence via a symplectic matrix, also with the same formula, and exactly $8$ isolated normal forms.
	\end{enumerate}
	This is in contrast with the real case, where there are exactly $2$ families, elliptic and hyperbolic, and $2$ isolated normal forms. Here by ``infinite family'' we mean a family of normal forms of the form $r_1M_1+r_2M_2+\ldots+r_kM_k$, where $r_i\in\Qp$ are parameters and $k$ is the number of degrees of freedom.
\end{theorem}

\begin{proof}
	This follows directly from Theorem \ref{thm:williamson}: the isolated normal forms correspond to the different values of $r$ for $c=0$, and each family of normal forms corresponds to a value of $c\in X_p$ with $c\ne 0$.
\end{proof}

Recall that, in the real case, there are two normal forms: elliptic and hyperbolic. To give them the form
\[\begin{pmatrix}
	r & 0 \\
	0 & cr
\end{pmatrix},\]
we need $c=1$ and $c=-1$, respectively. In the $p$-adic case, these two matrices are equivalent by multiplication by a symplectic matrix if and only if $p\equiv 1\mod 4$: the list for this case is the only one that does not contain $-1$.

\begin{proposition}\label{prop:choice}
	All choices of quadratic residue in Definition \ref{def:sets} (including the least of all, which is used in the definition) lead to the same normal form in Theorem \ref{thm:williamson}, up to congruence via a symplectic matrix. That is, if $c_0$, $c'_0$ are two quadratic residues modulo $p$ then the normal forms corresponding to $c_0$ and the normal forms corresponding to $c'_0$ are equivalent by multiplication by a symplectic matrix. (The \textit{order} of the forms, however, may vary: for example, taking $c_0'\equiv c_0^3\mod p$ results in the new form with $c=c_0'p$ being taken as the one which had previously $c=c_0^3p$.)
\end{proposition}

\begin{proof}
	By Theorem \ref{thm:williamson}, the normal forms of matrices in the first set are equivalent to one and only one normal form in the second set, so the two sets are different representatives of the same classes.
\end{proof}

\section{General preparatory results for fields of characteristic different from $2$}\label{sec:dim4-gen}

In this section we prove some results which are not specific to the reals or the $p$-adics, but which are needed to prove the $p$-adic matrix classification in dimension $4$ (\cite[Theorems F, G and H]{CrePel-integrable}) and the real classification in any dimension (\cite[Theorems K and L]{CrePel-integrable}). They have to do with finding normal forms for $4$-by-$4$ matrices in terms of their eigenvalues.

\begin{proposition}\label{prop:focus}
	Let $F$ be a field of characteristic different from $2$. Let $\Omega_0$ be the matrix of the standard symplectic form on $F^4$. Let $M\in\M_4(F)$ be a symmetric matrix such that the eigenvalues of $\Omega_0^{-1}M$ are of the form $\pm\lambda,\pm\mu$ where $\lambda,\mu\notin F$ but all of them are in a degree $2$ extension $F[\alpha]$ for some $\alpha$. Then there are $r,s\in F$ and a symplectic matrix $S\in\M_4(F)$ such that
	$S^TMS$ has the form of Theorem \ref{thm:williamson4}(2) for $c=\alpha^2$.
\end{proposition}

\begin{proof}
	The condition in the statement is equivalent to say that $\mu=\bar{\lambda}$ and both are in $F[\alpha]$. Let $\lambda=s+r\alpha$ and $\mu=s-r\alpha$, for $r,s\in F$.
	
	Let
	\[N=\begin{pmatrix}
		0 & s & 0 & r \\
		s & 0 & r\alpha^2 & 0 \\
		0 & r\alpha^2 & 0 & s \\
		r & 0 & s & 0
	\end{pmatrix},\]
	which is the matrix in Theorem \ref{thm:williamson4}(2) for $c=\alpha^2$. The eigenvalues of
	\[\Omega_0^{-1}N=\begin{pmatrix}
		-s & 0 & -r\alpha^2 & 0 \\
		0 & s & 0 & r \\
		-r & 0 & -s & 0 \\
		0 & r\alpha^2 & 0 & s
	\end{pmatrix}\]
	are precisely $\pm s\pm r\alpha$, that is, $\pm\lambda$ and $\pm\mu$, so we can apply Corollary \ref{cor:symplectomorphic}.
	The matrix $\Psi_2$ has the form
	\[\begin{pmatrix}
		0 & \alpha & 0 & -\alpha \\
		1 & 0 & 1 & 0 \\
		0 & 1 & 0 & 1 \\
		\alpha & 0 & -\alpha & 0
	\end{pmatrix}\]
	with the values in the order $(\lambda,-\lambda,\bar{\lambda},-\bar{\lambda})$ as needed by Corollary \ref{cor:symplectomorphic}, and the resulting matrix $S$ is
	\[S=\Psi_1\begin{pmatrix}
		d_1 & 0 & 0 & 0 \\
		0 & d_2 & 0 & 0 \\
		0 & 0 & d_3 & 0 \\
		0 & 0 & 0 & d_4
	\end{pmatrix}
	\begin{pmatrix}
		0 & 1 & 0 & \frac{1}{\alpha} \\
		\frac{1}{\alpha} & 0 & 1 & 0 \\
		0 & 1 & 0 & -\frac{1}{\alpha} \\
		-\frac{1}{\alpha} & 0 & 1 & 0
	\end{pmatrix}.\]
	The condition \eqref{eq:relation} implies that
	\begin{equation}\label{eq:relation1}
		d_1d_2=\frac{(0,1,0,\alpha)\Omega_0(\alpha,0,1,0)^T}{u^T\Omega_0 v}=\frac{-2\alpha}{u^T\Omega_0 v}
	\end{equation}
	and
	\begin{equation}\label{eq:relation2}
		d_3d_4=\frac{(0,1,0,-\alpha)\Omega_0(-\alpha,0,1,0)^T}{\bar{u}^T\Omega_0\bar{v}}=\frac{2\alpha}{\bar{u}^T\Omega_0\bar{v}},
	\end{equation}
	where $u$ and $v$ are the eigenvectors for $\lambda$ and $-\lambda$ respectively.
	
	We also want $S$ to have entries in $F$. Let $c_1,c_2,c_3,c_4$ be its columns, which should be vectors in $F^4$:
	\[\left\{\begin{aligned}
		c_1 & =\frac{d_2v-d_4\bar{v}}{\alpha}; \\
		c_2 & =d_1u+d_3\bar{u}; \\
		c_3 & =d_2v+d_4\bar{v}; \\
		c_4 & =\frac{d_1u-d_3\bar{u}}{\alpha}.
	\end{aligned}\right.\]
	These expressions imply that
	\[2d_1u=c_2+\alpha c_4, 2d_3\bar{u}=c_2-\alpha c_4=\overline{2d_1u}\Rightarrow d_3=\bar{d}_1\]
	and
	\[2d_2v=c_3+\alpha c_1, 2d_4\bar{v}=c_3-\alpha c_1=\overline{2d_2v}\Rightarrow d_4=\bar{d}_2.\]
	Now we can take $d_1$ and $d_2$ arbitrary such that \eqref{eq:relation1} holds, and \eqref{eq:relation2} will hold automatically because $d_3d_4=\overline{d_1d_2}$.
\end{proof}

In the real case, this is enough to complete the Weierstrass-Williamson classification in all dimensions if the eigenvalues of $\Omega_0^{-1}M$ are pairwise distinct. Indeed, these eigenvalues can be associated in pairs of the form $\{a,-a\}$ or $\{\ii a,-\ii a\}$ and quadruples of the form \[\Big\{a+\ii b,a-\ii b,-a+\ii b,-a-\ii b\Big\}.\] We can apply respectively Propositions \ref{prop:hyperbolic}, \ref{prop:elliptic} (we already explained why this is always possible) and \ref{prop:focus}, giving the hyperbolic, elliptic and focus-focus normal forms in Section \ref{sec:algclosed}.

In the $p$-adic case, such a classification is still not complete, even for $4$-by-$4$ matrices. The reason for this difference is that, if $\alpha\notin\R$ with $\alpha^2\in\R$, this means that $\alpha$ is an imaginary number and $\R[\alpha]=\C$, which is algebraically closed. But if $\alpha\notin\Qp$ with $\alpha^2\in\Qp$, $\Qp[\alpha]$ is not algebraically closed. So it is possible that $\lambda^2\notin\Qp$ and simultaneously $\lambda\notin\Qp[\lambda^2]$.

To fix this issue, consider a degree four polynomial of the form $t^4+At^2+B$ (at the moment in an arbitrary field $F$). Its roots are of the form $\lambda,-\lambda,\mu,-\mu$. If $\lambda^2$ and $\mu^2$ are not in $F$, they are conjugate in some degree 2 extension, that is, $\lambda^2=a+b\alpha$ and $\mu^2=a-b\alpha$ for some $\alpha\in F[\lambda^2]$. In turn, if $\lambda$ and $\mu$ are not in $F[\lambda^2]$, we have a hierarchy of fields:

\begin{center}
	\begin{tikzpicture}
		\node (a) at (0,0){$F$};
		\node (b) at (0,1.5){$F[\alpha]=F[\lambda^2]=F[\mu^2]$};
		\node (c) at (0,3){$F[\lambda,\mu]$};
		\draw (a) -- node[right]{2} (b) -- node[right]{2} (c);
	\end{tikzpicture}
\end{center}

There is an automorphism of $F[\lambda,\mu]$ that fixes $F$ and moves $\alpha$ to $-\alpha$ (an extension of the conjugation in $F[\alpha]$). We will denote this as $x\mapsto\bar{x}$. $\bar{\lambda}$ must be $\mu$ or $-\mu$, without loss of generality we suppose that it is $\mu$. There is another automorphism of $F[\lambda,\mu]$ that fixes $F[\alpha]$ and changes $\lambda$ to $-\lambda$ and $\mu$ to $-\mu$, which we will denote as $x\mapsto\hat{x}$.

\begin{proposition}\label{prop:exotic}
	Let $F$ be a field with characteristic different from $2$. Let $F[\alpha]$ be a degree two extension of $F$ and let $F[\gamma,\bar{\gamma}]$ be an extension of $F[\alpha]$ such that $\gamma^2\in F[\alpha]$. Let $\Omega_0$ be the matrix of the standard symplectic form on $F^4$. Let $t_1,t_2\in F$ such that
	\[
	\left\{\begin{aligned}
		\gamma^2 & =t_1+t_2\alpha; \\
		\bar{\gamma}^2 & =t_1-t_2\alpha.
	\end{aligned}\right.
	\]
	
	Let $M\in\M_4(F)$ be a symmetric matrix such that the eigenvalues of $\Omega_0^{-1}M$ are of the form $\pm\lambda,\pm\mu$ with \[\lambda=(r+s\alpha)\gamma\] and \[\mu=(r-s\alpha)\bar{\gamma},\] for $r,s\in F$. Let $a,b\in F$. Let $u$ be the eigenvector of $\lambda$. Then, there is a symplectic matrix $S\in\M_4(F)$ such that
	$S^TMS$ has the form of Theorem \ref{thm:williamson4}(3) with $c=\alpha^2$
	if and only if
	\[\frac{a\alpha\gamma(b+\alpha)}{u^T\Omega_0 \hat{u}}\in\DSq(F[\alpha],-\gamma^2).\]
\end{proposition}
\begin{proof}
	Let $N$ be the matrix
	\[
	\renewcommand{\arraystretch}{2}
	\begin{pmatrix}
		\dfrac{a\alpha^2(r-bs)}{\alpha^2-b^2} & 0 & \dfrac{s\alpha^2-rb}{\alpha^2-b^2} & 0 \\
		0 & \dfrac{-r(t_1+bt_2)-s(bt_1+\alpha^2t_2)}{a} & 0 & -r(bt_1+\alpha^2t_2)-s\alpha^2(t_1+bt_2) \\
		\dfrac{s\alpha^2-rb}{\alpha^2-b^2} & 0 & \dfrac{r-bs}{a(\alpha^2-b^2)} & 0 \\
		0 & -r(bt_1+\alpha^2t_2)-s\alpha^2(t_1+bt_2) & 0 & a\alpha^2(-r(t_1+bt_2)-s(bt_1+\alpha^2t_2))
	\end{pmatrix}.\]
	
	We have that $\Omega_0^{-1}N$ is equal to
	\[
	\renewcommand{\arraystretch}{2}
	\begin{pmatrix}
		0 & \dfrac{r(t_1+bt_2)+s(bt_1+\alpha^2t_2)}{a} & 0 & r(bt_1+\alpha^2t_2)+s\alpha^2(t_1+bt_2) \\
		\dfrac{a\alpha^2(r-bs)}{\alpha^2-b^2} & 0 & \dfrac{s\alpha^2-br}{\alpha^2-b^2} & 0 \\
		0 & r(bt_1+\alpha^2t_2)+s\alpha^2(t_1+bt_2) & 0 & a\alpha^2(r(t_1+bt_2)+s(bt_1+\alpha^2t_2)) \\
		\dfrac{s\alpha^2-br}{\alpha^2-b^2} & 0 & \dfrac{r-bs}{a(\alpha^2-b^2)} & 0
	\end{pmatrix}\]
	which has as set of eigenvalues \[\Big\{\pm(r+s\alpha)\gamma,\pm(r-s\alpha)\bar{\gamma}\Big\}=\Big\{\lambda,-\lambda,\mu,-\mu\Big\},\] and the condition of Corollary \ref{cor:symplectomorphic} is satisfied. The matrix
	\[\Psi_2=\begin{pmatrix}
		(b+\alpha)\gamma & -(b+\alpha)\gamma & (b-\alpha)\bar{\gamma} & -(b-\alpha)\bar{\gamma} \\
		a\alpha & a\alpha & -a\alpha & -a\alpha \\
		a\alpha\gamma(b+\alpha) & -a\alpha\gamma(b+\alpha) & -a\alpha\bar{\gamma}(b-\alpha) & a\alpha\bar{\gamma}(b-\alpha) \\
		1 & 1 & 1 & 1
	\end{pmatrix}\]
	and the columns of $\Psi_1$ are the eigenvectors of $\lambda,-\lambda,\mu$ and $-\mu$ in that order, which means that they are of the form $u,\hat{u},\bar{u}$ and $\hat{\bar{u}}$. Let $c_1,c_2,c_3$ and $c_4$ be the columns of $S$, which we want to be in $F$. We have that $\Psi_2S=\Psi_1 D$, that is,
	\[\Psi_2\begin{pmatrix}
		c_1 & c_2 & c_3 & c_4
	\end{pmatrix}
	=\begin{pmatrix}
		u & \hat{u} & \bar{u} & \hat{\bar{u}}
	\end{pmatrix}
	\begin{pmatrix}
		d_1 & 0 & 0 & 0 \\
		0 & d_2 & 0 & 0 \\
		0 & 0 & d_3 & 0 \\
		0 & 0 & 0 & d_4
	\end{pmatrix}=
	\begin{pmatrix}
		d_1u & d_2\hat{u} & d_3\bar{u} & d_4\hat{\bar{u}}
	\end{pmatrix}\]
	which expands to
	\[d_1u=(b+\alpha)\gamma c_1+a\alpha c_2+a\alpha\gamma(b+\alpha) c_3+c_4;\]
	\[d_2\hat{u}=-(b+\alpha)\gamma c_1+a\alpha c_2-a\alpha\gamma(b+\alpha) c_3+c_4;\]
	\[d_3\bar{u}=(b-\alpha)\bar{\gamma} c_1-a\alpha c_2-a\alpha\bar{\gamma}(b-\alpha) c_3+c_4;\]
	\[d_4\hat{\bar{u}}=-(b-\alpha)\bar{\gamma} c_1-a\alpha c_2+a\alpha\bar{\gamma}(b-\alpha) c_3+c_4;\]
	that is
	\[d_2\hat{u}=\widehat{d_1 u}\Rightarrow d_2=\hat{d}_1;\]
	\[d_3\bar{u}=\overline{d_1u}\Rightarrow d_3=\bar{d}_1;\]
	\[d_4\hat{\bar{u}}=\widehat{d_3\bar{u}}\Rightarrow d_4=\hat{d}_3=\hat{\bar{d}}_1.\]
	It only remains to apply the condition about $D$ in Corollary \ref{cor:symplectomorphic}. The result is
	\[d_1d_2=d_1\hat{d}_1=\frac{4a\alpha\gamma(b+\alpha)}{u^T\Omega_0 \hat{u}};\]
	\[d_3d_4=\bar{d}_1\hat{\bar{d}}_1=\frac{-4a\alpha\bar{\gamma}(b-\alpha)}{\bar{u}^T\Omega_0\hat{\bar{u}}}.\]
	Obviously, both conditions are equivalent. Putting $d_1=2(z_1+z_2\gamma)$ with $z_1,z_2\in F[\alpha]$, the first condition becomes
	\[z_1^2-z_2^2\gamma^2=\frac{a\alpha\gamma(b+\alpha)}{u^T\Omega_0 \hat{u}}\]
	so this must be in $\DSq(F[\alpha],-\gamma^2)$.
\end{proof}

Proposition \ref{prop:exotic} reduces the problem of classifying this type of matrices to classifying (adequate) degree $4$ extensions of $F$. In our case $F=\Qp$.

The values of $\alpha^2$ are given by \[\Qp^*/\Sq(\Qp^*),\] because two values whose quotient is a square are equivalent. The squares in $\Qp$ are the numbers with even order and with a leading digit in $\Sq(\F_p)$, so the quotient is $\{1,c_0,p,c_0p\}$, except if $p=2$, where squares have even order and end in $001$, and the quotient is $\{1,-1,2,-2,3,-3,6,-6\}$. The possible values of $\alpha^2$ are all except $1$ (because $\alpha\notin\Qp$). In each case, to find the normal form of $M$ we still need to know $\gamma$, $a$ and $b$.

In turn, the values of $\gamma$ are given by the quotient \[\Qp[\alpha]^*/\Sq(\Qp[\alpha]^*).\] So we need to determine which numbers are squares in $\Qp[\alpha]$.

\begin{lemma}\label{lemma:square}
	Let $F$ be a field with characteristic different from $2$. Let $F[\alpha]$ be a degree $2$ extension of $F$. Let $a,b\in F$. Then $a+b\alpha$ is a square in $F[\alpha]$ if and only if $a^2-b^2\alpha^2$ is a square in $F$ and one of the numbers
	\[\frac{a\pm\sqrt{a^2-b^2\alpha^2}}{2}\]
	is also a square in $F$.
\end{lemma}

\begin{proof}
	Suppose that
	$a+b\alpha=(r+s\alpha)^2.$
	Then also
	$a-b\alpha=(r-s\alpha)^2$,
	multiplying
	$a^2-b^2\alpha^2=(r^2-s^2\alpha^2)^2$,
	and finally
	\[\frac{a\pm\sqrt{a^2-b^2\alpha^2}}{2}=\frac{r^2+s^2\alpha^2+r^2-s^2\alpha^2}{2}=r^2.\]
	Reciprocally, if both numbers are squares, they give $r$ and $s$ such that $a+b\alpha=(r+s\alpha)^2$: we get $r$ from the previous formula, and then $s$ from $b=2rs$.
\end{proof}

By Proposition \ref{prop:exotic}, for fixed values of $\alpha$ and $\gamma$, two normal forms for $(a,b)$ and $(a',b')$ are equivalent if and only if the quotient between $a(b+\alpha)$ and $a'(b'+\alpha)$ is an element of $\DSq(\Qp[\alpha],-\gamma^2)$. Hence, the normal forms for fixed $\alpha$ and $\gamma$ are given by the classes in \[\Qp[\alpha]^*/\DSq(\Qp[\alpha],-\gamma^2),\] or equivalently in \[(\Qp[\alpha]^*/\Sq(\Qp[\alpha]^*))/\bDSq(\Qp[\alpha],-\gamma^2).\] That is, we have reduced the problem of determining subgroups of $\Qp[\alpha]$ to determining the subgroup \[\bDSq(\Qp[\alpha],-\gamma^2),\] which is easier because this is a subgroup of a finite group (and in all cases of interest, isomorphic to $\F_2^n$ for some $n$).

\section{$p$-adic matrix classification in dimension $4$}\label{sec:dim4}

We will now provide the $p$-adic classification of $4$-by-$4$ matrices, that is, \cite[Theorems F, G and H]{CrePel-integrable}. Some cases are a consequence of the previous results: the characteristic polynomial of $\Omega_0^{-1}M$ has the form \[At^4+Bt^2+C\] and the roots are $\lambda$, $-\lambda$, $\mu$ and $-\mu$. If $\lambda^2$ is in $\Qp$, then $\mu^2$ is also in $\Qp$. If $\lambda\ne\mu$, we can multiply $M$ by a symplectic matrix to separate it into two components, one with the eigenvalues $\lambda$ and $-\lambda$, and the other with $\mu$ and $-\mu$, and apply to each component Theorem \ref{thm:williamson}. In the real case this results in three rank $0$ normal forms: elliptic-elliptic, elliptic-hyperbolic and hyperbolic-hyperbolic. In the $p$-adic case, we have to combine analogously the normal forms for dimension $2$, getting a total of $\binom{8}{2}=28$ forms for $p\equiv 1\mod 4$, $\binom{6}{2}=15$ for $p\equiv 3\mod 4$, and $\binom{12}{2}=66$ for $p=2$.

The other case is when $\lambda^2\notin\Qp$. In this case, $\lambda^2$ and $\mu^2$ are conjugate roots in a degree two extension $L$ of $\Qp$. If $\lambda^2$ is a square in $L$, that is, $\lambda\in L$, we have $\mu^2=\overline{\lambda^2}$ and $\mu=\bar{\lambda}$ is also in $L$. We will now see that, in this case, the necessary condition of having the same eigenvalues is also sufficient to be congruent via a symplectic matrix.

\subsection{Case $p\ne 2$}\label{sec:dim4-no2}

We start with the case $p\equiv 1\mod 4$. This subdivides in three cases, depending on whether $\alpha^2$ is $c_0$, $p$ or $c_0p$. The following result gives a characterization of the squares in $\Qp[\alpha]$ in each case.

\begin{proposition}\label{prop:squares1}
	Let $p$ be a prime number such that $p\equiv 1\mod 4$. Let $c_0$ be a quadratic non-residue modulo $p$. Then the following statements hold.
	\begin{enumerate}
		\item $\Sq(\Qp[\sqrt{c_0}]^*)=\{a+b\sqrt{c_0}:a,b\in\Qp,\ord_p(a)\le\ord_p(b),\ord_p(a)\equiv 0\mod 2,a^2-b^2c_0\in\Sq(\Qp^*)\}$.
		\item $\Sq(\Qp[\sqrt{p}]^*)=\{a+b\sqrt{p}:a,b\in\Qp,\ord_p(a)\le\ord_p(b),\digit_0(a)\in\Sq(\F_p^*)\}$.
		\item $\Sq(\Qp[\sqrt{c_0p}]^*)=\{a+b\sqrt{c_0p}:a,b\in\Qp,\ord_p(a)\le\ord_p(b),\ord_p(a)\equiv 0\mod 2,\digit_0(a)\in\Sq(\F_p^*)\}\cup\{a+b\sqrt{c_0p}:\ord_p(a)\le\ord_p(b),\ord_p(a)\not\equiv 0\mod 2,\digit_0(a)\notin\Sq(\F_p^*)\}$.
	\end{enumerate}
\end{proposition}

\begin{proof}
	\begin{enumerate}
		\item Suppose that $a+b\sqrt{c_0}$ is a square in $\Qp$. By Lemma \ref{lemma:square}, \[a^2-b^2c_0=(r^2-s^2c_0)^2\] for some $r,s\in\Qp$. In particular, $a^2-b^2c_0$ is a square in $\Qp$. If $\ord(a)$ was higher than $\ord(b)$, that would make \[\digit_0(a^2-b^2c_0)=\digit_0(-b^2c_0)\notin\Sq(\F_p^*).\] So $\ord(a)\le\ord(b)$. We also have that $a=r^2+s^2c_0$, and by Proposition \ref{prop:images-extra}(2), $\ord(a)$ is even.
		
		Suppose now that $a$ and $b$ satisfy the three conditions. Then the first condition in Lemma \ref{lemma:square} is satisfied. Let $t_1$ and $t_2$ be the two candidates for $r^2$. Note that $t_1t_2=b^2c_0/4$.
		
		The leading terms cannot cancel simultaneously in $a+\sqrt{a^2-b^2c_0}$ and $a-\sqrt{a^2-b^2c_0}$. Without loss of generality, suppose that $t_1$ has no cancellation. Then $\ord(t_1)=\ord(a)$, which is even, and $\ord(t_2)=\ord(b^2c_0/4t_1)$ is also even. But their product $t_1t_2=b^2c_0/4$ is a non-square, hence one of $t_1$ and $t_2$ is a square (because the product of two even-order non-squares is a square).
		
		\item Suppose that $a+b\sqrt{p}$ is a square in $\Qp$. By Lemma \ref{lemma:square}, \[a^2-b^2p=(r^2-s^2p)^2\] for some $r,s\in\Qp$. This implies that $\ord(a^2-b^2p)$ is even, so $\ord(a)\le\ord(b)$. Here $a=r^2+s^2p$, so by Proposition \ref{prop:images-extra}(3), $\digit_0(a)\in\Sq(\F_p^*)$.
		
		Reciprocally, if $a$ and $b$ satisfy the conditions, $a^2-b^2p$ is a square because $a^2$ is. $t_1=(a+\sqrt{a^2-b^2p})/2$ has the same order and leading digit than $a$, so if $a$ is a square, $t_1$ is also a square. Otherwise, $a$ is $p$ times a square and the same applies to $t_1$, and $t_2=b^2p/4t_1$ is a square.
		
		\item Suppose that $a+b\sqrt{c_0p}$ is a square in $\Qp$. By Lemma \ref{lemma:square}, \[a^2-b^2c_0p=(r^2-s^2c_0p)^2\] for some $r,s\in\Qp$. This implies that $\ord(a^2-b^2c_0p)$ is even, so $\ord(a)\le\ord(b)$. Now we have $a=r^2+s^2c_0p$, which by Proposition \ref{prop:images-extra}(4) implies that either the order of $a$ is even and its leading digit is square, or the order is odd and the leading digit is non-square.
		
		Reciprocally, if $a$ and $b$ satisfy the conditions, $a^2-b^2c_0p$ is a square because $a^2$ is. In the first case \[t_1=\frac{a+\sqrt{a^2-b^2c_0p}}{2}\] is a square. In the second case, $a$ is $p$ times an even order non-square, $t_1$ is the same, and \[t_2=\frac{b^2c_0p}{4t_1}\] is a square. \qedhere
	\end{enumerate}
\end{proof}

\begin{corollary}\label{cor:squares1}
	Let $p$ be a prime number such that $p\equiv 1\mod 4$. Let $c_0$ be a quadratic non-residue modulo $p$. Then the following statements hold.
	\begin{enumerate}
		\item $\Qp[\sqrt{c_0}]^*/\Sq(\Qp[\sqrt{c_0}]^*)=\{1,p,\sqrt{c_0}, p\sqrt{c_0}\}$.
		\item $\Qp[\sqrt{p}]^*/\Sq(\Qp[\sqrt{p}]^*)=\{1,c_0,\sqrt{p}, c_0\sqrt{p}\}$.
		\item $\Qp[\sqrt{c_0p}]^*/\Sq(\Qp[\sqrt{c_0p}]^*)=\{1,c_0,\sqrt{c_0p}, c_0\sqrt{c_0p}\}$.
	\end{enumerate}
\end{corollary}

\begin{proof}
	\begin{enumerate}
		\item Given an element $a+b\sqrt{c_0}$ in $\Qp[\sqrt{c_0}]^*$, if $\ord(a)>\ord(b)$ or they are equal and $a^2-b^2c_0$ is not a square, we multiply it by $\sqrt{c_0}$. This guarantees that $\ord(a)\le\ord(b)$ and $a^2-b^2c_0$ is a square, because
		\[\sqrt{c_0}(a+b\sqrt{c_0})=a\sqrt{c_0}+bc_0\]
		and
		\[b^2c_0^2-a^2c_0=c_0(a^2-b^2c_0)\]
		so if $a^2-b^2c_0$ was non-square, it is now square. Hence, if the order of $a$ is odd, we multiply the element by $p$, and we obtain a square.
		\item Given an element $a+b\sqrt{p}$ in $\Qp[\sqrt{p}]^*$, if $\ord(a)>\ord(b)$, we multiply it by $\sqrt{p}$, so that it has $\ord(a)\le\ord(b)$. Now, multiplying it by $c_0$ if needed, we ensure that $\digit_0(a)$ is a square.
		\item Given an element $a+b\sqrt{c_0p}$ in $\Qp[\sqrt{c_0p}]^*$, if $\ord(a)>\ord(b)$, we multiply it by $\sqrt{c_0p}$, so that it has $\ord(a)\le\ord(b)$. Now, multiplying it by $c_0$ if needed, we ensure that $\digit_0(a)$ is a square or a non-square, depending on the order. \qedhere
	\end{enumerate}
\end{proof}

The element $\gamma$ is the square root of an element in this set, but different from $1$, which would lead to the case of Proposition \ref{prop:focus}. So there are three possible $\gamma$'s for each $\alpha$. Also, note that $\gamma^2$ always is in $\Qp$ or $\alpha$ times an element of $\Qp$: this means that $\bar{\gamma}^2$ is $\gamma^2$ or $-\gamma^2$, that is, $\bar{\gamma}$ is $\gamma$ or $\ii\gamma$. In any case, $\bar{\gamma}\in\Qp[\gamma]$ (here it is important that $p\equiv 1\mod 4$ so that $\ii\in\Qp$), or in other words, the extension $\Qp[\gamma,\bar{\gamma}]$ is the same as $\Qp[\gamma]$, which is different for each $\gamma$.

The next step is to determine \[\bDSq(\Qp[\alpha],-\gamma^2),\] or equivalently \[\bDSq(\Qp[\alpha],\gamma^2),\] because $-1$ is a square. This consists of seeing which classes of $\Qp[\alpha]^*$ modulo a square are attainable by elements of the form $x^2+\gamma^2$ for different $x$. Once this is done, the quotient of $\Qp[\alpha]^*/\Sq(\Qp[\alpha]^*)$ by this subgroup will give us the necessary $a$ and $b$. The computations are shown in Table \ref{table:normalforms1}.

\begin{table}\footnotesize
	\begin{tabular}{|c|c|c|c|c|c|c|c|}\hline
		$\alpha^2$ & all classes & $\gamma^2$ & attainable classes & $a$ & $b$ & $a(b+\alpha)$ & $[a(b+\alpha)]$ \\ \hline
		$c_0$ & $1,p,\sqrt{c_0},p\sqrt{c_0}$ & $p$ & $1,p$ & $1$ & $0$ & $\sqrt{c_0}$ & $\sqrt{c_0}$ \\ \cline{5-8}
		& & & & $p$ & $1/p$ & $1+p\sqrt{c_0}$ & $1$ \\ \cline{3-8}
		& & $\sqrt{c_0}$ & $1,\sqrt{c_0}$ & $1$ & $0$ & $\sqrt{c_0}$ & $\sqrt{c_0}$ \\ \cline{5-8}
		& & & & $p$ & $0$ & $p\sqrt{c_0}$ & $p\sqrt{c_0}$ \\ \cline{3-8}
		& & $p\sqrt{c_0}$ & $1,p\sqrt{c_0}$ & $1$ & $0$ & $\sqrt{c_0}$ & $\sqrt{c_0}$ \\ \cline{5-8}
		& & & & $p$ & $0$ & $p\sqrt{c_0}$ & $p\sqrt{c_0}$ \\ \hline
		$p$ & $1,c_0,\sqrt{p},c_0\sqrt{p}$ & $c_0$ & $1,c_0$ & $1$ & $0$ & $\sqrt{p}$ & $\sqrt{p}$ \\ \cline{5-8}
		& & & & $1$ & $1$ & $1+\sqrt{p}$ & $1$ \\ \cline{3-8}
		& & $\sqrt{p}$ & $1,\sqrt{p}$ & $1$ & $0$ & $\sqrt{p}$ & $\sqrt{p}$ \\ \cline{5-8}
		& & & & $c_0$ & $0$ & $c_0\sqrt{p}$ & $c_0\sqrt{p}$ \\ \cline{3-8}
		& & $c_0\sqrt{p}$ & $1,c_0\sqrt{p}$ & $1$ & $0$ & $\sqrt{p}$ & $\sqrt{p}$ \\ \cline{5-8}
		& & & & $c_0$ & $0$ & $c_0\sqrt{p}$ & $c_0\sqrt{p}$ \\ \hline
		$c_0p$ & $1,c_0,\sqrt{c_0p},c_0\sqrt{c_0p}$ & $c_0$ & $1,c_0$ & $1$ & $0$ & $\sqrt{c_0p}$ & $\sqrt{c_0p}$ \\ \cline{5-8}
		& & & & $1$ & $1$ & $1+\sqrt{c_0p}$ & $1$ \\ \cline{3-8}
		& & $\sqrt{c_0p}$ & $1,\sqrt{c_0p}$ & $1$ & $0$ & $\sqrt{c_0p}$ & $\sqrt{c_0p}$ \\ \cline{5-8}
		& & & & $c_0$ & $0$ & $c_0\sqrt{c_0p}$ & $c_0\sqrt{c_0p}$ \\ \cline{3-8}
		& & $c_0\sqrt{c_0p}$ & $1,c_0\sqrt{c_0p}$ & $1$ & $0$ & $\sqrt{c_0p}$ & $\sqrt{c_0p}$ \\ \cline{5-8}
		& & & & $c_0$ & $0$ & $c_0\sqrt{c_0p}$ & $c_0\sqrt{c_0p}$ \\ \hline
	\end{tabular}
	\caption{Values of $a$ and $b$ for Proposition \ref{prop:exotic} with $F=\Qp, p\equiv 1\mod 4$. The second column shows the classes of $\Qp[\alpha]$ modulo a square, the fourth shows the classes attainable as $x^2+\gamma^2$, the fifth and sixth show values of $a$ and $b$, the seventh shows the resulting $a(b+\alpha)$, and the eighth shows its class. These classes, multiplied by the ``attainable classes'', should cover the set of ``all classes''. (It is interesting that the attainable classes are always $1$ and $\gamma^2$. This may have to do with the field being non-archimedean.)}
	\label{table:normalforms1}
\end{table}

Now we make the analogous treatment with $p\equiv 3\mod 4$. The values of $\alpha$ are the same as in the previous case, but now we can take $c_0=-1$ to simplify the formulas, so we have $\alpha=\ii,\sqrt{p}$ or $\ii\sqrt{p}$.

\begin{proposition}\label{prop:squares3}
	Let $p$ be a prime number such that $p\equiv 3\mod 4$. Then the following statements hold.
	\begin{enumerate}
		\item $\Sq(\Qp[\ii]^*)=\{a+\ii b:a,b\in\Qp,\min\{\ord_p(a),\ord_p(b)\}\equiv 0\mod 2,a^2+b^2\in\Sq(\Qp^*)\}$.
		\item $\Sq(\Qp[\sqrt{p}]^*)=\{a+b\sqrt{p}:a,b\in\Qp,\ord_p(a)\le\ord_p(b),\digit_0(a)\in\Sq(\F_p^*)\}$.
		\item $\Sq(\Qp[\ii\sqrt{p}]^*)=\{a+\ii b\sqrt{p}:a,b\in\Qp,\ord_p(a)\le\ord_p(b),\ord_p(a)\equiv 0\mod 2,\digit_0(a)\in\Sq(\F_p^*)\}\cup\{a+\ii b\sqrt{p}:\ord_p(a)\le\ord_p(b),\ord_p(a)\not\equiv 0\mod 2,\digit_0(a)\notin\Sq(\F_p^*)\}$.
	\end{enumerate}
\end{proposition}

\begin{proof}
	Parts (2) and (3) have the same proof as the corresponding parts of Proposition \ref{prop:squares1}, so we focus on part (1).
	
	Suppose that $a+\ii b$ is a square. By Lemma \ref{lemma:square} $a^2+b^2=(r^2+s^2)^2$ for $r,s\in\Qp$. In particular, $a^2+b^2$ is a square. By Proposition \ref{prop:images-extra}(1), $r^2+s^2$ has even order, so $4\mid\ord(a^2+b^2)$. As $p\equiv 3\mod 4$, we cannot have a cancellation in $a^2+b^2$, so
	\[\ord(a^2+b^2)=\min\Big\{\ord(a^2),\ord(b^2)\Big\}=2\min\Big\{\ord(a),\ord(b)\Big\}.\]
	As this is a multiple of $4$, $\min\{\ord(a),\ord(b)\}$ is even.
	
	Now suppose that $a$ and $b$ satisfy the conditions. We have the first condition in Lemma \ref{lemma:square}. To check the second, let $t_1$ and $t_2$ be the two candidates for $r^2$ and suppose, without loss of generality, that $t_1=(a+\sqrt{a^2+b^2})/2$ does not cancel the leading terms. Then \[\ord(\sqrt{a^2+b^2})=\min\Big\{\ord(a),\ord(b)\Big\}\] is even. The order of $a$ is either higher than this (if $b$ has lower order) or the same, and in any case $\ord(t_1)$ is even. $t_2=-b^2/4t_1$ has also even order, and their product $-b^2/4$ is not a square (because $p\equiv 3\mod 4$). This implies that either $t_1$ or $t_2$ is a square.
\end{proof}

\begin{corollary}\label{cor:squares3}
	Let $p$ be a prime number such that $p\equiv 3\mod 4$. Let $a_0,b_0\in\Zp$ such that $a_0^2+b_0^2\equiv -1\mod p$. (This pair exists by Proposition \ref{prop:images-extra}(1).) Then the following statements hold.
	\begin{enumerate}
		\item $\Qp[\ii]^*/\Sq(\Qp[\ii]^*)=\{1,p,a_0+\ii b_0,p(a_0+\ii b_0)\}$.
		\item $\Qp[\sqrt{p}]^*/\Sq(\Qp[\sqrt{p}]^*)=\{1,-1,\sqrt{p},-\sqrt{p}\}$.
		\item $\Qp[\ii\sqrt{p}]^*/\Sq(\Qp[\ii\sqrt{p}]^*)=\{1,-1,\ii\sqrt{p},-\ii\sqrt{p}\}$.
	\end{enumerate}
\end{corollary}

\begin{proof}
	Again, parts (2) and (3) are similar to the corresponding ones in Corollary \ref{cor:squares1}, so we focus on part (1).
	
	Given a number in the form $a+\ii b$, we first ensure that $a^2+b^2$ is a square multiplying by $a_0+\ii b_0$ if needed (this changes $a^2+b^2\mod p$ to the opposite). Then we have to ensure that $\min\{\ord(a),\ord(b)\}$ is even, multiplying by $p$ if needed, and we have a square because multiplying by $p$ multiplies $a^2+b^2$ by $p^2$ and it will still be a square.
\end{proof}

Now we have determined the possible $\gamma$'s. In this case, it is not always true that $\bar{\gamma}\in\Qp[\gamma]$:
\begin{itemize}
	\item If $\gamma^2=-1$ or $p$, then $\bar{\gamma}=\gamma$.
	\item If $\gamma^2=a_0+\ii b_0$ or $p(a_0+\ii b_0)$, the product $\gamma^2\bar{\gamma}^2$ is $a_0^2+b_0^2$ or $p^2(a_0^2+b_0^2)$ respectively. But $a_0^2+b_0^2\equiv -1\mod p$ implies that $-\gamma^2\bar{\gamma}^2$ is a square in $\Qp$, so $\gamma\bar{\gamma}$ is $\ii$ times an element of $\Qp$, and in this case we also have $\bar{\gamma}\in\Qp[\gamma]$.
	\item Otherwise, $\bar{\gamma}=-\gamma$, which is a different class.
\end{itemize}
Hence, the five cases $p$, $-1$, $a_0+\ii b_0$ and $p(a_0+\ii b_0)$ give different extensions $\Qp[\gamma,\bar{\gamma}]$ and the other four cases give only two extensions, one for $\pm\sqrt{p}$ and the other for $\pm\ii\sqrt{p}$.

The next step is to determine $\bDSq(\Qp[\alpha],-\gamma^2)$ for each possible $\alpha$ and $\gamma$, and make the quotients. The computation is in Table \ref{table:normalforms3}.

\begin{table}\footnotesize
	\begin{tabular}{|c|c|c|c|c|c|c|c|}\hline
		$\alpha^2$ & all classes & $\gamma^2$ & attainable classes & $a$ & $b$ & $a(b+\alpha)$ & $[a(b+\alpha)]$ \\ \hline
		$-1$ & $1,p,a_0+\ii b_0,p(a_0+\ii b_0)$ & $p$ & $1,p$ & $1$ & $0$ & $\ii$ & $1$ \\ \cline{5-8}
		& & & & $b_0$ & $a_0/b_0$ & $a_0+\ii b_0$ & $a_0+\ii b_0$ \\ \cline{3-8}
		& & $a_0+\ii b_0$ & $1,a_0+\ii b_0$ & $1$ & $0$ & $\ii$ & $1$ \\ \cline{5-8}
		& & & & $p$ & $0$ & $\ii p$ & $p$ \\ \cline{3-8}
		& & $p(a_0+\ii b_0)$ & $1,p(a_0+\ii b_0)$ & $1$ & $0$ & $\ii$ & $1$ \\ \cline{5-8}
		& & & & $p$ & $0$ & $\ii p$ & $p$ \\ \hline
		$p$ & $1,-1,\sqrt{p},-\sqrt{p}$ & $-1$ & $1,-1$ & $1$ & $0$ & $\sqrt{p}$ & $\sqrt{p}$ \\ \cline{5-8}
		& & & & $1$ & $1$ & $1+\sqrt{p}$ & $1$ \\ \cline{3-8}
		& & $\sqrt{p}$ & $1,-\sqrt{p}$ & $1$ & $0$ & $\sqrt{p}$ & $\sqrt{p}$ \\ \cline{5-8}
		& & & & $-1$ & $0$ & $-\sqrt{p}$ & $-\sqrt{p}$ \\ \hline
		$-p$ & $1,-1,\ii\sqrt{p},-\ii\sqrt{p}$ & $-1$ & $1,-1$ & $1$ & $0$ & $\ii\sqrt{p}$ & $\ii\sqrt{p}$ \\ \cline{5-8}
		& & & & $1$ & $1$ & $1+\ii\sqrt{p}$ & $1$ \\ \cline{3-8}
		& & $\ii\sqrt{p}$ & $1,-\ii\sqrt{p}$ & $1$ & $0$ & $\ii\sqrt{p}$ & $\ii\sqrt{p}$ \\ \cline{5-8}
		& & & & $-1$ & $0$ & $-\ii\sqrt{p}$ & $-\ii\sqrt{p}$ \\ \hline
	\end{tabular}
	\caption{Values of $a$ and $b$ for Proposition \ref{prop:exotic} with $F=\Qp, p\equiv 3\mod 4$. The second column shows the classes of $\Qp[\alpha]$ modulo a square, the fourth shows the classes attainable as $x^2-\gamma^2$, the fifth and sixth show values of $a$ and $b$, the seventh shows the resulting $a(b+\alpha)$, and the eighth shows its class.}
	\label{table:normalforms3}
\end{table}

\subsection{Case $p=2$}\label{sec:dim4-2}

It only remains to make the analysis for $p=2$. This case is different from the rest because we now have seven values of $\alpha^2$ instead of three: $-1,2,-2,3,-3,6$ and $-6$.

\begin{lemma}\label{lemma:modular}
	Let $k,\ell\in\N$ with $k\ge \ell$. Let $a,b,r\in\Z_2$ such that $\ord_2(2r-a)=\ell$. Then we have that
	\[\frac{a\pm\sqrt{a^2-b^2\alpha^2}}{2}\equiv r\mod 2^k\]
	if and only if
	\[\left(\frac{b}{2}\right)^2\alpha^2\equiv r(a-r)\mod 2^{k+\ell}.\]
\end{lemma}

\begin{proof}
	The first equation is equivalent to
	$a\pm\sqrt{a^2-b^2\alpha^2}\equiv 2r\mod 2^{k+1}$, which itself is equivalent to
	\[\frac{\pm\sqrt{a^2-b^2\alpha^2}}{2^\ell}\equiv\frac{2r-a}{2^\ell}\mod 2^{k+1-\ell}.\]
	Since the right-hand side is odd, this is equivalent to all the following identities:
	\[\frac{a^2-b^2\alpha^2}{2^{2\ell}}\equiv \frac{(2r-a)^2}{2^{2\ell}}\mod 2^{k+2-\ell}\Leftrightarrow\]
	\[a^2-b^2\alpha^2\equiv (2r-a)^2\mod 2^{k+\ell+2}\Leftrightarrow\]
	\[-b^2\alpha^2\equiv 4r^2-4ra\mod 2^{k+\ell+2}\Leftrightarrow\]
	\[\left(\frac{b}{2}\right)^2\alpha^2\equiv r(a-r)\mod 2^{k+\ell},\]
	as we wanted.
\end{proof}

\begin{proposition}\label{prop:squares2}
	The following statements hold.
	\begin{enumerate}
		\item $\Sq(\Q_2[\ii]^*)=\{a+\ii b:a,b\in\Qp,\ord_2(b)-\ord_2(a)\ge 2,\ord_2(a)\equiv 0\mod 2, b/4a+\digit_1(a)+\digit_2(a)\equiv 0\mod 2\}\cup\{a+\ii b:a,b\in\Qp,\ord_2(a)-\ord_2(b)\ge 2,\ord_2(b)\equiv 1\mod 2, a/4b+\digit_1(b)+\digit_2(b)\equiv 0\mod 2\}$.
		\item $\Sq(\Q_2[\sqrt{2}]^*)=\{a+b\sqrt{2}:a,b\in\Qp,\ord_2(b)-\ord_2(a)\ge 1,\digit_2(a)=0,b/2a+\digit_1(a)\equiv 0\mod 2\}$.
		\item $\Sq(\Q_2[\ii\sqrt{2}]^*)=\{a+\ii b\sqrt{2}:a,b\in\Qp,\ord_2(b)-\ord_2(a)\ge 1,\digit_1(a)=\digit_2(a), b/2a+\ord_2(a)+\digit_1(a)\equiv 0\mod 2\}$.
		\item $\Sq(\Q_2[\sqrt{3}]^*)=\{a+b\sqrt{3}:a,b\in\Qp,\ord_2(b)-\ord_2(a)\ge 2,\ord_2(a)\equiv 0\mod 2, b/4a+\digit_2(a)\equiv 0\mod 2\}\cup\{a+b\sqrt{3}:a,b\in\Qp,\ord_2(a)-\ord_2(b)=1,\ord_2(a)\equiv 0\mod 2,\digit_1(a)+\digit_1(b)+\digit_2(b)\equiv 0\mod 2\}$.
		\item $\Sq(\Q_2[\ii\sqrt{3}]^*)=\{a+\ii b\sqrt{3}:a,b\in\Qp,\ord_2(b)-\ord_2(a)\ge 2,\ord_2(a)\equiv 0\mod 2,\digit_1(a)=0\}\cup\{a+\ii b\sqrt{3}:a,b\in\Qp,\ord_2(a)=\ord_2(b),\ord_2(a)\equiv 1\mod 2,\digit_1(a)=1,a^2+3b^2\in\Sq(\Q_2^*)\}$.
		\item $\Sq(\Q_2[\sqrt{6}]^*)=\{a+b\sqrt{6}:a,b\in\Qp,\ord_2(b)-\ord_2(a)\ge 1,\ord_2(a)+\digit_1(a)+\digit_2(a)\equiv 0\mod 2,b/2a+\digit_2(a)\equiv 0\mod 2\}$.
		\item $\Sq(\Q_2[\ii\sqrt{6}]^*)=\{a+\ii b\sqrt{6}:a,b\in\Qp,\ord_2(b)-\ord_2(a)\ge 1,\ord_2(a)+\digit_2(a)\equiv 0\mod 2,b/2a+\digit_1(a)\equiv 0\mod 2\}$.
	\end{enumerate}
\end{proposition}

\begin{proof}
	The first condition in Lemma \ref{lemma:square} implies that $a^2-\alpha^2b^2$ is a square, so it has even order and ends in $001$. This depends on $\alpha^2\in\{-1,2,-2,3,-3,6,-6\}$ as well as in the difference $\ord_2(b)-\ord_2(a)$, in the way described in Table \ref{table:image-extra} (where $c=-\alpha^2$). The valid values are as follows:
	
	\begin{enumerate}\renewcommand{\theenumi}{\roman{enumi}}
		\item $\alpha^2$ is odd and $\ord_2(b)-\ord_2(a)\ge 2$. Then $\sqrt{a^2-b^2\alpha^2}$ has the same order than $a$, and without loss of generality we suppose that $\digit_1(\sqrt{a^2-b^2\alpha^2})=\digit_1(a)$ (otherwise choose the other square root).
		
		We have that $t_1$ or $t_2$ is a square, so it has even order, and $t_1t_2=b^2\alpha^2/4$, which has even order, so both $t_1$ and $t_2$ have even order. But, as $\ord_2(\sqrt{a^2-b^2\alpha^2})=\ord_2(a)$ and $\digit_1(\sqrt{a^2-b^2\alpha^2})=\digit_1(a)$,
		\[\ord_2(2t_1)=\ord_2(a+\sqrt{a^2-b^2\alpha^2})=\ord_2(a)+1\]
		which implies $\ord_2(t_1)=\ord_2(a)$. As this is even, $a$ has even order.
		
		To simplify the computation, we divide $a$ and $b$ by some power of $4$ so that $\ord_2(a)=0$. (Obviously, dividing by $4$ does not affect being a square.) Now $\ord_2(t_1)=0$.
		
		If $t_1$ is a square, it must be $1$ modulo $8$. By Lemma \ref{lemma:modular}, using that in this case $\ord_2(2-a)=0$, this is equivalent to
		\[\left(\frac{b}{2}\right)^2\alpha^2\equiv a-1\mod 8.\]
		If $\ord_2(b)=2$, this implies $4\alpha^2\equiv a-1\mod 8$, and, using that $\alpha$ is odd, $4\equiv a-1\mod 8$, so $a\equiv 5\mod 8$. Otherwise, $\ord_2(b)\ge 3$ and $a\equiv 1\mod 8$. In any case, $t_1$ is square if and only if $\digit_1(a)=0$ and $b/4a+\digit_2(a)$ is even.
		
		If $t_2$ is a square, as $t_2=b^2\alpha^2/4t_1$, $\alpha^2/t_1$ is also a square and has order $0$, so it must be $1$ modulo $8$ and $t_1\equiv\alpha^2\mod 8$. Now $\ord_2(2\alpha^2-a)=0$ again, so this is equivalent to
		\[\left(\frac{b}{2}\right)^2\alpha^2\equiv \alpha^2(a-\alpha^2)\mod 8,\]
		that is
		\[\left(\frac{b}{2}\right)^2\equiv a-\alpha^2\mod 8.\]
		The left-hand side is equivalent to $4$ modulo $8$ if $\ord_2(b)=2$ and $0$ otherwise. If $\alpha^2=-1$, $a$ must be $3$ or $7$ respectively, so $\digit_1(a)=1$ and $b/4a+\digit_2(a)$ is odd. If $\alpha^2=3$, $a$ is $7$ or $3$ respectively, so $\digit_1(a)=1$ and $b/4a+\digit_2(a)$ is even. Finally, if $\alpha^2=-3$, $a$ is $5$ or $1$ respectively, $\digit_1(a)=0$ and $b/4a+\digit_2(a)$ is odd. Putting this together with the results for $t_1$, we obtain the first set in the cases (1), (4) and (5).
		
		\item $\alpha^2=-1$ and $\ord_2(b)-\ord_2(a)\le -2$. Now $a^2+b^2$ has the order of $b^2$, and $a\pm\sqrt{a^2+b^2}$ has the same order as $b$, so $\ord_2(t_1)=\ord_2(t_2)=\ord_2(b)-1$ and $\ord_2(b)$ is odd. By dividing $a$ and $b$ by a power of $4$, we assume that $\ord_2(b)=1$ and $\ord_2(t_1)=\ord_2(t_2)=0$.
		
		If $t_i$ is square, $t_i\equiv 1\mod 8$. Now $\ord_2(2-a)=1$ (because $a$ is multiple of $4$) and Lemma \ref{lemma:modular} implies that this is equivalent to
		\[\left(\frac{b}{2}\right)^2\equiv 1-a\mod 16.\]
		The left-hand side is $1$ if $\digit_3(b^2)$ is $0$, that is, if $\digit_1(b)+\digit_2(b)$ is even, and $9$ otherwise, so $a$ is $0$ or $8$ modulo $16$, respectively. Putting this together, we obtain the second set in case (1).
		
		\item $\alpha^2=3$ and $\ord_2(b)-\ord_2(a)=-1$. $a^2-3b^2$ has the order of $b^2$, and $a\pm\sqrt{a^2-3b^2}$ has the order of $b$, so again $\ord_2(t_1)=\ord_2(t_2)=\ord_2(b)-1$ and $\ord_2(b)$ is odd. By dividing $a$ and $b$ by a power of $4$, we assume that $\ord_2(b)=1$, $\ord_2(a)=2$ and $\ord_2(t_1)=\ord_2(t_2)=0$.
		
		If $t_i$ is square, $t_i\equiv 1\mod 8$. Now $\ord_2(2-a)=1$ again, and this is equivalent to
		\[3\left(\frac{b}{2}\right)^2\equiv a-1\mod 16.\]
		The left-hand side is $3$ if $\digit_3(b^2)$ is $0$, that is, if $\digit_1(b)+\digit_2(b)$ is even, and $11$ otherwise, so $a$ is $4$ or $12$ modulo $16$, respectively. Putting this together, we obtain the second set in case (4).
		
		\item $\alpha^2=-3$ and $\ord_2(b)=\ord_2(a)$. In this case $\ord_2(\sqrt{a^2+3b^2})=\ord_2(a)+1$ and $\ord_2(t_1)=\ord_2(t_2)=\ord_2(a)-1$, so $\ord_2(a)$ is odd. By dividing $a$ and $b$ by a power of $4$, we assume that $\ord_2(a)=\ord_2(b)=1$ and $\ord_2(t_1)=\ord_2(t_2)=0$.
		
		If $t_i$ is square, $t_i\equiv 1\mod 8$. In this case, as $\ord_2(a)=1$, $\ord_2(2-a)$ is at least $2$, and
		\[3\left(\frac{b}{2}\right)^2\equiv 1-a\mod 32.\]
		The left-hand side is $3,27,11$ or $19$ if $b$ is $\pm 2$, $\pm 6$, $\pm 10$ or $\pm 14$ modulo $32$ respectively, so $a$ is $30,6,22$ or $14$ modulo $32$. These are exactly the cases where $a^2+3b^2$ is a square and $\digit_1(a)=1$ (the opposite remainder modulo $32$ for $a$ also makes $a^2+3b^2$ a square, but has $\digit_1(a)=0$). So we obtain the second set in case (5).
		
		\item $\alpha^2$ is even and $\ord_2(b)-\ord_2(a)\ge 1$. $\sqrt{a^2-b^2\alpha^2}$ has the same order than $a$, and without loss of generality we suppose that $\digit_1(\sqrt{a^2-b^2\alpha^2})=\digit_1(a)$ (otherwise choose the other square root).
		
		We have that $t_1t_2=b^2\alpha^2/4$, which has odd order, so one of $t_1$ and $t_2$ has even order and the other has odd order. But, as $\ord_2(\sqrt{a^2-b^2\alpha^2})=\ord_2(a)$ and $\digit_1(\sqrt{a^2-b^2\alpha^2})=\digit_1(a)$,
		\[\ord_2(2t_1)=\ord_2(a+\sqrt{a^2-b^2\alpha^2})=\ord_2(a)+1\]
		which implies $\ord_2(t_1)=\ord_2(a)$. That is, if $a$ has even order $t_1$ is a square and otherwise $t_2$ is a square.
		
		To simplify the computation, we divide $a$ and $b$ by some power of $4$ so that $\ord_2(a)$ is $0$ or $1$.
		
		If $t_1$ is a square, then $\ord_2(a)=\ord_2(t_1)=0$, so $t_1\equiv 1\mod 8$, and $\ord_2(2-a)=0$, so
		\[\left(\frac{b}{2}\right)^2\alpha^2\equiv a-1\mod 8.\]
		If $\ord_2(b)=1$, this implies $\alpha^2\equiv a-1\mod 8$: for $\alpha^2=2$ or $-6$, $a\equiv 3\mod 8$, and otherwise $a\equiv 7\mod 8$. If $\ord_2(b)\ge 2$, as $\alpha$ is even, we get $a\equiv 1\mod 8$. In any case, $t_1$ is square if and only if $\digit_2(a)=0$ and $b/2a+\digit_1(a)$ is even, if $\alpha^2=2$ or $-6$, or $\digit_1(a)=\digit_2(a)$ and $b/2a+\digit_1(a)$ is even, otherwise.
		
		If $t_2$ is a square, then $\ord_2(a)=\ord_2(t_1)=1$, $t_2=b^2\alpha^2/4t_1$, $\alpha^2/t_1$ is also a square and has order $0$, so it must be $1$ modulo $8$ and $t_1\equiv\alpha^2\mod 16$. Now $\ord_2(2\alpha^2-a)=1$, so
		\[\left(\frac{b}{2}\right)^2\alpha^2\equiv \alpha^2(a-\alpha^2)\mod 32,\]
		that is,
		\[\left(\frac{b}{2}\right)^2\equiv a-\alpha^2\mod 16.\]
		The left-hand side is equivalent to $4$ modulo $16$ if $\ord_2(b)=2$ and $0$ otherwise. If $\alpha^2=2$, $a$ must be $6$ or $2$ respectively, so $\digit_2(a)=0$ and $b/2a+\digit_1(a)$ is even. If $\alpha^2=-2$, $a$ must be $2$ or $14$ respectively, so $\digit_1(a)=\digit_2(a)$ and $b/2a+\digit_1(a)$ is odd. If $\alpha^2=6$, $a$ is $10$ or $6$ respectively, so $\digit_1(a)\ne\digit_2(a)$ and $b/2a+\digit_2(a)$ is even. Finally, if $\alpha^2=-6$, $a$ is $14$ or $10$ respectively, $\digit_2(a)=1$ and $b/2a+\digit_1(a)$ is even. Putting this together with the results for $t_1$, we obtain the result for the cases (2), (3), (6) and (7). \qedhere
	\end{enumerate}
\end{proof}

Now we need the analogous result to Corollaries \ref{cor:squares1} and \ref{cor:squares3}. As it turns out, the quotient group has now $16$ elements, instead of $4$ like for the other primes, so it will be given as the list of generators. The notation $G=\langle g_1,\ldots,g_n\rangle$ means that $G$ is generated by the elements $g_1,\ldots,g_n\in G$. For example, the group $\{1,c_0,p,c_0p\}$ can be described as $\langle c_0,p\rangle$.

\begin{corollary}\label{cor:squares2}
	The following statements hold.
	\begin{enumerate}
		\item $\Q_2[\ii]^*/\Sq(\Q_2[\ii]^*)=\langle 2,3,1+\ii,1+2\ii\rangle$.
		\item $\Q_2[\sqrt{3}]^*/\Sq(\Q_2[\sqrt{3}]^*)=\langle -1,2,\sqrt{3},1+\sqrt{3}\rangle$.
		\item $\Q_2[\ii\sqrt{3}]^*/\Sq(\Q_2[\ii\sqrt{3}]^*)=\langle -1,2,\ii\sqrt{3},1+2\ii\sqrt{3}\rangle$.
		\item $\Q_2[\alpha]^*/\Sq(\Q_2[\alpha]^*)=\langle -1,3,\alpha,1+\alpha\rangle$, for $\alpha^2\in\{2,-2,6,-6\}$.
	\end{enumerate}
\end{corollary}

\begin{proof}
	In all cases, the quotient can be computed, as with other primes, by starting with an arbitrary number in $\Qp[\alpha]$ and proving that it can be multiplied by some generators to make it a square.
	
	\begin{enumerate}
		\item First we ensure that $\ord_2(a)\ne\ord_2(b)$ multiplying by $1+\ii$ if needed. Then we ensure that the orders are not consecutive, multiplying by $1+2\ii$ if they are (such operation will increment the highest order). Then we ensure that the order of $a$ is even, if $\ord_2(a)<\ord_2(b)$, and that the order of $b$ is odd, otherwise, multiplying by $2$ if needed. Finally, we ensure the condition on the digits of $a$ or $b$, multiplying by $3$ if needed (in general, multiplying a $2$-adic number $x$ by $3$ preserves $\digit_2(x)$ and inverts $\digit_1(x)$).
		\item First we ensure that $\ord_2(a)\ne\ord_2(b)$ multiplying by $1+\sqrt{3}$ if needed. Then we ensure that the difference $\ord_2(b)-\ord_2(a)$ is correct ($-1$ or at least $2$), multiplying by $\sqrt{3}$ if needed (this inverts the difference). Then we multiply by $2$ if needed so that the order of $a$ is even. Finally, we ensure the condition on digits by multiplying by $-1$: both conditions involve an odd number of digits, so they will invert on multiplication by $-1$.
		\item This will be split in two cases.
		\begin{enumerate}
			\item If $\ord_2(a)=\ord_2(b)$, we first make $\digit_2(a^2+3b^2)=0$ multiplying by $1+2\ii\sqrt{3}$:
			\[(1+2\ii\sqrt{3})(a+b\ii\sqrt{3})=a-6b+(b+2a)\sqrt{3}\]
			and
			\[(a-6b)^2+3(b+2a)^2=13(a^2+3b^2).\]
			As $13\equiv 5\mod 8$, this inverts $\digit_2(a^2+3b^2)$. Now we make $\digit_1(a^2+3b^2)=0$ multiplying by $\ii\sqrt{3}$, which will analogously multiply $a^2+3b^2$ by $3$. So now $a^2+3b^2$ is a square. Next we multiply by $2$ to make the order odd, and by $-1$ to make $\digit_1(a)=1$, all without affecting $a^2+3b^2$.
			\item If $\ord_2(a)\ne\ord_2(b)$, we first make $\ord_2(a)<\ord_2(b)$ multiplying by $\ii\sqrt{3}$ if needed, then $\ord_2(b)-\ord_2(a)\ge 2$ multiplying by $1+2\ii\sqrt{3}$, then by $2$ to make $\ord_2(a)$ even, and finally by $-1$ to make $\digit_1(a)=0$.
		\end{enumerate}
		\item First we ensure $\ord_2(b)\ge\ord_2(a)$ multiplying by $\alpha$ and then $\ord_2(b)>\ord_2(a)$ multiplying by $1+\alpha$. In each of the four cases, there are two conditions left, both related to $\digit_1(a)$ and $\digit_2(a)$. We set $\digit_2(a)$ to the required value, multiplying by $-1$ if needed, and finally $\digit_1(a)$, multiplying by $3$.\qedhere
	\end{enumerate}
\end{proof}

A depiction of the $16$ classes can be found at Figure \ref{fig:-1} for $\alpha^2=-1$, at Figure \ref{fig:2} for $\alpha^2=2$, at Figure \ref{fig:-2} for $\alpha^2=-2$, at Figure \ref{fig:3} for $\alpha^2=3$, at Figure \ref{fig:-3} for $\alpha^2=-3$, at Figure \ref{fig:6} for $\alpha^2=6$, and at Figure \ref{fig:-6} for $\alpha^2=-6$.

Now we have to compute which of the classes are ``paired'' in the sense of being the classes of $\gamma^2$ and $\bar{\gamma}^2$, so that they give the same extension $\Qp[\gamma,\bar{\gamma}]$. In general, if $\gamma=t_1+t_2\alpha$,
\[\gamma^2\bar{\gamma}^2=(t_1+t_2\alpha)(t_1-t_2\alpha)=t_1^2-t_2^2\alpha^2\]
which is always in $\Q_2$, so two paired classes differ in a factor in $\Q_2$. In the last column of Tables \ref{table:-1} to \ref{table:-6} we give the pair of each class. After identifying the paired classes, if $\alpha^2\in\{-1,-2,3,6\}$, $9$ classes remain (not counting $1$), and if $\alpha^2\in\{2,-3,-6\}$, $11$ classes remain.

The next step in order to achieve the classification is to compute the classes of \[\bDSq(\Q_2[\alpha],-\gamma^2)\] for each possible $\alpha$ and $\gamma$. For the other primes this meant three different $\alpha$'s and two or three $\gamma$'s for each one, but here we need seven $\alpha$'s and nine or eleven $\gamma$'s for each one. To simplify what would otherwise be a very long and error-prone computation, we will now use the Hilbert symbol for $\Q_2[\alpha]$.

\begin{lemma}\label{lemma:hilbert}
	The Hilbert symbol $(a,b)_F$ in any field $F$ (concretely $F=\Q_2[\alpha]$) has the following properties:
	\begin{enumerate}
		\item $(1,u)_F=(u,-u)_F=1$ for any $u$.
		\item $(u,v)_F=(v,u)_F$.
		\item $(u,v)_F=1$ if and only if $v\in\bDSq(F,-u)$.
		\item $(u,v_1v_2)_F=(u,v_1)_F(u,v_2)_F$.
	\end{enumerate}
\end{lemma}

We define a subset \[S_\alpha\subset (\Q_2[\alpha]^*/\Sq(\Q_2[\alpha]^*))^2\] for different values of $\alpha$: $S_{\ii}$ is defined in Table \ref{table:-1}, $S_{\sqrt{2}}$ in Table \ref{table:2}, $S_{\ii\sqrt{2}}$ in Table \ref{table:-2}, $S_{\sqrt{3}}$ in Table \ref{table:3}, $S_{\ii\sqrt{3}}$ in Table \ref{table:-3}, $S_{\sqrt{6}}$ in Table \ref{table:6}, and $S_{\sqrt{-6}}$ in Table \ref{table:-6}.

\begin{lemma}\label{lemma:excluded}
	Let $F=\Q_2[\alpha]$ be a degree $2$ extension of $\Q_2$. If $(u,v)_F=1$ for all $(u,v)\in S_\alpha$ and there exists $(u,v)\in(\Q_2[\alpha]^*)^2$ such that $(u,v)_F=-1$, then $(u,v)_F=-1$ for all $(u,v)\notin S_\alpha$.
\end{lemma}
\begin{proof}
	In all cases, the set $\{v:(u,v)\in S_\alpha\}$ for a fixed $u\ne 1$ is a multiplicative subgroup of the quotient $\Q_2[\alpha]^*/\Sq(\Q_2[\alpha]^*)$ with eight elements. If all them have $(u,v)_F=1$ and other $v$ has $(u,v)_F=-1$, then by multiplicativity all the other $v$ have $(u,v)_F=-1$.
\end{proof}

\begin{proposition}
	For all degree $2$ extensions $F=\Q_2[\alpha]$, $(u,v)_F=1$ if and only if $(u,v)\in S_\alpha$.
\end{proposition}

\begin{proof}
	We use Lemma \ref{lemma:excluded} for each possible $\alpha$. For some values $(u,v)\in S_\alpha$, it can be easily computed that they have $(u,v)_F=1$, and this can be deduced for the rest of $S_\alpha$ by Lemma \ref{lemma:hilbert}. Then we just need to prove that there is $(u,v)$ such that $(u,v)_F=-1$, and we are done.
	
	\begin{itemize}
		\item Case $\alpha=\ii$: since we have that $3-2=1$, $2(1+2\ii)-2(1+\ii)=2\ii$, $6(-1+3\ii)+6(1+2\ii)=30\ii$, $(2+2\ii)-2=2\ii$ and $4(1+2\ii)-3=1+8\ii$, all of which are squares, all elements of $S_\ii$ have $(u,v)_F=1$. Now we prove that $(2,1+2\ii)_F=-1$, for which we have to see that $x^2-2$ will never be $1+2\ii$ times a square, for any $x$. Suppose this happens. Then, if $x^2=a+b\ii$, the orders of $a$ and $b$ must differ in at least $2$ and the orders of $a-2$ and $b$ differ in $1$. This implies that $a-2$ and $a$ have different order. There are two possibilities:
		\begin{itemize}
			\item $\ord_2(a-2)=1$ and $\ord_2(a)>1$. Then, $b$ has order $0$ or $2$. Since $a+b\ii$ is a square, $\ord_2(a)$ is at most $0$, a contradiction.
			\item $\ord_2(a-2)>1$ and $\ord_2(a)=1$. Since $a+b\ii$ is a square, $b$ has odd order at most $-1$, but it should differ in $1$ with $\ord_2(a-2)$, also a contradiction.
		\end{itemize}
		\item Case $\alpha=\sqrt{2}$: now the pairs that add up to a square are $3-2=1$, $5-4=1$, $(1+\sqrt{2})-\sqrt{2}=1$, $(2+\sqrt{2})-(1+\sqrt{2})=1$, $4(1+\sqrt{2})-3=1+4\sqrt{2}$ and $3+2\sqrt{2}=3+2\sqrt{2}$. We prove that $(-1,\sqrt{2})_F=-1$, for which we have to see that $x^2+1$ will never be $\sqrt{2}$ times a square. Suppose it is. Let $x^2=a+b\sqrt{2}$. We have $\ord_2(a)<\ord_2(b)<\ord_2(a+1)$, which is possible only if $a$ has order $0$ and ends in $11$. But then $b$ must have order $1$ and $a$ ends in $011$, so $a+1$ has order $2$, and $a+1+b\sqrt{2}$ cannot be $\sqrt{2}$ times a square (the difference in order between $a+1$ and $b$ should be at least $2$).
		\item Case $\alpha=\ii\sqrt{2}$: now the pairs that add up to a square are $2-1=1$, $5-4=1$, $(1+\ii\sqrt{2})-\ii\sqrt{2}=1$, $3(1+\ii\sqrt{2})-3(-2+\ii\sqrt{2})=9$, $-1+2\ii\sqrt{2}=-1+2\ii\sqrt{2}$ and $2(1+\ii\sqrt{2})-3=-1+2\ii\sqrt{2}$. We prove that $(-1,1+\ii\sqrt{2})_F=-1$, for which we have to see that $x^2+1$ will never be $1+\ii\sqrt{2}$ times a square. Suppose it is. Let $x^2=a+b\ii\sqrt{2}$. We have $\ord_2(a)<\ord_2(b)=\ord_2(a+1)$, which is possible only if $\ord_2(a)=0$. Also, $b/2a+\digit_1(a)$ is odd, so that $\ord_2(a+1)=\ord_2(b)$. This makes $a+b\ii\sqrt{2}$ not a square.
		\item Case $\alpha=\sqrt{3}$: now the pairs that add up to a square are $2-1=1$, $5-4=1$, $(1+\sqrt{3})-\sqrt{3}=1$, $2(3+\sqrt{3})-2(1+\sqrt{3})=4$, $2(3+\sqrt{3})-2=4+2\sqrt{3}$ and $-1+\sqrt{3}(4+2\sqrt{3})=5+4\sqrt{3}$. We prove that $(-1,1+\sqrt{3})_F=-1$, for which we have to see that $x^2+1$ will never be $1+\sqrt{3}$ times a square. Suppose it is. Let $x^2=a+b\sqrt{3}$.
		\begin{itemize}
			\item If $\ord_2(b)-\ord_2(a)\ge 2$, $\ord_2(a)=0$ and $a$ must end in $11$ so that $\ord_2(a+1)=\ord_2(b)$. Also, $b/4a+\digit_2(a)$ is odd. This makes $a+b\sqrt{3}$ not a square.
			\item If $\ord_2(a)-\ord_2(b)=1$, $\ord_2(a)$ must be even, so it is impossible that $a+1$ and $b$ have the same order.
		\end{itemize}
		\item Case $\alpha=\ii\sqrt{3}$: now the pairs that add up to a square are $2-1=1$, $3-2=1$, $(1+2\ii\sqrt{3})-2\ii\sqrt{3}=1$, $(1+2\ii\sqrt{3})-2(-6+\ii\sqrt{3})=13$, $2+2\ii\sqrt{3}=2+2\ii\sqrt{3}$ and $2(-6+\ii\sqrt{3})+14=2+2\ii\sqrt{3}$. We prove that $(-1,\ii\sqrt{3})_F=-1$, for which we have to see that $x^2+1$ will never be $\ii\sqrt{3}$ times a square. Suppose it is. Let $x^2=a+b\ii\sqrt{3}$.
		\begin{itemize}
			\item If $\ord_2(b)-\ord_2(a)\ge 2$, $\ord_2(a)=0$ and $a$ must end in $11$ so that $\ord_2(a+1)\ge\ord_2(b)$, but then $a+b\ii\sqrt{3}$ is not a square.
			\item If $\ord_2(a)=\ord_2(b)$, they must be odd, and $\ord_2(a+1)$ cannot be greater than $\ord_2(b)$, so they must be equal, and $\ord_2(a)=\ord_2(a+1)=\ord_2(b)<0$. If the order is $-3$ or less,
			\[(a+1)^2+3b^2=a^2+3b^2+2a+1\]
			must be three times a square in $\Q_2$, but $\ord_2(a^2+3b^2)\le -4$ and $\digit_1((a+1)^2+3b^2)=\digit_1(a^2+3b^2)=0$, so it is impossible.
			
			If the order is $-1$, let $a_1=4a$, $b_1=4b$. We want $a_1+4+b_1\ii\sqrt{3}$ to be $\ii\sqrt{3}$ times a square. Let $a_2+b_2\ii\sqrt{3}$ be this square. All of $a_1$, $b_1$, $a_2$ and $b_2$ have order $1$, and
			\[a_1+4+b_1\ii\sqrt{3}=\ii\sqrt{3}(a_2+b_2\ii\sqrt{3})\]
			implies that $a_1+4=-3b_2$ and $b_1=a_2$. Hence, $\digit_1(b_1)=\digit_1(a_2)=1$. $a_1$ is $30,6,22$ or $14$ modulo $32$ exactly when $b_1$ can take those remainders, so $a_1\equiv b_1\mod 32$, and
			\[-3b_2=a_1+4\equiv b_1+4=a_2+4\mod 32\]
			But now $a_2$ is $30,6,22$ and $14$ when $b_2$ is $\pm 2,\pm 6,\pm 10$ and $\pm 14$ respectively, and the equation does not hold in any case.
		\end{itemize}
		\item Case $\alpha=\sqrt{6}$: now the pairs that add up to a square are $2-1=1$, $5-4=1$, $(1+\sqrt{6})-\sqrt{6}=1$, $3(1+\sqrt{6})-3(6+\sqrt{6})=-15$, $-1+2\sqrt{6}=-1+2\sqrt{6}$ and $2(1+\sqrt{6})-3=-1+2\sqrt{6}$. We prove that $(-1,1+\sqrt{6})_F=-1$, for which we have to see that $x^2+1$ will never be $1+\sqrt{6}$ times a square. Suppose it is. Let $x^2=a+b\sqrt{6}$. We have $\ord_2(a)<\ord_2(b)=\ord_2(a+1)$, which is possible only if $\ord_2(a)=0$. Also, $b/2a+\digit_1(a)$ is odd, so that $\ord_2(a+1)=\ord_2(b)$. This makes $a+b\sqrt{6}$ not a square.
		\item Case $\alpha=\ii\sqrt{6}$: now the pairs that add up to a square are $2-1=1$, $3-2=1$, $(1+\ii\sqrt{6})-\ii\sqrt{6}=1$, $(-6+\ii\sqrt{6})-(1+\ii\sqrt{6})=-7$, $3+2\ii\sqrt{6}=3+2\ii\sqrt{6}$ and $6(1+\ii\sqrt{6})-3=3+6\ii\sqrt{6}$. We prove that $(-1,\ii\sqrt{6})_F=-1$, for which we have to see that $x^2+1$ will never be $\ii\sqrt{6}$ times a square. Suppose it is. Let $x^2=a+b\ii\sqrt{6}$. We have $\ord_2(a)<\ord_2(b)<\ord_2(a+1)$, which is possible only if $a$ has order $0$ and ends in $11$. But then $b$ must have order $1$ and $a$ ends in $011$, so $a+1$ has order $2$, and $a+1+b\ii\sqrt{6}$ cannot be $\ii\sqrt{6}$ times a square (the difference in order between $a+1$ and $b$ should be at least $2$).\qedhere
	\end{itemize}
\end{proof}

\clearpage

\begin{table}
	\footnotesize
	\begin{tabular}{r|r|c|c|c|c|c|c|c|c|c|c|c|c|c|c|c|c|r|r|r|}
		\multicolumn{2}{l|}{} & \multicolumn{4}{l|}{$1$} & \multicolumn{4}{l|}{$1+\ii$} & \multicolumn{4}{l|}{$1+2\ii$} & \multicolumn{4}{l|}{$-1+3\ii$} & & & \\ \cline{3-18}
		\multicolumn{2}{l|}{} & 1 & 2 & 3 & 6 & 1 & 2 & 3 & 6 & 1 & 2 & 3 & 6 & 1 & 2 & 3 & 6 & $a_1$ & $b_1$ & pair \\ \hline
		1&1&$\bullet$&$\bullet$&$\bullet$&$\bullet$&$\bullet$&$\bullet$&$\bullet$&$\bullet$&$\bullet$&$\bullet$&$\bullet$&$\bullet$&$\bullet$&$\bullet$&$\bullet$&$\bullet$ &   &   & 1 \\ \cline{2-21}
		&2&$\bullet$&$\bullet$&$\bullet$&$\bullet$&$\bullet$&$\bullet$&$\bullet$&$\bullet$&&&&&&&&          & 1 & 2 & 2 \\ \cline{2-21}
		&3&$\bullet$&$\bullet$&$\bullet$&$\bullet$&&&&&$\bullet$&$\bullet$&$\bullet$&$\bullet$&&&&          & 1 & 1 & 3 \\ \cline{2-21}
		&6&$\bullet$&$\bullet$&$\bullet$&$\bullet$&&&&&&&&&$\bullet$&$\bullet$&$\bullet$&$\bullet$          & 1 & 1 & 6 \\ \hline
		$1+\ii$&1&$\bullet$&$\bullet$&&&$\bullet$&$\bullet$&&&&&$\bullet$&$\bullet$&&&$\bullet$&$\bullet$   & 3 & 0 & 2 \\ \cline{2-21}
		&2&$\bullet$&$\bullet$&&&$\bullet$&$\bullet$&&&$\bullet$&$\bullet$&&&$\bullet$&$\bullet$&&          & 3 & 0 & 1 \\ \cline{2-21}
		&3&$\bullet$&$\bullet$&&&&&$\bullet$&$\bullet$&&&$\bullet$&$\bullet$&$\bullet$&$\bullet$&&          & 3 & 0 & 6 \\ \cline{2-21}
		&6&$\bullet$&$\bullet$&&&&&$\bullet$&$\bullet$&$\bullet$&$\bullet$&&&&&$\bullet$&$\bullet$          & 3 & 0 & 3 \\ \hline
		$1+2\ii$&1&$\bullet$&&$\bullet$&&&$\bullet$&&$\bullet$&$\bullet$&&$\bullet$&&&$\bullet$&&$\bullet$  & 2 & 0 & 3 \\ \cline{2-21}
		&2&$\bullet$&&$\bullet$&&&$\bullet$&&$\bullet$&&$\bullet$&&$\bullet$&$\bullet$&&$\bullet$&          & 2 & 0 & 6 \\ \cline{2-21}
		&3&$\bullet$&&$\bullet$&&$\bullet$&&$\bullet$&&$\bullet$&&$\bullet$&&$\bullet$&&$\bullet$&          & 2 & 0 & 1 \\ \cline{2-21}
		&6&$\bullet$&&$\bullet$&&$\bullet$&&$\bullet$&&&$\bullet$&&$\bullet$&&$\bullet$&&$\bullet$          & 2 & 0 & 2 \\ \hline
		$-1+3\ii$&1&$\bullet$&&&$\bullet$&&$\bullet$&$\bullet$&&&$\bullet$&$\bullet$&&$\bullet$&&&$\bullet$ & 2 & 0 & 6 \\ \cline{2-21}
		&2&$\bullet$&&&$\bullet$&&$\bullet$&$\bullet$&&$\bullet$&&&$\bullet$&&$\bullet$&$\bullet$&          & 2 & 0 & 3 \\ \cline{2-21}
		&3&$\bullet$&&&$\bullet$&$\bullet$&&&$\bullet$&&$\bullet$&$\bullet$&&&$\bullet$&$\bullet$&          & 2 & 0 & 2 \\ \cline{2-21}
		&6&$\bullet$&&&$\bullet$&$\bullet$&&&$\bullet$&$\bullet$&&&$\bullet$&$\bullet$&&&$\bullet$          & 2 & 0 & 1 \\ \hline
	\end{tabular}
	\caption{$S_\ii$ as a subset of $(\Q_2[\ii]^*/\Sq(\Q_2[\ii]^*))^2$. For each row except the first, we need two normal forms to cover all the classes: $a=1,b=0$ and $a=a_1,b=b_1$. The last column indicates the second index of the class that pairs with each class; the first index is always the same.}
	\label{table:-1}
\end{table}

\begin{figure}
	\includegraphics{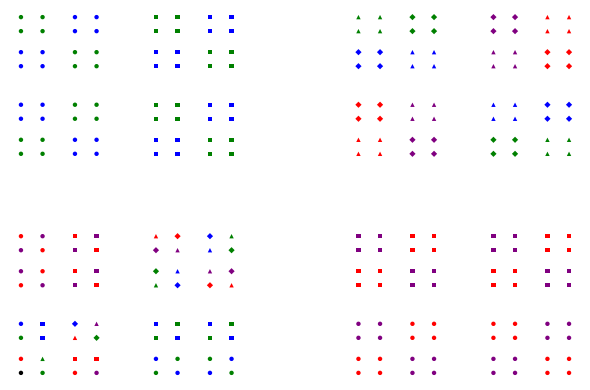}
	\caption{The $16$ classes of Table \ref{table:-1}. Each class contains the points $x+y\ii$ with a given symbol, where $x$ and $y$ are the horizontal and vertical coordinates. The circles, triangles, squares and diamonds correspond to the four values of the first index (here $1$, $1+\ii$, $1+2\ii$ and $-1+3\ii$), and the colors red, green, purple and blue to the four values of the second index (here $1$, $2$, $3$ and $6$).}
	\label{fig:-1}
\end{figure}

\clearpage

\begin{table}
	\footnotesize
	\begin{tabular}{r|r|c|c|c|c|c|c|c|c|c|c|c|c|c|c|c|c|r|r|r|}
		\multicolumn{2}{l|}{} & \multicolumn{4}{l|}{$1$} & \multicolumn{4}{l|}{$\sqrt{2}$} & \multicolumn{4}{l|}{$1+\sqrt{2}$} & \multicolumn{4}{l|}{$2+\sqrt{2}$} & & & \\ \cline{3-18}
		\multicolumn{2}{l|}{} & 1 & $-1$ & 3 & $-3$ & 1 & $-1$ & 3 & $-3$ & 1 & $-1$ & 3 & $-3$ & 1 & $-1$ & 3 & $-3$ & $a_1$ & $b_1$ & pair \\ \hline
		$1$          & $1$  &$\bullet$&$\bullet$&$\bullet$&$\bullet$&$\bullet$&$\bullet$&$\bullet$&$\bullet$&$\bullet$&$\bullet$&$\bullet$&$\bullet$&$\bullet$&$\bullet$&$\bullet$&$\bullet$ & & & 1 \\ \cline{2-21}
		& $-1$ &$\bullet$&$\bullet$&$\bullet$&$\bullet$& & & & & & & & &$\bullet$&$\bullet$&$\bullet$&$\bullet$ & 1 & 2 & $-1$ \\ \cline{2-21}
		& $3$  &$\bullet$&$\bullet$&$\bullet$&$\bullet$&$\bullet$&$\bullet$&$\bullet$&$\bullet$& & & & & & & &  & 1 & 1 & 3 \\ \cline{2-21}
		& $-3$ &$\bullet$&$\bullet$&$\bullet$&$\bullet$& & & & &$\bullet$&$\bullet$&$\bullet$&$\bullet$& & & &  & 1 & 1 & $-3$ \\ \hline
		$\sqrt{2}$   & $1$  &$\bullet$& &$\bullet$& & &$\bullet$& &$\bullet$& &$\bullet$& &$\bullet$&$\bullet$& &$\bullet$&  & $-1$ & 0 & $-1$ \\ \cline{2-21}
		& $-1$ &$\bullet$& &$\bullet$& &$\bullet$& &$\bullet$& &$\bullet$& &$\bullet$& &$\bullet$& &$\bullet$&  & $-1$ & 0 & 1 \\ \cline{2-21}
		& $3$  &$\bullet$& &$\bullet$& & &$\bullet$& &$\bullet$&$\bullet$& &$\bullet$& & &$\bullet$& &$\bullet$ & $-1$ & 0 & $-3$ \\ \cline{2-21}
		& $-3$ &$\bullet$& &$\bullet$& &$\bullet$& &$\bullet$& & &$\bullet$& &$\bullet$& &$\bullet$& &$\bullet$ & $-1$ & 0 & 3 \\ \hline
		$1+\sqrt{2}$ & $1$  &$\bullet$& & &$\bullet$& &$\bullet$&$\bullet$& & &$\bullet$&$\bullet$& &$\bullet$& & &$\bullet$ & $-1$ & 0 & $-1$ \\ \cline{2-21}
		& $-1$ &$\bullet$& & &$\bullet$&$\bullet$& & &$\bullet$&$\bullet$& & &$\bullet$&$\bullet$& & &$\bullet$ & $-1$ & 0 & 1 \\ \cline{2-21}
		& $3$  &$\bullet$& & &$\bullet$& &$\bullet$&$\bullet$& &$\bullet$& & &$\bullet$& &$\bullet$&$\bullet$&  & $-1$ & 0 & $-3$ \\ \cline{2-21}
		& $-3$ &$\bullet$& & &$\bullet$&$\bullet$& & &$\bullet$& &$\bullet$&$\bullet$& & &$\bullet$&$\bullet$&  & $-1$ & 0 & 3 \\ \hline
		$2+\sqrt{2}$ & $1$  &$\bullet$&$\bullet$& & &$\bullet$&$\bullet$& & &$\bullet$&$\bullet$& & &$\bullet$&$\bullet$& &  & 3 & 0 & 1 \\ \cline{2-21}
		& $-1$ &$\bullet$&$\bullet$& & & & &$\bullet$&$\bullet$& & &$\bullet$&$\bullet$&$\bullet$&$\bullet$& &  & 3 & 0 & $-1$ \\ \cline{2-21}
		& $3$  &$\bullet$&$\bullet$& & &$\bullet$&$\bullet$& & & & &$\bullet$&$\bullet$& & &$\bullet$&$\bullet$ & 3 & 0 & 3 \\ \cline{2-21}
		& $-3$ &$\bullet$&$\bullet$& & & & &$\bullet$&$\bullet$&$\bullet$&$\bullet$& & & & &$\bullet$&$\bullet$ & 3 & 0 & $-3$ \\ \hline
	\end{tabular}
	\caption{$S_{\sqrt{2}}$ as a subset of $(\Q_2[\sqrt{2}]^*/\Sq(\Q_2[\sqrt{2}]^*))^2$.}
	\label{table:2}
\end{table}

\begin{figure}
	\includegraphics{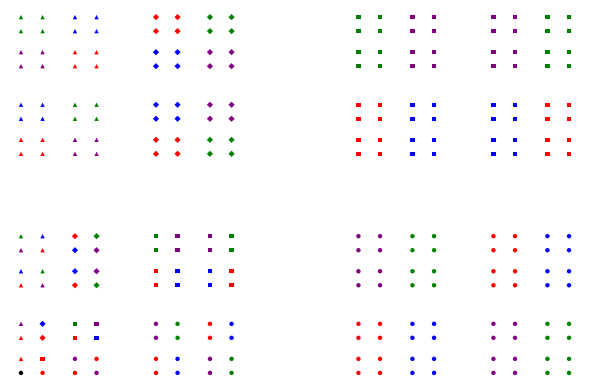}
	\caption{The $16$ classes of Table \ref{table:2}.}
	\label{fig:2}
\end{figure}

\clearpage

\begin{table}
	\footnotesize
	\begin{tabular}{r|r|c|c|c|c|c|c|c|c|c|c|c|c|c|c|c|c|r|r|r|}
		\multicolumn{2}{l|}{} & \multicolumn{4}{l|}{$1$} & \multicolumn{4}{l|}{$\ii\sqrt{2}$} & \multicolumn{4}{l|}{$1+\ii\sqrt{2}$} & \multicolumn{4}{l|}{$-2+\ii\sqrt{2}$} & & & \\ \cline{3-18}
		\multicolumn{2}{l|}{} & 1 & $-1$ & 3 & $-3$ & 1 & $-1$ & 3 & $-3$ & 1 & $-1$ & 3 & $-3$ & 1 & $-1$ & 3 & $-3$ & $a_1$ & $b_1$ & pair \\ \hline
		$1$              & $1$  &$\bullet$&$\bullet$&$\bullet$&$\bullet$&$\bullet$&$\bullet$&$\bullet$&$\bullet$&$\bullet$&$\bullet$&$\bullet$&$\bullet$&$\bullet$&$\bullet$&$\bullet$&$\bullet$ & & & 1 \\ \cline{2-21}
		& $-1$ &$\bullet$&$\bullet$&$\bullet$&$\bullet$&$\bullet$&$\bullet$&$\bullet$&$\bullet$& & & & & & & &  & 1 & 1 & $-1$ \\ \cline{2-21}
		& $3$  &$\bullet$&$\bullet$&$\bullet$&$\bullet$& & & & & & & & &$\bullet$&$\bullet$&$\bullet$&$\bullet$ & 1 & $-2$ & 3 \\ \cline{2-21}
		& $-3$ &$\bullet$&$\bullet$&$\bullet$&$\bullet$& & & & &$\bullet$&$\bullet$&$\bullet$&$\bullet$& & & &  & 1 & 1 & $-3$ \\ \hline
		$\ii\sqrt{2}$    & $1$  &$\bullet$&$\bullet$& & &$\bullet$&$\bullet$& & & & &$\bullet$&$\bullet$& & &$\bullet$&$\bullet$ & $3$ & 0 & $-1$ \\ \cline{2-21}
		& $-1$ &$\bullet$&$\bullet$& & &$\bullet$&$\bullet$& & &$\bullet$&$\bullet$& & &$\bullet$&$\bullet$& &  & $3$ & 0 & 1 \\ \cline{2-21}
		& $3$  &$\bullet$&$\bullet$& & & & &$\bullet$&$\bullet$&$\bullet$&$\bullet$& & & & &$\bullet$&$\bullet$ & $3$ & 0 & $-3$ \\ \cline{2-21}
		& $-3$ &$\bullet$&$\bullet$& & & & &$\bullet$&$\bullet$& & &$\bullet$&$\bullet$&$\bullet$&$\bullet$& &  & $3$ & 0 & 3 \\ \hline
		$1+\ii\sqrt{2}$  & $1$  &$\bullet$& & &$\bullet$& &$\bullet$&$\bullet$& & &$\bullet$&$\bullet$& &$\bullet$& & &$\bullet$ & $-1$ & 0 & 3 \\ \cline{2-21}
		& $-1$ &$\bullet$& & &$\bullet$& &$\bullet$&$\bullet$& &$\bullet$& & &$\bullet$& &$\bullet$&$\bullet$&  & $-1$ & 0 & $-3$ \\ \cline{2-21}
		& $3$  &$\bullet$& & &$\bullet$&$\bullet$& & &$\bullet$&$\bullet$& & &$\bullet$&$\bullet$& & &$\bullet$ & $-1$ & 0 & 1 \\ \cline{2-21}
		& $-3$ &$\bullet$& & &$\bullet$&$\bullet$& & &$\bullet$& &$\bullet$&$\bullet$& & &$\bullet$&$\bullet$&  & $-1$ & 0 & $-1$ \\ \hline
		$-2+\ii\sqrt{2}$ & $1$  &$\bullet$& &$\bullet$& & &$\bullet$& &$\bullet$&$\bullet$& &$\bullet$& & &$\bullet$& &$\bullet$ & $-1$ & 0 & $-3$ \\ \cline{2-21}
		& $-1$ &$\bullet$& &$\bullet$& & &$\bullet$& &$\bullet$& &$\bullet$& &$\bullet$&$\bullet$& &$\bullet$&  & $-1$ & 0 & 3 \\ \cline{2-21}
		& $3$  &$\bullet$& &$\bullet$& &$\bullet$& &$\bullet$& & &$\bullet$& &$\bullet$& &$\bullet$& &$\bullet$ & $-1$ & 0 & $-1$ \\ \cline{2-21}
		& $-3$ &$\bullet$& &$\bullet$& &$\bullet$& &$\bullet$& &$\bullet$& &$\bullet$& &$\bullet$& &$\bullet$&  & $-1$ & 0 & 1 \\ \hline
	\end{tabular}
	\caption{$S_{\ii\sqrt{2}}$ as a subset of $(\Q_2[\ii\sqrt{2}]^*/\Sq(\Q_2[\ii\sqrt{2}]^*))^2$.}
	\label{table:-2}
\end{table}

\begin{figure}
	\includegraphics{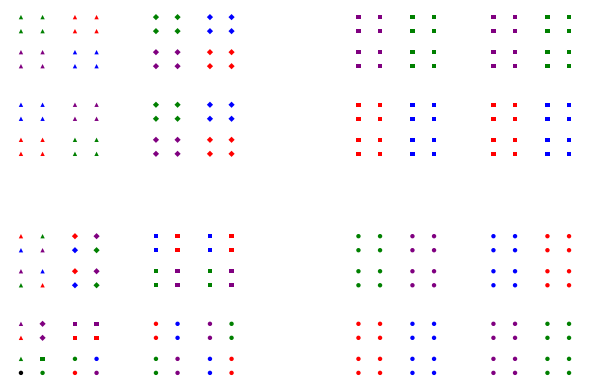}
	\caption{The $16$ classes of Table \ref{table:-2}.}
	\label{fig:-2}
\end{figure}

\clearpage

\begin{table}
	\footnotesize
	\begin{tabular}{r|r|c|c|c|c|c|c|c|c|c|c|c|c|c|c|c|c|r|r|r|}
		\multicolumn{2}{l|}{} & \multicolumn{4}{l|}{$1$} & \multicolumn{4}{l|}{$\sqrt{3}$} & \multicolumn{4}{l|}{$1+\sqrt{3}$} & \multicolumn{4}{l|}{$3+\sqrt{3}$} & & & \\ \cline{3-18}
		\multicolumn{2}{l|}{} & 1 & $-1$ & 2 & $-2$ & 1 & $-1$ & 2 & $-2$ & 1 & $-1$ & 2 & $-2$ & 1 & $-1$ & 2 & $-2$ & $a_1$ & $b_1$ & pair \\ \hline
		1            &$1$ &$\bullet$&$\bullet$&$\bullet$&$\bullet$&$\bullet$&$\bullet$&$\bullet$&$\bullet$&$\bullet$&$\bullet$&$\bullet$&$\bullet$&$\bullet$&$\bullet$&$\bullet$&$\bullet$ &   &   & 1 \\ \cline{2-21}
		&$-1$&$\bullet$&$\bullet$&$\bullet$&$\bullet$&$\bullet$&$\bullet$&$\bullet$&$\bullet$& & & & & & & &  & 1 & 1 & $-1$ \\ \cline{2-21}
		&$2$ &$\bullet$&$\bullet$&$\bullet$&$\bullet$& & & & &$\bullet$&$\bullet$&$\bullet$&$\bullet$& & & &  & 1 & 1 & 2 \\ \cline{2-21}
		&$-2$&$\bullet$&$\bullet$&$\bullet$&$\bullet$& & & & & & & & &$\bullet$&$\bullet$&$\bullet$&$\bullet$ & 1 & 3 & $-2$ \\ \hline
		$\sqrt{3}$   &$1$ &$\bullet$&$\bullet$& & &$\bullet$&$\bullet$& & & & &$\bullet$&$\bullet$& & &$\bullet$&$\bullet$ & 2 & 0 & $-1$ \\ \cline{2-21}
		&$-1$&$\bullet$&$\bullet$& & &$\bullet$&$\bullet$& & &$\bullet$&$\bullet$& & &$\bullet$&$\bullet$& &  & 2 & 0 & 1 \\ \cline{2-21}
		&$2$ &$\bullet$&$\bullet$& & & & &$\bullet$&$\bullet$& & &$\bullet$&$\bullet$&$\bullet$&$\bullet$& &  & 2 & 0 & $-2$ \\ \cline{2-21}
		&$-2$&$\bullet$&$\bullet$& & & & &$\bullet$&$\bullet$&$\bullet$&$\bullet$& & & & &$\bullet$&$\bullet$ & 2 & 0 & 2 \\ \hline
		$1+\sqrt{3}$ &$1$ &$\bullet$& &$\bullet$& & &$\bullet$& &$\bullet$& &$\bullet$& &$\bullet$&$\bullet$& &$\bullet$&  & $-1$ & 0 & $-2$ \\ \cline{2-21}
		&$-1$&$\bullet$& &$\bullet$& & &$\bullet$& &$\bullet$&$\bullet$& &$\bullet$& & &$\bullet$& &$\bullet$ & $-1$ & 0 & 2 \\ \cline{2-21}
		&$2$ &$\bullet$& &$\bullet$& &$\bullet$& &$\bullet$& & &$\bullet$& &$\bullet$& &$\bullet$& &$\bullet$ & $-1$ & 0 & $-1$ \\ \cline{2-21}
		&$-2$&$\bullet$& &$\bullet$& &$\bullet$& &$\bullet$& &$\bullet$& &$\bullet$& &$\bullet$& &$\bullet$&  & $-1$ & 0 & 1 \\ \hline
		$3+\sqrt{3}$ &$1$ &$\bullet$& & &$\bullet$& &$\bullet$&$\bullet$& &$\bullet$& & &$\bullet$& &$\bullet$&$\bullet$&  & $-1$ & 0 & 2 \\ \cline{2-21}
		&$-1$&$\bullet$& & &$\bullet$& &$\bullet$&$\bullet$& & &$\bullet$&$\bullet$& &$\bullet$& & &$\bullet$ & $-1$ & 0 & $-2$ \\ \cline{2-21}
		&$2$ &$\bullet$& & &$\bullet$&$\bullet$& & &$\bullet$&$\bullet$& & &$\bullet$&$\bullet$& & &$\bullet$ & $-1$ & 0 & 1 \\ \cline{2-21}
		&$-2$&$\bullet$& & &$\bullet$&$\bullet$& & &$\bullet$& &$\bullet$&$\bullet$& & &$\bullet$&$\bullet$&  & $-1$ & 0 & $-1$ \\ \hline
	\end{tabular}
	\caption{$S_{\sqrt{3}}$ as a subset of $(\Q_2[\sqrt{3}]^*/\Sq(\Q_2[\sqrt{3}]^*))^2$.}
	\label{table:3}
\end{table}

\begin{figure}
	\includegraphics{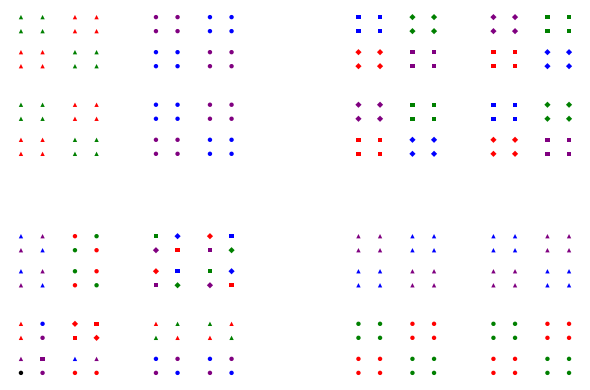}
	\caption{The $16$ classes of Table \ref{table:3}.}
	\label{fig:3}
\end{figure}

\clearpage

\begin{table}
	\footnotesize
	\begin{tabular}{r|r|c|c|c|c|c|c|c|c|c|c|c|c|c|c|c|c|r|r|r|}
		\multicolumn{2}{l|}{} & \multicolumn{4}{l|}{$1$} & \multicolumn{4}{l|}{$\ii\sqrt{3}$} & \multicolumn{4}{l|}{$1+2\ii\sqrt{3}$} & \multicolumn{4}{l|}{$-6+\ii\sqrt{3}$} & & & \\ \cline{3-18}
		\multicolumn{2}{l|}{} & 1 & $-1$ & 2 & $-2$ & 1 & $-1$ & 2 & $-2$ & 1 & $-1$ & 2 & $-2$ & 1 & $-1$ & 2 & $-2$ & $a_1$ & $b_1$ & pair \\ \hline
		1            &$1$ &$\bullet$&$\bullet$&$\bullet$&$\bullet$&$\bullet$&$\bullet$&$\bullet$&$\bullet$&$\bullet$&$\bullet$&$\bullet$&$\bullet$&$\bullet$&$\bullet$&$\bullet$&$\bullet$ &   &   & 1 \\ \cline{2-21}
		&$-1$&$\bullet$&$\bullet$&$\bullet$&$\bullet$& & & & &$\bullet$&$\bullet$&$\bullet$&$\bullet$& & & &  & 2 & $1/2$ & $-1$ \\ \cline{2-21}
		&$2$ &$\bullet$&$\bullet$&$\bullet$&$\bullet$&$\bullet$&$\bullet$&$\bullet$&$\bullet$& & & & & & & &  & 2 & $1/2$ & 2 \\ \cline{2-21}
		&$-2$&$\bullet$&$\bullet$&$\bullet$&$\bullet$& & & & & & & & &$\bullet$&$\bullet$&$\bullet$&$\bullet$ & 1 & $-6$ & $-2$ \\ \hline
		$\ii\sqrt{3}$   &$1$ &$\bullet$& &$\bullet$& & &$\bullet$& &$\bullet$& &$\bullet$& &$\bullet$&$\bullet$& &$\bullet$&  & $-1$ & 0 & $-1$ \\ \cline{2-21}
		&$-1$&$\bullet$& &$\bullet$& &$\bullet$& &$\bullet$& & &$\bullet$& &$\bullet$& &$\bullet$& &$\bullet$ & $-1$ & 0 & 1 \\ \cline{2-21}
		&$2$ &$\bullet$& &$\bullet$& & &$\bullet$& &$\bullet$&$\bullet$& &$\bullet$& & &$\bullet$& &$\bullet$ & $-1$ & 0 & $-2$ \\ \cline{2-21}
		&$-2$&$\bullet$& &$\bullet$& &$\bullet$& &$\bullet$& &$\bullet$& &$\bullet$& &$\bullet$& &$\bullet$&  & $-1$ & 0 & 2 \\ \hline
		$1+2\ii\sqrt{3}$ &$1$ &$\bullet$&$\bullet$& & & & &$\bullet$&$\bullet$&$\bullet$&$\bullet$& & & & &$\bullet$&$\bullet$ & $2$ & 0 & 1 \\ \cline{2-21}
		&$-1$&$\bullet$&$\bullet$& & &$\bullet$&$\bullet$& & &$\bullet$ &$\bullet$& & &$\bullet$&$\bullet$& &  & $2$ & 0 & $-1$ \\ \cline{2-21}
		&$2$ &$\bullet$&$\bullet$& & & & &$\bullet$&$\bullet$& & &$\bullet$&$\bullet$&$\bullet$&$\bullet$& &  & $2$ & 0 & 2 \\ \cline{2-21}
		&$-2$&$\bullet$&$\bullet$& & &$\bullet$&$\bullet$& & & & &$\bullet$&$\bullet$& & &$\bullet$&$\bullet$ & $2$ & 0 & $-2$ \\ \hline
		$-6+\ii\sqrt{3}$ &$1$ &$\bullet$& & &$\bullet$&$\bullet$& & &$\bullet$& &$\bullet$&$\bullet$& & &$\bullet$&$\bullet$&  & $-1$ & 0 & $-1$ \\ \cline{2-21}
		&$-1$&$\bullet$& & &$\bullet$& &$\bullet$&$\bullet$& & &$\bullet$&$\bullet$& &$\bullet$& & &$\bullet$ & $-1$ & 0 & 1 \\ \cline{2-21}
		&$2$ &$\bullet$& & &$\bullet$&$\bullet$& & &$\bullet$&$\bullet$& & &$\bullet$&$\bullet$& & &$\bullet$ & $-1$ & 0 & $-2$ \\ \cline{2-21}
		&$-2$&$\bullet$& & &$\bullet$& &$\bullet$&$\bullet$& &$\bullet$& & &$\bullet$& &$\bullet$&$\bullet$&  & $-1$ & 0 & 2 \\ \hline
	\end{tabular}
	\caption{$S_{\ii\sqrt{3}}$ as a subset of $(\Q_2[\ii\sqrt{3}]^*/\Sq(\Q_2[\ii\sqrt{3}]^*))^2$.}
	\label{table:-3}
\end{table}

\begin{figure}
	\includegraphics{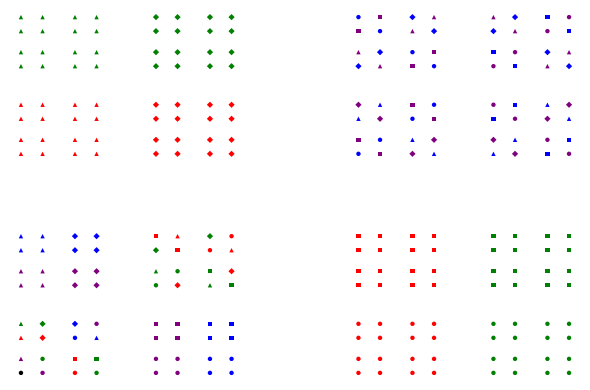}
	\caption{The $16$ classes of Table \ref{table:-3}.}
	\label{fig:-3}
\end{figure}

\clearpage

\begin{table}
	\footnotesize
	\begin{tabular}{r|r|c|c|c|c|c|c|c|c|c|c|c|c|c|c|c|c|r|r|r|}
		\multicolumn{2}{l|}{} & \multicolumn{4}{l|}{$1$} & \multicolumn{4}{l|}{$\sqrt{6}$} & \multicolumn{4}{l|}{$1+\sqrt{6}$} & \multicolumn{4}{l|}{$6+\sqrt{6}$} & & & \\ \cline{3-18}
		\multicolumn{2}{l|}{} & 1 & $-1$ & 3 & $-3$ & 1 & $-1$ & 3 & $-3$ & 1 & $-1$ & 3 & $-3$ & 1 & $-1$ & 3 & $-3$ & $a_1$ & $b_1$ & pair \\ \hline
		$1$              & $1$  &$\bullet$&$\bullet$&$\bullet$&$\bullet$&$\bullet$&$\bullet$&$\bullet$&$\bullet$&$\bullet$&$\bullet$&$\bullet$&$\bullet$&$\bullet$&$\bullet$&$\bullet$&$\bullet$ & & & 1 \\ \cline{2-21}
		& $-1$ &$\bullet$&$\bullet$&$\bullet$&$\bullet$&$\bullet$&$\bullet$&$\bullet$&$\bullet$& & & & & & & &  & 1 & 1 & $-1$ \\ \cline{2-21}
		& $3$  &$\bullet$&$\bullet$&$\bullet$&$\bullet$& & & & & & & & &$\bullet$&$\bullet$&$\bullet$&$\bullet$ & 1 & $6$ & 3 \\ \cline{2-21}
		& $-3$ &$\bullet$&$\bullet$&$\bullet$&$\bullet$& & & & &$\bullet$&$\bullet$&$\bullet$&$\bullet$& & & &  & 1 & 1 & $-3$ \\ \hline
		$\sqrt{6}$    & $1$  &$\bullet$&$\bullet$& & &$\bullet$&$\bullet$& & & & &$\bullet$&$\bullet$& & &$\bullet$&$\bullet$ & $3$ & 0 & $-1$ \\ \cline{2-21}
		& $-1$ &$\bullet$&$\bullet$& & &$\bullet$&$\bullet$& & &$\bullet$&$\bullet$& & &$\bullet$&$\bullet$& &  & $3$ & 0 & 1 \\ \cline{2-21}
		& $3$  &$\bullet$&$\bullet$& & & & &$\bullet$&$\bullet$&$\bullet$&$\bullet$& & & & &$\bullet$&$\bullet$ & $3$ & 0 & $-3$ \\ \cline{2-21}
		& $-3$ &$\bullet$&$\bullet$& & & & &$\bullet$&$\bullet$& & &$\bullet$&$\bullet$&$\bullet$&$\bullet$& &  & $3$ & 0 & 3 \\ \hline
		$1+\sqrt{6}$  & $1$  &$\bullet$& & &$\bullet$& &$\bullet$&$\bullet$& & &$\bullet$&$\bullet$& &$\bullet$& & &$\bullet$ & $-1$ & 0 & 3 \\ \cline{2-21}
		& $-1$ &$\bullet$& & &$\bullet$& &$\bullet$&$\bullet$& &$\bullet$& & &$\bullet$& &$\bullet$&$\bullet$&  & $-1$ & 0 & $-3$ \\ \cline{2-21}
		& $3$  &$\bullet$& & &$\bullet$&$\bullet$& & &$\bullet$&$\bullet$& & &$\bullet$&$\bullet$& & &$\bullet$ & $-1$ & 0 & 1 \\ \cline{2-21}
		& $-3$ &$\bullet$& & &$\bullet$&$\bullet$& & &$\bullet$& &$\bullet$&$\bullet$& & &$\bullet$&$\bullet$&  & $-1$ & 0 & $-1$ \\ \hline
		$6+\sqrt{6}$ & $1$  &$\bullet$& &$\bullet$& & &$\bullet$& &$\bullet$&$\bullet$& &$\bullet$& & &$\bullet$& &$\bullet$ & $-1$ & 0 & $-3$ \\ \cline{2-21}
		& $-1$ &$\bullet$& &$\bullet$& & &$\bullet$& &$\bullet$& &$\bullet$& &$\bullet$&$\bullet$& &$\bullet$&  & $-1$ & 0 & 3 \\ \cline{2-21}
		& $3$  &$\bullet$& &$\bullet$& &$\bullet$& &$\bullet$& & &$\bullet$& &$\bullet$& &$\bullet$& &$\bullet$ & $-1$ & 0 & $-1$ \\ \cline{2-21}
		& $-3$ &$\bullet$& &$\bullet$& &$\bullet$& &$\bullet$& &$\bullet$& &$\bullet$& &$\bullet$& &$\bullet$&  & $-1$ & 0 & 1 \\ \hline
	\end{tabular}
	\caption{$S_{\sqrt{6}}$ as a subset of $(\Q_2[\sqrt{6}]^*/\Sq(\Q_2[\sqrt{6}]^*))^2$.}
	\label{table:6}
\end{table}

\begin{figure}
	\includegraphics{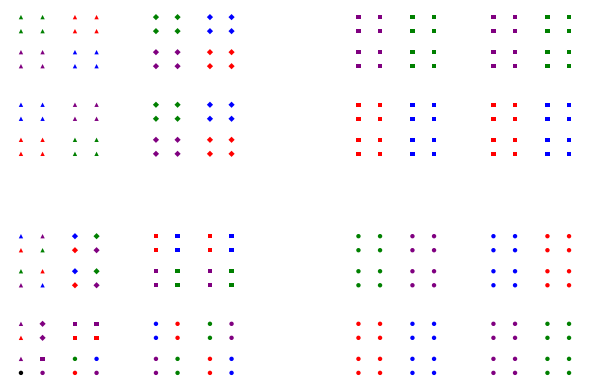}
	\caption{The $16$ classes of Table \ref{table:6}.}
	\label{fig:6}
\end{figure}

\clearpage

\begin{table}
	\footnotesize
	\begin{tabular}{r|r|c|c|c|c|c|c|c|c|c|c|c|c|c|c|c|c|r|r|r|}
		\multicolumn{2}{l|}{} & \multicolumn{4}{l|}{$1$} & \multicolumn{4}{l|}{$\ii\sqrt{6}$} & \multicolumn{4}{l|}{$1+\ii\sqrt{6}$} & \multicolumn{4}{l|}{$-6+\ii\sqrt{6}$} & & & \\ \cline{3-18}
		\multicolumn{2}{l|}{} & 1 & $-1$ & 3 & $-3$ & 1 & $-1$ & 3 & $-3$ & 1 & $-1$ & 3 & $-3$ & 1 & $-1$ & 3 & $-3$ & $a_1$ & $b_1$ & pair \\ \hline
		$1$          & $1$  &$\bullet$&$\bullet$&$\bullet$&$\bullet$&$\bullet$&$\bullet$&$\bullet$&$\bullet$&$\bullet$&$\bullet$&$\bullet$&$\bullet$&$\bullet$&$\bullet$&$\bullet$&$\bullet$ & & & 1 \\ \cline{2-21}
		& $-1$ &$\bullet$&$\bullet$&$\bullet$&$\bullet$& & & & & & & & &$\bullet$&$\bullet$&$\bullet$&$\bullet$ & 1 & $-6$ & $-1$ \\ \cline{2-21}
		& $3$  &$\bullet$&$\bullet$&$\bullet$&$\bullet$&$\bullet$&$\bullet$&$\bullet$&$\bullet$& & & & & & & &  & 1 & 1 & 3 \\ \cline{2-21}
		& $-3$ &$\bullet$&$\bullet$&$\bullet$&$\bullet$& & & & &$\bullet$&$\bullet$&$\bullet$&$\bullet$& & & &  & 1 & 1 & $-3$ \\ \hline
		$\ii\sqrt{6}$   & $1$  &$\bullet$& &$\bullet$& & &$\bullet$& &$\bullet$& &$\bullet$& &$\bullet$&$\bullet$& &$\bullet$&  & $-1$ & 0 & $-1$ \\ \cline{2-21}
		& $-1$ &$\bullet$& &$\bullet$& &$\bullet$& &$\bullet$& &$\bullet$& &$\bullet$& &$\bullet$& &$\bullet$&  & $-1$ & 0 & 1 \\ \cline{2-21}
		& $3$  &$\bullet$& &$\bullet$& & &$\bullet$& &$\bullet$&$\bullet$& &$\bullet$& & &$\bullet$& &$\bullet$ & $-1$ & 0 & $-3$ \\ \cline{2-21}
		& $-3$ &$\bullet$& &$\bullet$& &$\bullet$& &$\bullet$& & &$\bullet$& &$\bullet$& &$\bullet$& &$\bullet$ & $-1$ & 0 & 3 \\ \hline
		$1+\ii\sqrt{6}$ & $1$  &$\bullet$& & &$\bullet$& &$\bullet$&$\bullet$& & &$\bullet$&$\bullet$& &$\bullet$& & &$\bullet$ & $-1$ & 0 & $-1$ \\ \cline{2-21}
		& $-1$ &$\bullet$& & &$\bullet$&$\bullet$& & &$\bullet$&$\bullet$& & &$\bullet$&$\bullet$& & &$\bullet$ & $-1$ & 0 & 1 \\ \cline{2-21}
		& $3$  &$\bullet$& & &$\bullet$& &$\bullet$&$\bullet$& &$\bullet$& & &$\bullet$& &$\bullet$&$\bullet$&  & $-1$ & 0 & $-3$ \\ \cline{2-21}
		& $-3$ &$\bullet$& & &$\bullet$&$\bullet$& & &$\bullet$& &$\bullet$&$\bullet$& & &$\bullet$&$\bullet$&  & $-1$ & 0 & 3 \\ \hline
		$-6+\ii\sqrt{6}$ & $1$  &$\bullet$&$\bullet$& & &$\bullet$&$\bullet$& & &$\bullet$&$\bullet$& & &$\bullet$&$\bullet$& &  & 3 & 0 & 1 \\ \cline{2-21}
		& $-1$ &$\bullet$&$\bullet$& & & & &$\bullet$&$\bullet$& & &$\bullet$&$\bullet$&$\bullet$&$\bullet$& &  & 3 & 0 & $-1$ \\ \cline{2-21}
		& $3$  &$\bullet$&$\bullet$& & &$\bullet$&$\bullet$& & & & &$\bullet$&$\bullet$& & &$\bullet$&$\bullet$ & 3 & 0 & 3 \\ \cline{2-21}
		& $-3$ &$\bullet$&$\bullet$& & & & &$\bullet$&$\bullet$&$\bullet$&$\bullet$& & & & &$\bullet$&$\bullet$ & 3 & 0 & $-3$ \\ \hline
	\end{tabular}
	\caption{$S_{\ii\sqrt{6}}$ as a subset of $(\Q_2[\ii\sqrt{6}]^*/\Sq(\Q_2[\ii\sqrt{6}]^*))^2$.}
	\label{table:-6}
\end{table}

\begin{figure}
	\includegraphics{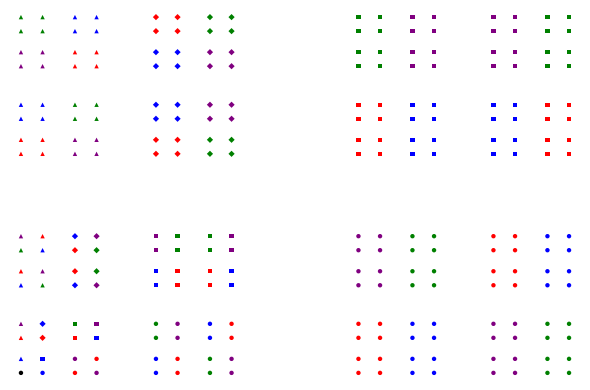}
	\caption{The $16$ classes of Table \ref{table:-6}.}
	\label{fig:-6}
\end{figure}

\clearpage

After finding the values of the Hilbert symbol for each $\alpha$, we have $\bDSq(\Q_2[\alpha],-\gamma^2)$: it is the row indexed by $\gamma^2$ of the corresponding table. It always contains eight classes, so we need two pairs $(a,b)$ to cover all classes: we can always take one of them $(1,0)$, and the other is $(a_1,b_1)$ such that the class of $a_1(b_1+\alpha)$ is marked (in the row of $\gamma^2$) if and only if that of $\alpha$ is unmarked. A possibility is included at the right of the corresponding row.

Now we have all the necessary results to prove \cite[Theorems F, G and H]{CrePel-integrable}. We recall here the statements.

\begin{theorem}[{\cite[Theorem F]{CrePel-integrable}}]\label{thm:williamson4}
	\letpprime. Let $\Omega_0$ be the matrix of the standard symplectic form on $(\Qp)^4$. Let $X_p,Y_p$ be the non-residue sets in Definition \ref{def:sets}. Let $M\in\M_4(\Qp)$ be a symmetric matrix such that all the eigenvalues of $\Omega_0^{-1}M$ are distinct. Then there exists a symplectic matrix $S\in\M_4(\Qp)$ and $r,s\in\Qp$ such that one of the following three possibilities holds:
	\begin{enumerate}
		\item There exist $c_1,c_2\in X_p$ such that
		\[S^TMS=\begin{pmatrix}
			r & 0 & 0 & 0 \\
			0 & c_1r & 0 & 0 \\
			0 & 0 & s & 0 \\
			0 & 0 & 0 & c_2s
		\end{pmatrix}.\]
		\item There exists $c\in Y_p$ such that
		\[S^TMS=\begin{pmatrix}
			0 & s & 0 & r \\
			s & 0 & cr & 0 \\
			0 & cr & 0 & s \\
			r & 0 & s & 0
		\end{pmatrix}.\]
		\item There exist $c,t_1$ and $t_2$ corresponding to one row of \cite[Table 1]{CrePel-integrable} such that $S^TMS$ is equal to the matrix
		\[
		\renewcommand{\arraystretch}{2}
		\begin{pmatrix}
			\dfrac{ac(r-bs)}{c-b^2} & 0 & \dfrac{sc-rb}{c-b^2} & 0 \\
			0 & \dfrac{-r(t_1+bt_2)-s(bt_1+ct_2)}{a} & 0 & -r(bt_1+ct_2)-sc(t_1+bt_2) \\
			\dfrac{sc-rb}{c-b^2} & 0 & \dfrac{r-bs}{a(c-b^2)} & 0 \\
			0 & -r(bt_1+ct_2)-sc(t_1+bt_2) & 0 & ac(-r(t_1+bt_2)-s(bt_1+ct_2))
		\end{pmatrix}\]
		where $(a,b)$ are either $(1,0)$ or $(a_1,b_1)$ of the corresponding row.
	\end{enumerate}
	Furthermore, if two matrices $S$ and $S'$ reduce $M$ to one of the normal forms in the right-hand side of the three equalities above, then the two normal forms are in the same case; if it is case (2) or (3), they coincide, and in case (1) they coincide up to exchanging the $2$ by $2$ diagonal blocks. Moreover, the family of normal forms (that is, the normal form except for $r$ and $s$) and the matrix $S$ are uniquely determined by the eigenvectors of $\Omega_0^{-1}M$.
\end{theorem}

\begin{euproof}
	\item First we prove existence. Let $\lambda,-\lambda,\mu,-\mu$ be the eigenvalues of $\Omega_0^{-1}M$. If $\lambda^2$ is in $\Qp$, $\mu^2$ is also in $\Qp$. Let $\{u_1,v_1,u_2,v_2\}$ be the associated basis. By Corollary \ref{cor:symplectomorphic}, there is a matrix $S$ with entries in $\Qp[\lambda,\mu]$ such that $S^T\Omega_0 S=\Omega_0$ and
	\[S^TMS=\begin{pmatrix}
		r & 0 & 0 & 0 \\
		0 & c_1r & 0 & 0 \\
		0 & 0 & s & 0 \\
		0 & 0 & 0 & c_2s
	\end{pmatrix}\]
	if and only if $\lambda=r\sqrt{-c_1}$ and $\mu=s\sqrt{-c_2}$. There are always $r,s\in\Qp$ and $c_1,c_2\in X_p$ such that this is possible; moreover, there may be two valid values of $c_1$ or $c_2$.
	
	Again by Corollary \ref{cor:symplectomorphic}, a matrix $S$ with this property must have the form $\Psi_1 D\Psi_2^{-1}$, where
	\begin{equation}\label{eq:psi}
		\Psi_1=\begin{pmatrix}
			u_1 & v_1 & u_2 & v_2
		\end{pmatrix},
		\Psi_2=\begin{pmatrix}
			\lambda & -\lambda & 0 & 0 \\
			r & r & 0 & 0 \\
			0 & 0 & \mu & -\mu \\
			0 & 0 & s & s
		\end{pmatrix}
	\end{equation}
	and $D$ is a diagonal matrix. Let $(d_1,d_2,d_3,d_4)$ be the diagonal of $D$. Let $S_1$ be the matrix formed by the first two columns of $S$ and $S_2$ the one formed by the last two columns. The equation $S=\Psi_1 D\Psi_2^{-1}$ can be written as
	\[S_1=\begin{pmatrix}
		u_1 & v_1
	\end{pmatrix}
	\begin{pmatrix}
		d_1 & 0 \\
		0 & d_2
	\end{pmatrix}
	\begin{pmatrix}
		\lambda & -\lambda \\
		r & r
	\end{pmatrix}^{-1},
	S_2=\begin{pmatrix}
		u_2 & v_2
	\end{pmatrix}
	\begin{pmatrix}
		d_3 & 0 \\
		0 & d_4
	\end{pmatrix}
	\begin{pmatrix}
		\mu & -\mu \\
		s & s
	\end{pmatrix}^{-1}.\]
	We want $S_1$ and $S_2$ to have entries in $\Qp$. This situation is the same as in Proposition \ref{prop:hyperbolic} or \ref{prop:elliptic}, depending on whether $\lambda$ and $\mu$ are in $\Qp$; that is, we are back in the situation of dimension $2$, but with the vectors $u_i$ and $v_i$ having four entries instead of two. The proof of Theorem \ref{thm:williamson} applies here as well, meaning that there are always $c_1$ and $c_2$ in $X_p$ for which $S$ is in $\Qp$, so this leads to case (1).
	
	Now suppose that $\lambda^2\notin\Qp$. If $\lambda\in\Qp[\lambda^2]$, we are in the situation of Proposition \ref{prop:focus}. $\lambda$ and $\mu$ are in a degree $2$ extension $\Qp[\alpha]$, and $M$ is equivalent by multiplication by a symplectic matrix to the matrix of case (2) for some $r,s\in\Qp$. The possible values of $\alpha^2$ are the classes of $\Qp^*$ modulo squares, that is, precisely the elements of $Y_p$, and we have case (2).
	
	Finally, if $\lambda\notin\Qp[\lambda^2]$, we are in the situation of Proposition \ref{prop:exotic}: we have a hierarchy of extensions \[\Qp\subsetneq\Qp[\alpha]\subsetneq\Qp[\gamma,\bar{\gamma}],\] and $M$ is equivalent by multiplication by a symplectic matrix to the matrix in case (3), for some $r,s\in\Qp$, which depends on the parameters $\alpha^2$, $t_1$, $t_2$, $a$ and $b$. The only ones that are not fixed by the extension are $a$ and $b$: a choice of them is valid if and only if
	\[\frac{a\alpha\gamma(b+\alpha)}{u^T\Omega_0 \hat{u}}\in\DSq(F[\alpha],-\gamma^2).\]
	The denominator is a constant, which implies that the valid values form a class modulo $\DSq(F[\alpha],-\gamma^2)$ in $F[\alpha]$. These classes are as described in Tables \ref{table:normalforms1} to \ref{table:-6}, depending on $p$ and $\alpha^2$. After substituting $c=\alpha^2$ and extracting $t_1, t_2, a_1$ and $b_1$ from these tables, we obtain case (3) of the theorem.
	
	\item Now we prove uniqueness. Let $N$ and $N'$ be the two normal forms. The case (1), (2) or (3) of the normal form is determined uniquely by the eigenvalues of $A=\Omega_0^{-1}M$, so both $N$ and $N'$ are in the same case. Now we split between the three cases.
	\begin{itemize}
		\item In case (1), there are two eigenvalues of $A$ composing the first block of the normal form and two eigenvalues composing the second one. Hence, by Theorem \ref{thm:williamson}, the normal form is unique up to changing the order of the blocks.
		\item In case (2), the extension which contains the eigenvalues is different for each $c$, hence $N=N'$.
		\item In case (3), analogously, the extension is different for each $c$, $t_1$ and $t_2$, so these parameters must coincide. If $a$ and $b$ do not coincide, we have a number in $\Qp[\alpha]$ which is in $\DSq(\Qp[\alpha],-\gamma^2)$ and which cannot be there by previous results (Tables \ref{table:normalforms1} to \ref{table:-6}), so $a$ and $b$ must also coincide and $N=N'$.
	\end{itemize}

	\item Finally we prove the last part of the statement. The extension of $\Qp$ which contains the eigenvectors is the same which contains the eigenvalues, and together with the exact vectors, this determines the family of normal forms. $\Psi_1$ is determined by the eigenvectors by definition, and $\Psi_2$ also is (in cases (2) and (3) this follows from Propositions \ref{prop:focus} and \ref{prop:exotic}, and in case (1) it follows from the form of $\Psi_2$ in \eqref{eq:psi}, which depends on $r$ and $s$ only by multiplicative constants, which can be absorbed in $D$). Because of the formula $S=\Psi_1D\Psi_2$ in Corollary \ref{cor:symplectomorphic}, $S$ is also determined by the eigenvectors.
\end{euproof}

\begin{definition}[Non-residue function {\cite[Definition 2.3]{CrePel-integrable}}]\label{def:h}
	\letpprime. If $p\equiv 1\mod 4$, let $c_0$ be the smallest quadratic non-residue modulo $p$. We define the \emph{non-residue function}:
	\[h_p:Y_p\to\Qp\text{ given by }\begin{cases}
		h_p(c_0)=p,h_p(p)=h_p(c_0p)=c_0 & \text{if }p\equiv 1\mod 4; \\
		h_p(-1)=p,h_p(p)=h_p(-p)=-1 & \text{if }p\equiv 3\mod 4; \\
		h_p(-1)=h_p(-2)=h_p(3)=h_p(6)=-1, & \\
		h_p(-3)=h_p(-6)=2,h_p(2)=3 & \text{if }p=2.
	\end{cases}\]
\end{definition}

\begin{theorem}[{\cite[Theorem G]{CrePel-integrable}}]\label{thm:williamson4-deg}
	\letpprime. Let $\Omega_0$ be the matrix of the standard symplectic form on $(\Qp)^4$. Let $X_p,Y_p$ be the non-residue sets in Definition \ref{def:sets}. Let $h_p:Y_p\to\Qp$ be the non-residue function in Definition \ref{def:h}. Let $M\in\M_4(\Qp)$ a symmetric matrix such that $\Omega_0^{-1}M$ has at least one multiple eigenvalue. Then there exists a symplectic matrix $S\in\M_4(\Qp)$ such that one of the following three possibilities holds:
	\begin{enumerate}
		\item There exist $r,s\in\Qp$ and $c_1,c_2\in X_p\cup\{0\}$ such that $S^TMS$ has the form in the case (1) of Theorem \ref{thm:williamson4}. Moreover, if $c_1=0$ then $r\in Y_p\cup\{1\}$, and if $c_2=0$ then $s\in Y_p\cup\{1\}$.
		\item There exists $r\in\Qp$ such that
		\[S^TMS=\begin{pmatrix}
			0 & r & 0 & 0 \\
			r & 0 & 1 & 0 \\
			0 & 1 & 0 & r \\
			0 & 0 & r & 0
		\end{pmatrix}\]
		\item There exist $r\in\Qp$, $c\in Y_p$ and $a\in \{1,h_p(c)\}$ such that
		\[S^TMS=\begin{pmatrix}
			a & 0 & 0 & r \\
			0 & 0 & cr & 0 \\
			0 & cr & a & 0 \\
			r & 0 & 0 & 0
		\end{pmatrix}.\]
		\item There exists $c\in Y_p\cup\{1\}$ such that
		\[S^TMS=\begin{pmatrix}
			c & 0 & 0 & 0 \\
			0 & 0 & 0 & -c \\
			0 & 0 & 0 & 0 \\
			0 & -c & 0 & 0
		\end{pmatrix}.\]
	\end{enumerate}
	Furthermore, if two matrices $S$ and $S'$ reduce $M$ to one of these normal forms on the right-hand side of the three equalities above, then the two normal forms are in the same case; if it is case (2), (3) or (4), they coincide completely, and in case (1) they coincide up to exchanging the $2$ by $2$ diagonal blocks.
\end{theorem}

\begin{euproof}
	\item First we prove existence. Suppose first that the eigenvalues of $A$ are $\lambda$, $\lambda$, $-\lambda$ and $-\lambda$, with $\lambda\ne 0$.
	If $A$ is diagonalizable, Lemma \ref{lemma:eig2} implies that there is a symplectic basis $\{u_1,v_1,u_2,v_2\}$ such that $Au_i=\lambda u_i$ and $Av_i=-\lambda v_i$. This means we are in the first case of Theorem \ref{thm:williamson4}, that is, case (1) of this theorem, and we can proceed from there.
	
	If $A$ is not diagonalizable, we can also apply Lemma \ref{lemma:eig2}, getting a symplectic basis $\{u_1,v_1,u_2,v_2\}$, or equivalently a symplectic matrix $\Psi_1$, such that
	\[\Psi_1^{-1}A\Psi_1=J=\begin{pmatrix}
		\lambda & 0 & 1 & 0 \\
		0 & -\lambda & 0 & 0 \\
		0 & 0 & \lambda & 0 \\
		0 & -1 & 0 & -\lambda
	\end{pmatrix}.\]
	
	If $\lambda\in\Qp$, we can rearrange the coordinates to make $J$ equal to $\Omega_0^{-1}M_2$, where $M_2$ is the matrix in case (2) with $r=\lambda$. As $\Psi_1$ is symplectic, rearranging its columns in the same way gives the $S$ we need.
	
	Otherwise, we can write $\lambda=r\alpha$, with $\alpha=\sqrt{c}$ for some $c\in Y_p$. Let $M_2$ be the matrix in case (3) and $A_2=\Omega_0^{-1}M_2$. We have $\Psi_2^{-1}A_2\Psi_2=J$, where
	\[\Psi_2=\begin{pmatrix}
		0 & \alpha z_1 & -\alpha z_1 & 0 \\
		1 & 0 & 0 & 1 \\
		0 & z_1 & z_1 & 0 \\
		\alpha & -t_1 & t_1 & -\alpha
	\end{pmatrix},\]
	where
	\[z_1=\frac{2\alpha}{a(1-\alpha^2)},t_1=\frac{1+\alpha^2}{r(1-\alpha^2)}.\]
	
	A matrix that commutes with $J$ has the form
	\[D=\begin{pmatrix}
		d_1 & 0 & d_2 & 0 \\
		0 & d_3 & 0 & 0 \\
		0 & 0 & d_1 & 0 \\
		0 & d_4 & 0 & d_3
	\end{pmatrix}.\]
	
	We apply the condition $D^T\Psi_1^T\Omega_0\Psi_1D=\Psi_2^T\Omega_0\Psi_2$ of Proposition \ref{prop:symplectomorphic}. As $\Psi_1$ is symplectic, $\Psi_1^T\Omega_0\Psi_1=\Omega_0$ and the condition becomes $d_1d_3=-2\alpha z_1$ and $d_1d_4+d_2d_3=0$.
	
	We also want that $S$ has the entries in $\Qp$. The first and fourth columns of $\Psi_1$ are the eigenvectors of $A$ with value $\lambda$ and $-\lambda$, which are conjugate up to a multiplicative constant: we call them $u$ and $k\bar{u}$. The second and third columns correspond to $v$ and $v'$ such that $Av=\lambda v+u$ and $Av'=-\lambda v'-k\bar{u}$. This implies $A\bar{v}=-\lambda\bar{v}+\bar{u}$, that is, $v'=-k\bar{v}$, and
	\begin{align*}
		S\Psi_2 & =\Psi_1 D \\
		& =\begin{pmatrix}
			u & -k\bar{v} & v & k\bar{u}
		\end{pmatrix}D \\
		& =
		\begin{pmatrix}
			d_1u & -d_3k\bar{v}+d_4k\bar{u} & d_2u+d_1v & d_3k\bar{u}
		\end{pmatrix}.
	\end{align*}
	If we call $c_i$ the $i$-th column of $S$, we have
	\[c_2+\alpha c_4=d_1u, c_2-\alpha c_4=d_3k\bar{u}\Rightarrow d_3k=\bar{d}_1;\]
	\[\alpha z_1c_1+z_1c_3-t_1c_4=-d_3k\bar{v}+d_4k\bar{u}=\bar{d}_1\bar{v}+d_4k\bar{u};\]
	\[-\alpha z_1c_1+z_1c_3+t_1c_4=d_2u+d_1v.\]
	Changing sign and conjugating
	\[\alpha z_1c_1+z_1c_3-t_1c_4=\bar{d}_1\bar{v}-\bar{d}_2\bar{u},\]
	so we have $d_4k=-\bar{d}_2$. We can take $d_2=d_4=0$, and the condition reduces to find $d_1$ such that
	\begin{equation}\label{eq:prod-proof-deg}
		d_1\bar{d}_1=\frac{-4\alpha^2k}{a(1-\alpha^2)}.
	\end{equation}
	So we need (\ref{eq:prod-proof-deg}) to be in \[\DSq(\Qp,-\alpha^2)=\DSq(\Qp,-c),\] which is possible for a value of $a$ in $\Qp^*/\DSq(\Qp,-c)$. This quotient is exactly the set called $\{1,h_p(c)\}$ in the statement.
	
	Now suppose that the eigenvalues are $\lambda,-\lambda,0,0$ for $\lambda\ne 0$. By Lemma \ref{lemma:eig2}, we can choose $u_1$ and $v_1$ as eigenvectors with values $\lambda$ and $-\lambda$ such that $u_1^T\Omega_0v_1=1$ and they are $\Omega_0$-complementary to the kernel of $A$. We then complete to a symplectic basis $\{u_1,v_1,u_2,v_2\}$, with $Au_2=Av_2=0$. At this point we are again in case (1) of Theorem \ref{thm:williamson4}, with $c_2=0$.
	
	The only case left is that all the eigenvalues of $A$ are $0$. Then Theorem \ref{thm:zeros} gives a good tuple $K$ with sum $4$ and a basis. The possible cases for $K$ are $(4),(2,2),(2,1,1)$ or $(1,1,1,1)$.
	
	\begin{itemize}
		\item If $K=(1,1,1,1)$, $M=0$ and the result follows trivially.
		\item If $K=(2,1,1)$, the basis is $\{u_{11},u_{12},u_{21},u_{31}\}$. We can multiply $u_{11}$ and $u_{12}$ by a constant so that $u_{11}^T\Omega_0u_{12}=1/r$ for $r\in Y_p\cup\{1\}$, and $u_{31}$ so that $u_{21}^T\Omega_0u_{31}=1$. Taking as $S$ the matrix with the columns $\{ru_{12},u_{11},u_{21},u_{31}\},$ we are in case (1) of Theorem \ref{thm:williamson4}, with $c_1=c_2=0$.
		\item If $K=(2,2)$, the basis is $\{u_{11},u_{12},u_{21},u_{22}\}$. We multiply $u_{11}$ and $u_{12}$ by a constant and $u_{21}$ and $u_{22}$ by another constant so that $u_{i1}^T\Omega_0u_{i2}=1/r_i$ for $r_i\in Y_p\cup\{1\}, i=1,2.$ Taking as $S$ the matrix with columns $\{r_1u_{12},u_{11},r_2u_{22},u_{21}\},$ we are in the same case as before.
		\item If $K=(4)$, the basis is $\{u_1,u_2,u_3,u_4\}$. Let $k=u_1^T\Omega_0 u_4=-u_2^T\Omega_0 u_3$. We can multiply the four vectors by a constant so that $1/k\in Y_p\cup\{1\}$. Taking $S$ with the columns $\{u_3/k,u_2,ku_1,u_4/k^2\}$, we are in case (4) of this theorem, with $c=1/k$.
	\end{itemize}
	
	\item Finally we prove uniqueness. If there are two normal forms $N$ and $N'$, by Proposition \ref{prop:symplectomorphic}, they must have the same Jordan form. The matrices in each case have different Jordan forms, so $N$ and $N'$ are in the same case.
	\begin{itemize}
		\item If it is case (1), by Theorem \ref{thm:williamson} we have $N=N'$, except perhaps for the order of the blocks.
		\item If it is case (2), the equality of eigenvalues implies $N=N'$.
		\item If it is case (3), each $c$ corresponds to a different extension, so the equality of eigenvalues implies $c=c'$ and $r=r'$. It is left to prove $a=a'$. Suppose on the contrary that $a=1$ and $a'=h_p(c)$. Applying the proof of existence, we have that
		\[\frac{-4\alpha^2k}{1-\alpha^2}\in\DSq(\Qp,-c)\text{ and }\frac{-4\alpha^2k}{h_p(c)(1-\alpha^2)}\in\DSq(\Qp,-c).\]
		As $\DSq(\Qp,-c)$ is a group, we also have $h_p(c)\in\DSq(\Qp,-c)$, which is a contradiction.
		\item If it is case (4), again by equality of eigenvalues we have $c=c'$.
	\end{itemize}
\end{euproof}

\begin{proposition}\label{prop:choice4}
	Proposition \ref{prop:choice} also holds for dimension 4, that is, the choice of $c_0$ only affects the choice of representatives of each class of matrices up to congruence via a symplectic matrix. The same happens for $a_0$ and $b_0$.
\end{proposition}

\begin{proof}
	Applying Theorems \ref{thm:williamson4} and \ref{thm:williamson4-deg} to the normal forms of one set gives for each one and only one form of the other set which is equivalent.
\end{proof}

\begin{theorem}[{\cite[Theorem H]{CrePel-integrable}}]\label{thm:num-forms1}
	\letpprime. Let $X_p,Y_p$ be the non-residue sets in Definition \ref{def:sets}. Let $h_p:Y_p\to\Qp$ be the non-residue function in Definition \ref{def:h}.
	\begin{enumerate}
		\item If $p\equiv 1\mod 4$, there are exactly $49$ infinite families of normal forms of $p$-adic $4$-by-$4$ matrices with two degrees of freedom, exactly $35$ infinite families with one degree of freedom, and exactly $20$ isolated normal forms, up to congruence via a symplectic matrix.
		\item If $p\equiv 3\mod 4$, there are exactly $32$ infinite families of normal forms of $p$-adic $4$-by-$4$ matrices with two degrees of freedom, exactly $27$ infinite families with one degree of freedom, and exactly $20$ isolated normal forms, up to congruence via a symplectic matrix.
		\item If $p=2$, there are exactly $211$ infinite families of normal forms of $p$-adic $4$-by-$4$ matrices with two degrees of freedom, exactly $103$ infinite families with one degree of freedom, and exactly $72$ isolated normal forms, up to congruence via a symplectic matrix.
	\end{enumerate}
	
	In the three cases above, the infinite families with two degrees of freedom are given by
	\[\Big\{\Big\{\begin{pmatrix}
		r & 0 & 0 & 0 \\
		0 & c_1r & 0 & 0 \\
		0 & 0 & s & 0 \\
		0 & 0 & 0 & c_2s
	\end{pmatrix}:r,s\in\Qp\Big\}:c_1,c_2\in X_p\Big\}
	\cup\Big\{\Big\{\begin{pmatrix}
		0 & s & 0 & r \\
		s & 0 & cr & 0 \\
		0 & cr & 0 & s \\
		r & 0 & s & 0
	\end{pmatrix}:r,s\in\Qp\Big\}:c\in Y_p\Big\}\]
	\[\cup\Big\{\Big\{\renewcommand{\arraystretch}{2}
	\begin{pmatrix}
		\dfrac{ac(r-bs)}{c-b^2} & 0 & \dfrac{sc-rb}{c-b^2} & 0 \\
		0 & \dfrac{-r(t_1+bt_2)-s(bt_1+ct_2)}{a} & 0 & -r(bt_1+ct_2)-sc(t_1+bt_2) \\
		\dfrac{sc-rb}{c-b^2} & 0 & \dfrac{r-bs}{a(c-b^2)} & 0 \\
		0 & -r(bt_1+ct_2)-sc(t_1+bt_2) & 0 & ac(-r(t_1+bt_2)-s(bt_1+ct_2))
	\end{pmatrix}:\]
	\[r,s\in\Qp\Big\}:(a,b)\in\Big\{(1,0),(a_1,b_1)\Big\},c,t_1,t_2,a_1,b_1\text{ in one row of \cite[Table 1]{CrePel-integrable}}\Big\},\]
	those with one degree of freedom are
	\[\Big\{\Big\{\begin{pmatrix}
		r & 0 & 0 & 0 \\
		0 & c_1r & 0 & 0 \\
		0 & 0 & s & 0 \\
		0 & 0 & 0 & 0
	\end{pmatrix}:r\in\Qp\Big\}:c_1\in X_p,s\in Y_p\cup\{1\}\Big\}
	\cup\Big\{\Big\{\begin{pmatrix}
		0 & r & 0 & 0 \\
		r & 0 & 1 & 0 \\
		0 & 1 & 0 & r \\
		0 & 0 & r & 0
	\end{pmatrix}:r\in\Qp\Big\}\Big\}\]
	\[\cup\Big\{\Big\{\begin{pmatrix}
		a & 0 & 0 & r \\
		0 & 0 & cr & 0 \\
		0 & cr & a & 0 \\
		r & 0 & 0 & 0
	\end{pmatrix}:r\in\Qp\Big\}:c\in Y_p,a\in\{1,h_p(c)\}\Big\},\]
	and the isolated forms are
	\[\Big\{\begin{pmatrix}
		r & 0 & 0 & 0 \\
		0 & 0 & 0 & 0 \\
		0 & 0 & s & 0 \\
		0 & 0 & 0 & 0
	\end{pmatrix}:r,s\in Y_p\cup\{1\}\Big\}
	\cup\Big\{\begin{pmatrix}
		c & 0 & 0 & 0 \\
		0 & 0 & 0 & -c \\
		0 & 0 & 0 & 0 \\
		0 & -c & 0 & 0
	\end{pmatrix}:c\in Y_p\cup\{1\}\Big\}.\]
	
	This is in contrast with the real case, where there are exactly $4$ infinite families with two degrees of freedom, exactly $7$ infinite families with one degree of freedom and exactly $6$ isolated normal forms. Here by ``infinite family'' we mean a family of normal forms of the form $r_1M_1+r_2M_2+\ldots+r_kM_k$, where $r_i$ are parameters in $\Qp$ and $k$ is the number of degrees of freedom, and by ``isolated'' we mean a form that is not part of any family.
\end{theorem}

\begin{proof}
	From Theorem \ref{thm:williamson4}, if $p\equiv 1\mod 4$, case (1) leads to $\binom{8}{2}=28$ normal forms (there are seven possible values for $c_1$ and $c_2$), case (2) to $3$ normal forms, and case (3) has $9$ possibilities for $c$, $t_1$ and $t_2$, each one with two possible $a$ and $b$. Hence, there is a total of
	$49$ normal forms if $p\equiv 1\mod 4$. Analogously, there is a total of
	$\binom{6}{2}+3+7\cdot 2=32$
	normal forms if $p\equiv 3\mod 4$, and a total of
	$\binom{12}{2}+7+(9+11+9+9+11+9+11)\cdot 2=211$
	normal forms if $p=2$ ($7$ possibilities for $c$, some of them with $9$ options for $t_1$ and $t_2$ and others with $11$, and $2$ for $a$ and $b$).
	
	From Theorem \ref{thm:williamson4-deg}, case (1) produces $7\cdot 4=28$ families of normal forms with one degree of freedom if $p\equiv 1\mod 4$, $5\cdot 4=20$ if $p\equiv 3\mod 4$ and $11\cdot 8=88$ if $p=2$, case (2) produces one such family and case (3) produces $6$, $6$ and $14$ families, respectively. Case (1) produces $16$, $16$ and $64$ isolated forms, and case (4) produces $4$, $4$ and $8$ such forms.
\end{proof}

\section{Matrix classification in the real case, for any dimension}\label{sec:real}
In this section we give a new proof of the most general case of the Weierstrass-Williamson classification theorem \cite[Theorems K and L]{CrePel-integrable} using the new strategy introduced in the previous sections of this paper. In the simplest case, that is, for positive definite symmetric matrices, the proof reduces only to a few lines.

\subsection{The general case}\label{sec:real-general}

\begin{definition}[{\cite[Definition B.1]{CrePel-integrable}}]\label{def:blocks}
	A \emph{diagonal block of hyperbolic type} is any matrix of the form
	\[M_\h(k,r,a)=\begin{pmatrix}
		& r &   &   &   & & \\
		r &   & 1 &   &   & & \\
		& 1 &   & r &   & & \\
		&   & r &   & \ddots & & \\
		&   &   & \ddots &   & 1 & \\
		&   &   &   & 1 &   & r \\
		&   &   &   &   & r & a
	\end{pmatrix},\]
	for some positive integer $k$, $r\in\R$ and $a\in\{-1,0,1\}$ with $a=0$ if $r\ne 0$, and which has a total of $2k$ rows. A \emph{diagonal block of elliptic type} is any matrix of the form
	\[M_\e(k,r,a)=\begin{pmatrix}
		M_{\e1}(r) & M_{\e2}(1,a) & & & \\
		M_{\e2}(1,a) & M_{\e1}(r) & M_{\e2}(2,a) & & \\
		& M_{\e2}(2,a) & \ddots & & \\
		& & & M_{\e1}(r) & M_{\e2}'(\ell,a) \\
		& & & M_{\e2}'(\ell,a)^T & M_{\e1}'(r)
	\end{pmatrix}\]
	if $k=2\ell+1$ is odd, and
	\[M_\e(k,r,a)=\begin{pmatrix}
		M_{\e1}(r) & M_{\e2}(1,a) & & & \\
		M_{\e2}(1,a) & M_{\e1}(r) & M_{\e2}(2,a) & & \\
		& M_{\e2}(2,a) & \ddots & & \\
		& & & M_{\e1}(r) & M_{\e2}(\ell-1,a) \\
		& & & M_{\e2}(\ell-1,a) & M_{\e1}(r)+M_{\e2}(\ell,a)
	\end{pmatrix}\]
	if $k=2\ell$ is even, for some positive integer $k$, $r\in\R$ and $a\in\{-1,1\}$, and which has a total of $2k$ rows. A \emph{diagonal block of focus-focus type} is any matrix of the form
	\[M_{\f\f}(k,r,s)=\begin{pmatrix}
		M_{\f\f1}(r,s) & M_{\e2}(1,1) & & & \\
		M_{\e2}(1,1) & M_{\f\f1}(r,s) & M_{\e2}(1,1) & & \\
		& M_{\e2}(1,1) & \ddots & & \\
		& & & M_{\f\f1}(r,s) & M_{\e2}(1,1) \\
		& & & M_{\e2}(1,1) & M_{\f\f1}(r,s)
	\end{pmatrix},\]
	for some positive integer $k$ and $r,s\in\R$, and which has a total of $4k$ rows. In the previous blocks the following sub-blocks are used:
	\[
	M_{\e1}(r)=\begin{pmatrix}
		0 & 0 & 0 & r \\
		0 & 0 & -r & 0 \\
		0 & -r & 0 & 0 \\
		r & 0 & 0 & 0
	\end{pmatrix},
	M_{\e1}'(r)=\begin{pmatrix}
		r & 0 \\
		0 & r
	\end{pmatrix},
	M_{\f\f1}(r,s)=\begin{pmatrix}
		0 & s & 0 & r \\
		s & 0 & -r & 0 \\
		0 & -r & 0 & s \\
		r & 0 & s & 0
	\end{pmatrix},\]
	\[M_{\e2}(j,a)=\begin{pmatrix}
		a & 0 & 0 & 0 \\
		0 & 0 & 0 & 0 \\
		0 & 0 & a & 0 \\
		0 & 0 & 0 & 0
	\end{pmatrix}\text{ if $j$ is odd},
	\begin{pmatrix}
		0 & 0 & 0 & 0 \\
		0 & a & 0 & 0 \\
		0 & 0 & 0 & 0 \\
		0 & 0 & 0 & a
	\end{pmatrix}\text{ if $j$ is even},\]
	\[M_{\e2}'(j,a)=\begin{pmatrix}
		a & 0 \\
		0 & 0 \\
		0 & a \\
		0 & 0
	\end{pmatrix}\text{ if $j$ is odd},
	\begin{pmatrix}
		0 & 0 \\
		a & 0 \\
		0 & 0 \\
		0 & a
	\end{pmatrix}\text{ if $j$ is even}.\]
\end{definition}

\begin{theorem}[{\cite[Theorem K]{CrePel-integrable}}]\label{thm:williamson-real}
	Let $n$ be a positive integer. Let $\Omega_0$ be the matrix of the standard symplectic form on $\R^{2n}$. Let $M\in\M_{2n}(\R)$ be a symmetric and invertible matrix such that $\Omega_0^{-1}M$ is diagonalizable. Then, there exists a symplectic matrix $S\in\M_{2n}(\R)$ such that $S^TMS$ is a block-diagonal matrix with blocks of hyperbolic, elliptic type or focus-focus type with $k=1$.
\end{theorem}

\begin{proof}
	Applying Lemma \ref{lemma:eig2}, we can take a symplectic basis $\{u_1,v_1,\ldots,u_n,v_n\}$ of $\R^{2n}$ such that all the vectors in the basis are eigenvectors of $A$, and $u_i$ and $v_i$ have eigenvalues with opposite sign $\lambda_i$ and $-\lambda_i$. We can sort these vectors in such a way that two $\lambda_i$'s which are conjugate appear with consecutive indices.
	
	Taking as $\Psi_1$ the matrix with these vectors as columns, the problem decomposes into finding normal forms for each block of columns associated to eigenvalues of the form $\{r,-r\}$, $\{\ii r,-\ii r\}$ or \[\{r+\ii s,-r-\ii s,r-\ii s,-r+\ii s\},\] for $r,s\in\R^*$. The first block, by Proposition \ref{prop:hyperbolic}, gives the hyperbolic block. The second block, by Proposition \ref{prop:elliptic} with $a=b=r$ or $a=b=-r$ (one of them will always work), gives the elliptic block. The third block, by Proposition \ref{prop:focus}, gives the focus-focus block.
\end{proof}

\begin{theorem}[{\cite[Theorem L]{CrePel-integrable}}]\label{thm:williamson-real2}
	Let $n$ be a positive integer and let $M\in\M_{2n}(\R)$ be a symmetric matrix. Then, there exists a symplectic matrix $S\in\M_{2n}(\R)$ such that $S^TMS$ is a block diagonal matrix with each of the diagonal blocks being of hyperbolic, elliptic or focus-focus type, as in Definition \ref{def:blocks}.
	
	Furthermore, if there are two matrices $S$ and $S'$ such that $N=S^TMS$ and $N'=S'^TMS'$ are normal forms, then $N=N'$ except by the order of the blocks.
\end{theorem}

\begin{euproof}
	\item First we prove existence. The proof starts as in Theorems \ref{thm:algclosed} and \ref{thm:williamson-real}, applying Lemma \ref{lemma:eig2}. This gives us a partial symplectic basis $\{u_1,v_1,\ldots,u_m,v_m\}$. However, unlike Theorem \ref{thm:algclosed}, these vectors only work for the hyperbolic blocks, giving $M_\h(k,r,0)$; for the rest, the blocks would not be real, so we need to recombine the vectors. We can sort the blocks of the Jordan form in such a way that the blocks of two $\lambda_i$'s which are conjugate appear with consecutive indices.
	
	For an elliptic block, $\{u_1,v_1,\ldots,u_k,v_k\}$ are the vectors corresponding to the values $\ii r$ and $-\ii r$, for $r\in\R$. We see that $M_\e(k,r,1)$ and $M_\e(k,r,-1)$ have this block as Jordan form, so we can apply Proposition \ref{prop:symplectomorphic}. The columns of $\Psi_1$ are $\{u_1,v_1,\ldots,u_k,v_k\}$, which are part of a symplectic basis, and $u_1$ and $v_k$ are eigenvectors with values $\ii r$ and $-\ii r$, so $v_k=c\bar{u}_1$ for some $c\in\C$. Using that $Au_j=\ii ru_j+u_{j-1}$ and $Av_j=-\ii rv_j-v_{j+1}$, we deduce that $v_{k+1-j}=(-1)^{j-1}c\bar{u}_j$. Concretely,
	\begin{align*}
		cu_1^T\Omega_0\bar{u}_k & =u_1^T\Omega_0v_1 \\
		& =1 \\
		& =\overline{u_k^T\Omega_0v_k} \\
		& =\overline{(-1)^{k-1}cu_k^T\Omega_0\bar{u}_1} \\
		& =(-1)^{k-1}\bar{c}\bar{u}_k^T\Omega_0u_1
	\end{align*}
	which implies $c=(-1)^k\bar{c}$, that is, $c$ is real if $k$ is even and imaginary if $k$ is odd.
	
	The columns of $\Psi_2$ have two nonzero entries, of the form \[(\ldots,\pm 1,0,\pm\ii,\ldots)\] or \[(\ldots,\pm a,0,\pm\ii a,\ldots),\] except the two central ones if $k$ is odd, which are $(\ldots,1,\ii)$ or a similar form. In any case, we have $u_j'^T\Omega_0v_j'=2a$ if $k$ is even, and $2\ii a$ if $k$ is odd.
	
	We can take as $D$ a diagonal matrix with the entries alternating between $d_1$ and $d_2$, that commutes with $J$. The condition \[D^T\Omega_0D=D^T\Psi_1^T\Omega_0\Psi_1D=\Psi_2\Omega_0\Psi_2\] implies that $d_1d_2=2a$ for $k$ even, and $2\ii a$ for $k$ odd. As $S$ must be a real matrix, in $S\Psi_2=\Psi_1D$ the first and last columns are conjugate and we get $\bar{d}_1=cd_2$, that is, $d_1\bar{d}_1=2ac$ for $k$ even and $2\ii ac$ for $k$ odd. We just need to take $a\in\{1,-1\}$ so that this is positive: note that $a$ is unique.
	
	For the focus-focus case, we have in the symplectic basis a block of vectors \[\Big\{u_1,v_1,\ldots,u_k,v_k,u_1',v_1',\ldots,u_k',v_k'\Big\},\] where $u_1$, $v_k$, $u_1'$ and $v_k'$ are eigenvectors for $\lambda,-\lambda,\bar{\lambda}$ and $-\bar{\lambda}$ respectively. If $\lambda=s+\ii r$, the matrix $M_\f(k,r,s)$ has the same Jordan form.
	
	The columns of $\Psi_1$ are the vectors $u_i$, $v_i$, $u_i'$ and $v_i'$; we have that $u_1'=c\bar{u}_1$ for some $c\in\C$, which implies $u_i'=c\bar{u}_i$ for all $i$ and, from $u_i^T\Omega_0v_i=u_i'^T\Omega_0v_i'=1$, we deduce $v_i'=\bar{v}_i/c$ for all $i$. The columns of $\Psi_2$ have now the form $(0,1,0,\ii,\ldots)$, $(\ii,0,1,0,\ldots)$, $(\ldots,0,1,0,\ii,\ldots)$, and so on, for a total of $2k$, followed by their conjugates.
	
	We can take as $D$ a diagonal matrix with the values \[d_1,d_2,\ldots,d_1,d_2,d_3,d_4,\ldots,d_3,d_4,\] which commutes with $J$. The condition \[D^T\Omega_0D=D^T\Psi_1^T\Omega_0\Psi_1D=\Psi_2\Omega_0\Psi_2\] implies that $d_1d_2=-2\ii$ and $d_3d_4=2\ii$. Using that $S$ is a real matrix, the left and right halves of $S\Psi_2=\Psi_1D$ are conjugate, which implies $\bar{d}_1=cd_3$ and $c\bar{d}_2=d_4$, so the condition $d_3d_4=2\ii$ reduces to a consequence of $d_1d_2=-2\ii$, and we can take for example $d_1=1$ and $d_2=-2\ii$.
	
	This finishes the treatment of the nonzero eigenspaces. For the other part, we can use the same treatment as in Theorem \ref{thm:algclosed}, but in the even case we cannot always make $c_i=1$; instead we make $c_i=a_i$. Now \[\Big\{u_{i1},a_iu_{i,2\ell_i},-u_{i2},a_iu_{i,2\ell_i-1},\ldots,(-1)^{\ell_i-1}u_{i\ell_i},a_iu_{i,\ell_i+1}\Big\}\] is a partial symplectic basis which gives the form $M_\h(\ell_i,0,(-1)^{\ell_i}a_i)$.
	
	\item Finally we prove uniqueness. If $N$ and $N'$ are two normal forms which are equivalent, by Proposition \ref{prop:symplectomorphic}, they have the same Jordan form. This means that the set of blocks is the same except for the values of $a_i$. But in the elliptic case we already saw that $a$ is unique. This only leaves the hyperbolic case with $r=0$.
	
	In this case, if $a$ can be $1$ and $-1$ at the same time, there is a chain \[\Big\{u_1,u_2,\ldots,u_{2k}\Big\}\] such that $Au_i=u_{i-1}$, $Au_1=0$, and $u_i\Omega_0u_{2k+1-i}=(-1)^i$, and another one \[\Big\{u_1',u_2',\ldots,u_{2k}'\Big\}\] with the same properties except that \[u_i'\Omega_0u_{2k+1-i}'=(-1)^{i+1}.\] In the space generated by these vectors there is only one vector in the kernel, so $u_1'=ku_1$ for some $k\in\R$. As $Au_i=u_{i-1}$ and $Au_i'=u_{i-1}'$, we have that $u_i'=ku_i$ for all $i$. This together implies that 
	\[(-1)^{i+1}=u_i'\Omega_0u_{2k+1-i}'=k^2u_i\Omega_0u_{2k+1-i}=k^2(-1)^i\]
	and $k^2=-1$, a contradiction.
\end{euproof}

The matrix $\Psi_1$ in the previous proof gives a complex symplectic basis in which $M$ has the block diagonal form of Theorem \ref{thm:algclosed}. The relation between this and the final matrix $S$ can be written in terms of vectors. In the hyperbolic case, they are the same matrix. In the elliptic case, we first multiply the vectors by the corresponding $d_j$: $u_j:=\sqrt{|2c|}u_j$, $v_j:=\sqrt{2/|c|}v_j$. Then, the matrix $\Psi_2$ indicates how the vectors in the final basis relate to these $u_j$ and $v_j$: each column has a $\pm 1$ entry, a $\pm\ii$ entry and the rest are $0$, so we have $\pm u_h'+\pm\ii u_\ell'=u_j$, for some indices $h$ and $\ell$. As the new vectors $u_h'$ and $u_\ell'$ must be real, this has a unique solution, and the new vectors are the real and imaginary parts of the old ones (maybe with the sign changed).

In the focus-focus case, we also start multiplying the vectors by the $d_j$: \[v_j:=-2\ii v_j, u_j':=u_j'/c=\bar{u}_j, v_j':=2\ii cv_j'=\bar{v}_j.\] Then we apply $\Psi_2^{-1}$: now each column of the left half of $\Psi_2$ has a $1$ and an $\ii$, and the right half has a $1$ and a $-\ii$ in the same positions. For every $j$ with $1\le j\le k$, the equations are $v_{2j-1}''+\ii v_{2j}''=u_j$ and $\ii u_{2j-1}''+u_{2j}''=v_j$ for the first half and their conjugates for the second half, where $u_j''$ and $v_j''$ are the new vectors. The solution consists of taking the real and imaginary parts of the new vectors.

\begin{corollary}
	Let $M\in\M_4(\R)$ be a symmetric matrix. Then there exists $r,s\in\R$, $a,b\in\{-1,1\}$, and a symplectic matrix $S\in\M_4(\R)$ such that $S^TMS$ is one of the following ten matrices:
	\[\begin{pmatrix}
		0 & r & 0 & 0 \\
		r & 0 & 0 & 0 \\
		0 & 0 & 0 & s \\
		0 & 0 & s & 0
	\end{pmatrix},
	\begin{pmatrix}
		0 & 0 & 0 & 0 \\
		0 & a & 0 & 0 \\
		0 & 0 & 0 & s \\
		0 & 0 & s & 0
	\end{pmatrix},
	\begin{pmatrix}
		0 & 0 & 0 & 0 \\
		0 & a & 0 & 0 \\
		0 & 0 & 0 & 0 \\
		0 & 0 & 0 & b
	\end{pmatrix},\]
	\[\begin{pmatrix}
		0 & r & 0 & 0 \\
		r & 0 & 1 & 0 \\
		0 & 1 & 0 & r \\
		0 & 0 & r & 0
	\end{pmatrix},
	\begin{pmatrix}
		0 & 0 & 0 & 0 \\
		0 & 0 & 1 & 0 \\
		0 & 1 & 0 & 0 \\
		0 & 0 & 0 & a
	\end{pmatrix},
	\begin{pmatrix}
		0 & r & 0 & 0 \\
		r & 0 & 0 & 0 \\
		0 & 0 & s & 0 \\
		0 & 0 & 0 & s
	\end{pmatrix},\]
	\[\begin{pmatrix}
		0 & 0 & 0 & 0 \\
		0 & a & 0 & 0 \\
		0 & 0 & s & 0 \\
		0 & 0 & 0 & s
	\end{pmatrix},
	\begin{pmatrix}
		r & 0 & 0 & 0 \\
		0 & r & 0 & 0 \\
		0 & 0 & s & 0 \\
		0 & 0 & 0 & s
	\end{pmatrix},
	\begin{pmatrix}
		a & 0 & 0 & r \\
		0 & 0 & -r & 0 \\
		0 & -r & a & 0 \\
		r & 0 & 0 & 0
	\end{pmatrix},
	\begin{pmatrix}
		0 & s & 0 & r \\
		s & 0 & -r & 0 \\
		0 & -r & 0 & s \\
		r & 0 & s & 0
	\end{pmatrix},\]
	Furthermore, if there are two matrices of this form equivalent to $M$, they are both in the first or in the eighth case swapping $r$ and $s$, or in the third case swapping $a$ and $b$.
\end{corollary}

\begin{proof}
	They are in this order: two hyperbolic blocks with $k=1$ and $a=0$, one with $a=0$ and one with $a\ne 0$, two with $a\ne 0$, one hyperbolic block with $k=2$ and $a=0$, the same with $a\ne 0$, one hyperbolic with $a=0$ and one elliptic, the same with $a\ne 0$, two elliptic blocks with $k=1$, one elliptic block with $k=2$, and one focus-focus block.
\end{proof}

These are the same (up to symplectic transformations) that Williamson gives in his paper \cite[page 24]{Williamson} and which we give in \cite[Section 2]{CrePel-integrable}.

\subsection{Example of application of our method for matrices of arbitrary order}\label{sec:real-example}
Consider the matrix
\[M=\begin{pmatrix}
	1 & & & & & & &\\
	& & 1 & & & & & \\
	& 1 & & & & & & \\
	& & & & 1 & & & \\
	& & & 1 & & \ddots & & \\
	& & & & \ddots & & 1 & \\
	& & & & & 1 & & \\
	& & & & & & & 1
\end{pmatrix}\]
We have that
\[\Omega_0^{-1}M=\begin{pmatrix}
	& & -1 & & & & & \\
	1 & & & & & & &\\
	& & & & -1 & & & \\
	& 1 & & & & \ddots & & \\
	& & & & & & -1 & \\
	& & & 1 & & & & \\
	& & & & \ddots & & & -1 \\
	& & & & & 1 & &
\end{pmatrix}\]
whose characteristic polynomial is $\lambda^{2n}+(-1)^n$, and the eigenvalues are $\lambda=\e^{\pi\ii k/n}$ for $0\le k\le 2n-1$, if $n$ is odd, and $\lambda=\e^{\pi\ii (2k+1)/2n}$ for $0\le k\le 2n-1$, if $n$ is even.

If $n=2m+1$ is odd, the normal form of this matrix contains a hyperbolic block, which corresponds to the eigenvalues $\{1,-1\}$, and $m$ focus-focus blocks, for the eigenvalues \[\Big\{\e^{\pi\ii k/n},\e^{\pi\ii (n-k)/n},\e^{\pi\ii (n+k)/n},\e^{\pi\ii (2n-k)/n}\Big\},\] for $1\le k\le m$:
\[S^TMS=\begin{pmatrix}
	0 & 1 & & & \\
	1 & 0 & & & \\
	& & N(1) & & \\
	& & & \ddots & \\
	& & & & N(m)
\end{pmatrix},
N(k)=\begin{pmatrix}
	0 & \cos\frac{\pi k}{n} & 0 & \sin\frac{\pi k}{n} \\
	\cos\frac{\pi k}{n} & 0 & -\sin\frac{\pi k}{n} & 0 \\
	0 & -\sin\frac{\pi k}{n} & 0 & \cos\frac{\pi k}{n} \\
	\sin\frac{\pi k}{n} & 0 & \cos\frac{\pi k}{n} & 0
\end{pmatrix}.\]
On the other hand, if $n=2m$ is even, the normal form contains only $m$ focus-focus blocks $N(k+1/2)$, for $0\le k\le m-1$. In this case, if we change any $1$ entry in $M$ to $-1$, the normal form contains again the hyperbolic block, the focus-focus blocks $N(k)$ for $1\le k\le m-1$ and also an elliptic block, corresponding to the eigenvalues $\ii$ and $-\ii$:
\[S^TMS=\begin{pmatrix}
	0 & 1 & & & & & \\
	1 & 0 & & & & & \\
	& & 1 & 0 & & & \\
	& & 0 & 1 & & & \\
	& & & & N(1) & & \\
	& & & & & \ddots & \\
	& & & & & & N(m-1)
\end{pmatrix}.\]

\subsection{Positive-definite case}\label{sec:real-posdef}
If the matrix is positive-definite, only the elliptic case may appear: this happens because, if $u$ is an eigenvector of $A$ with value $\lambda$,
\[\bar{u}^TMu=\bar{u}^T\Omega_0Au=\lambda\bar{u}^T\Omega_0 u\]
The left-hand side is a positive real, and for $\bar{u}^T\Omega_0 u$ we have
\[\overline{\bar{u}^T\Omega_0 u}=u^T\Omega_0\bar{u}=-\bar{u}^T\Omega_0 u,\]
so it is imaginary, and $\lambda$ must be imaginary.

Also, it is impossible to obtain for these matrices the elliptic blocks with size greater than $2$: if such a block appeared, there would be an element in the diagonal of $S^TMS$ equal to $0$. Let $k$ be its index and $u$ the $k$-th column of $S$. Then $u^TMu=0$, which contradicts $M$ being positive definite. This means that any positive definite symmetric matrix can be diagonalized by a symplectic matrix, which is the result most often referred to as ``Williamson theorem''.

\section{The $p$-adic classification in higher dimensions}\label{sec:num-forms}

In this section we prove \cite[Theorems I and J]{CrePel-integrable}. Our strategy for the $4$-dimensional case extends to any dimension, using the fact that all algebraic extensions of $\Qp$ are extensions by radicals (though it would be needed to take higher order radicals) but for brevity we do not deal with those cases (we expect hundreds or even thousands of possibilities for the model matrices already in dimension $6$, see Theorem \ref{thm:num-forms-lower-bound}). In dimension $10$ or higher, however, the class of a given matrix cannot be determined by a formula involving radicals, as the general equation of degree five or greater is not solvable by radicals.

\begin{definition}[{\cite[Definition 2.6]{CrePel-integrable}}]\label{def:M}
	\letpprime\ and let $n$ be a positive integer. For each partition $P=(a_1,\ldots,a_k)$ of $n$, we define $M(P,p)\in\M_{2n}(\Qp)$ as the block-diagonal matrix whose blocks have sizes $(2a_1,\ldots,2a_k)$ and each block has a form which is itself block-diagonal, with blocks of the form
	\[\begin{pmatrix}
		1 & & & & & & & &\\
		& 0 & 1 & & & & & & \\
		& 1 & 0 & & & & & & \\
		& & & 0 & 1 & & & & \\
		& & & 1 & 0 & & & & \\
		& & & & & \ddots & & \\
		& & & & & & 0 & 1 & \\
		& & & & & & 1 & 0 & \\
		& & & & & & & & p
	\end{pmatrix}.\]
\end{definition}

\begin{theorem}[{\cite[Theorem I]{CrePel-integrable}}]\label{thm:num-forms}
	\letpprime. Let $n$ be a positive integer. The number of $p$-adic families of non-degenerate normal forms of $(2n)$-by-$(2n)$ matrices up to congruence via a symplectic matrix, each family being of the form $r_1M_1+\ldots+r_nM_n$, where $r_i$ are parameters in $\Qp$, grows asymptotically at least with \[\frac{\e^{\pi\sqrt{2n/3}}}{4n\sqrt{3}}.\]
	Explicitly, if $P$ and $P'$ are distinct partitions of $n$ then the matrices $M(P,p)$ and $M(P',p)$ in Definition \ref{def:M} are not equivalent by multiplication by a symplectic matrix.
\end{theorem}

\begin{theorem}[{\cite[Theorem J]{CrePel-integrable}}]\label{thm:num-forms-lower-bound}
	Let $p\in\{2,3,5,7\}$. Let $n$ be a positive integer with $n\le 10$. The number of $p$-adic families of non-degenerate normal forms of $(2n)$-by-$(2n)$ matrices up to congruence via a symplectic matrix, each family being of the form $r_1M_1+\ldots+r_nM_n$, where $r_i$ are parameters in $\Qp$, and the number of $(2n)$-by-$(2n)$ blocks which may appear in those normal forms, are at least as follows:
	
	\upshape\footnotesize
	\begin{tabular}{|*{11}{c|}}\hline
		$2n$ & \multicolumn{2}{c|}{$\R$} & \multicolumn{2}{c|}{$\Q_2$} & \multicolumn{2}{c|}{$\Q_3$} & \multicolumn{2}{c|}{$\Q_5$} & \multicolumn{2}{c|}{$\Q_7$} \\ \cline{2-11}
		& blocks & forms & blocks & forms & blocks & forms & blocks & forms & blocks & forms \\ \hline
		2 & 2 & 2 & 11 & 11 & 5 & 5 & 7 & 7 & 5 & 5 \\ \hline
		4 & 1 & 4 & 145 & 211 & 17 & 32 & 21 & 49 & 17 & 32 \\ \hline
		6 & 0 & 6 & 2 & 1883 & 3 & 123 & 3 & 234 & 9 & 129 \\ \hline
		8 & 0 & 9 & 1 & 21179 & 2 & 495 & 4 & 1054 & 2 & 525 \\ \hline
		10 & 0 & 12 & 2 & 161343 & 3 & 1595 & 3 & 4021 & 3 & 1787 \\ \hline
		12 & 0 & 16 & 1 & 1374427 & 2 & 5111 & 4 & 14493 & 6 & 5874 \\ \hline
		14 & 0 & 20 & 2 & 9232171 & 3 & 14491 & 3 & 47462 & 3 & 17586 \\ \hline
		16 & 0 & 25 & 1 & 65570626 & 2 & 40244 & 4 & 148087 & 2 & 50614 \\ \hline
		18 & 0 & 30 & 2 & 397086458 & 3 & 103484 & 3 & 433330 & 9 & 137311 \\ \hline
		20 & 0 & 36 & 1 & 2469098766 & 2 & 259712 & 4 & 1217761 & 2 & 359463 \\ \hline
	\end{tabular}
	
	\normalsize\itshape
	For comparison, the table includes the exact number of forms in the real case in each dimension.
\end{theorem}

In order to prove Theorem \ref{thm:num-forms} and \ref{thm:num-forms-lower-bound}, we need some lemmas.

\begin{lemma}\label{lemma:irred}
	\letpprime\ and let $n$ be a positive integer. The polynomial \[P(x)=x^n-ap,\] where $\ord_p(a)=0$, is irreducible in $\Qp$.
\end{lemma}

\begin{proof}
	The roots of $P$ have order $1/n$. If $P$ was reducible, a factor should have a subset of the roots whose product has integer order, but this would need all the $n$ roots.
\end{proof}

\begin{lemma}\label{lemma:power}
	\letpprime\ and let $n$ be a positive integer. Let $a,b\in\Qp$ such that $\ord_p(a)=0$ and $\ord_p(b)=1$. If $a$ is an $n$-th power in $\Qp[b^{1/n}]$, then it is an $n$-th power in $\F_p$.
\end{lemma}

\begin{proof}
	Suppose that $a=c^n$ for $c\in\Qp[b^{1/n}]$. We can write
	\[c=c_0+c_1b^{\frac{1}{n}}+\ldots+c_{n-1}b^{\frac{n-1}{n}}\]
	where $c_i\in\Qp$ for all $i$. Raising this to the $n$-th power, we have $c^n=a$ at the left, and $c_0^n$ plus terms of positive order at the right. Then, $a-c_0^n$ has positive order, and as it is in $\Qp$ the order must be at least $1$, and $a\equiv c_0^n\mod p$, as we wanted.
\end{proof}

\begin{lemma}\label{lemma:numblocks}
	Let $p$ be a prime number and let $n$ be a positive integer. There are at least $\gcd(2n,p-1)+\gcd(n,p-1)$, if $n$ is odd, and $\gcd(2n,p-1)$, if $n$ is even, infinite families of blocks of size $2n$ in the normal form of a matrix up to congruence via a symplectic matrix, where each family is of the form $r_1M_1+\ldots+r_nM_n$.
\end{lemma}

\begin{proof}
	Consider the polynomial $P(x)=x^{2n}-ap$ where $\ord_p(a)=0$. This is irreducible by Lemma \ref{lemma:irred}, so it will give a block of size $2n$ in the normal form. This block may not be unique up to congruence via a symplectic matrix (as happens in Propositions \ref{prop:elliptic} and \ref{prop:exotic}), but, in analogy with the proofs of those results, two blocks corresponding to different $a$ will be in the same family only if the roots of the polynomials are in the same extension of $\Qp$. Suppose that this happens for $a_1$ and $a_2$. In particular, $(a_1p)^{1/2n}$ and $(a_2p)^{1/2n}$ are in the same extension, that is,
	\[\left(\frac{a_2}{a_1}\right)^{\frac{1}{2n}}\in\Qp[(a_1p)^{\frac{1}{2n}}]\]
	By Lemma \ref{lemma:power}, $a_2/a_1$ must be a $2n$-th power in $\F_p$. This implies that the number of families of blocks is at least the cardinality of $\F_p^*$ modulo $2n$-th powers, which is $\gcd(2n,p-1)$, because that group is cyclic of order $p-1$.
	
	If $n$ is odd, we also consider \[Q(x)=x^{2n}-a^2p^2=(x^n+ap)(x^n-ap).\] The two factors are again irreducible and it also gives a block of size $2n$ (one factor comes from changing the sign of $x$ in the other). Two blocks for $a_1$ and $a_2$ are in the same family only if $(a_1p)^{1/n}$ and $(a_2p)^{1/n}$ are in the same extension, that is,
	\[\left(\frac{a_2}{a_1}\right)^{\frac{1}{n}}\in\Qp[(a_1p)^{\frac{1}{n}}]\]
	(note that choosing $-ap$ instead of $ap$ gives the same extension because $(-1)^{1/n}=-1$). Again by Lemma \ref{lemma:power}, $a_2/a_1$ is an $n$-th power in $\F_p$. So the number of families is now the cardinality of $\F_p^*$ modulo $n$-th powers, which is $\gcd(n,p-1)$.
\end{proof}

\begin{remark}\label{rem:num-forms-real}
	Concerning the number of families of normal forms (instead of just blocks), in the real case, supposing that there are $k$ focus-focus blocks, there are $2n-4k$ variables left, which can be distributed between hyperbolic and elliptic blocks in $n-2k+1$ ways. The total number of forms is
	\[\sum_{k=0}^{m}2m-2k+1=2m^2-m(m+1)+m+1=m^2+1\]
	if $n=2m$ and
	\[\sum_{k=0}^{m}2m+1-2k+1=m(2m+1)-m(m+1)+m+1=m^2+m+1\]
	if $n=2m+1$.
\end{remark}

\begin{proof}[Proof of Theorem \ref{thm:num-forms}]
	Lemma \ref{lemma:numblocks} tells us that there is at least one block with each even size. Hence, the number of normal forms is at least the number of partitions of $n$ in positive integers, that grows with $\e^{\pi\sqrt{2n/3}}/4n\sqrt{3}$ by the Hardy-Ramanujan formula \cite{HarRam}. In order to find the exact formulas of the matrices, we need to devise, for each partition, a matrix in $\M_{2n}(\Qp)$ with the product of the corresponding factors as characteristic polynomial. This can be done with the same strategy as in Section \ref{sec:real-example}, and gives the matrix $M(P,p)$ for each partition $P$.
\end{proof}

\begin{proof}[Proof of Theorem \ref{thm:num-forms-lower-bound}]
	The numbers in the real case follow from Remark \ref{rem:num-forms-real}. The numbers of blocks in the $p$-adic case follow from Theorems \ref{thm:num-forms2} and \ref{thm:num-forms1} in dimension 2 and 4, respectively; for higher dimensions we use Lemma \ref{lemma:numblocks}. The numbers of forms are obtained making a sum over the partitions of $n$.
\end{proof}

\subsection{Remarks and applications}

Theorem \ref{thm:num-forms} could be strengthened by using that there is not only \textit{one} block of each size, but this would imply making a sum over the partitions as in Theorem \ref{thm:num-forms-lower-bound}. We do not know how to make that for general $n$.

From the point of view of symplectic geometry and topology of integrable systems, which is the main motivation of the authors to write this paper, currently the only known global \emph{symplectic} classifications of integrable systems which include physically intriguing local models (that is, essentially non-elliptic models) concern dimension $4$ \cite{PPT,PelVuN-semitoric,PelVuN-construct} in the real case. These real classifications include for example the coupled angular momentum \cite{LeFPel} and the Jaynes-Cummings model \cite{PelVuN-spin}. Hence, in the $p$-adic case, with hundreds of local models (\cite[Theorem B]{CrePel-integrable}), we expect that the $4$-dimensional case is already extremely complicated and that the $2n$-dimensional case, $n\ge 3$, is out of reach (since it is out of reach in the real case with only a very small proportion of local models in comparison, see \cite[Theorem C]{CrePel-integrable}).

In dimension $4$ the authors analyzed one of these systems, the $p$-adic Jaynes-Cummings model \cite{CrePel-JC}, whose treatment is very extensive compared to its real counterpart, as expected. Although as we said, a classification of $p$-adic integrable systems in dimension $4$, extending \cite{PPT,PelVuN-semitoric,PelVuN-construct}, seems out of reach, the present paper settles completely the first step: understanding explicitly $p$-adic local models. The proofs of \cite{PPT,PelVuN-semitoric,PelVuN-construct} are based on gluing local models.

\section{Application to classical mechanical systems} \label{sec:JC}

In this section we carry out the proof of \cite[Corollary 6.1]{CrePel-integrable} using the methods developed in this paper.

By \cite[Theorem A]{CrePel-integrable}, $F_1$ and $F_2$ are locally symplectomorphic to one of the possibilities listed in its statement, so it is enough to see that it is the same for $p\equiv 1\mod 4$ and different otherwise. For the first normal form $F_1$, we have that
\[\Omega_0^{-1}\dd^2 (rJ_1+sH_1)=
\begin{pmatrix}
	0 & -2 & 0 & 0 \\
	2 & 0 & 0 & 0 \\
	0 & 0 & 0 & -2 \\
	0 & 0 & 2 & 0
\end{pmatrix}
\begin{pmatrix}
	r & 0 & 0 & 0 \\
	0 & r & 0 & 0 \\
	0 & 0 & s & 0 \\
	0 & 0 & 0 & s
\end{pmatrix}=
\begin{pmatrix}
	0 & -2r & 0 & 0 \\
	2r & 0 & 0 & 0 \\
	0 & 0 & 0 & -2s \\
	0 & 0 & 2s & 0
\end{pmatrix}\]
whose eigenvalues are $\pm 2\ii r$ and $\pm 2\ii s$. If $p\equiv 1\mod 4$, we are in the situation of Proposition \ref{prop:hyperbolic}, so this is in case (1) of \cite[Theorem A]{CrePel-integrable} locally symplectomorphic to $x^2+\xi^2$. Otherwise, we are in the situation of Proposition \ref{prop:elliptic}, and for $\lambda=2\ii r$,
\[u^T\Omega_0\bar{u}=
\begin{pmatrix}
	\ii & 1 & 0 & 0
\end{pmatrix}
\begin{pmatrix}
	0 & \frac{1}{2} & 0 & 0 \\
	-\frac{1}{2} & 0 & 0 & 0 \\
	0 & 0 & 0 & \frac{1}{2} \\
	0 & 0 & -\frac{1}{2} & 0 
\end{pmatrix}
\begin{pmatrix}
	-\ii \\ 1 \\ 0 \\ 0
\end{pmatrix}=\ii\]
and
\[\frac{2\lambda a}{u^T\Omega_0\bar{u}}=4ra.\]

We need to find $a$ and $b$ such that $ab=4r^2$ and $4ra\in\DSq(\Qp,1)$, or equivalently $r/a\in\DSq(\Qp,1)$. Taking $b=ac$, we have that $c=4r^2/a^2$ is a square; moreover, if $p\equiv 3\mod 4$, $r/a$ has even order, hence $4\mid\ord_p(c)$, and in the set \[\Big\{1,-1,p,-p,p^2\Big\},\] the $c$ that we need is $1$. If $p=2$, the only $c$ that is square is $1$. The same reasoning holds for $s$ instead of $r$, so this critical point is in case (1) with $c_1=c_2=1$.

Respecting to $F_2$, we get
\[\Omega_0^{-1}\dd^2 (rJ_2+sH_2)=
\begin{pmatrix}
	0 & -2 & 0 & 0 \\
	2 & 0 & 0 & 0 \\
	0 & 0 & 0 & -2 \\
	0 & 0 & 2 & 0
\end{pmatrix}
\begin{pmatrix}
	0 & s & 0 & r \\
	s & 0 & -r & 0 \\
	0 & -r & 0 & s \\
	r & 0 & s & 0
\end{pmatrix}=
\begin{pmatrix}
	-2s & 0 & 2r & 0 \\
	0 & 2s & 0 & 2r \\
	-2r & 0 & -2s & 0 \\
	0 & -2r & 0 & 2s
\end{pmatrix}\]
has as eigenvalues $\pm 2s\pm 2\ii r$. If $p\equiv 1\mod 4$, this is again in case (1) with $c_1=c_2=1$, otherwise it is in case (2) with $c=-1$.

\section{Examples}\label{sec:examples}

In this section we show examples which illustrate our theorems, so that they can be understood more concretely.

\begin{example}\label{ex:matrix2}
	This example follows the method given in the proof of Theorem \ref{thm:williamson4} in order to find the normal form of a symmetric matrix. Let $M$ be the following symmetric matrix:
	\[M=\begin{pmatrix}
		1 & 2 & 3 & 4 \\
		2 & 5 & 6 & 7 \\
		3 & 6 & 8 & 9 \\
		4 & 7 & 9 & 10
	\end{pmatrix}.\]
	
	The characteristic polynomial of $A=\Omega_0^{-1}M$ is $t^4-6t^2-2$. We start with $p=2$. In order to classify $M$, we need to find two things: the family of the normal form, and the normal form itself. We first calculate
	\[\lambda^2=\frac{6\pm\sqrt{36+8}}{2}=3\pm\sqrt{11}.\]
	As $11\equiv 3\mod 8$, $\lambda^2\notin\Q_2$ and we are in case (2) or (3) of Theorem \ref{thm:williamson4} with $c=3$. In order to find which one, we need to check whether $3\pm\sqrt{11}$ is a square in $\Q_2[\sqrt{3}]$: it can be written as $3\pm\sqrt{11/3}\sqrt{3}$, where $\sqrt{11/3}\in\Q_2$. Applying the criterion on Proposition \ref{prop:squares2}(4), we see that \[\ord(3)=\ord(\sqrt{11/3})=0,\] so it is not a square and we are in case (3).
	
	The next step is to find $t_1$ and $t_2$: this is the class of $3+\sqrt{11}$ in $\Q_2[\sqrt{3}]$ modulo squares. This can be found with the procedure in the proof of Corollary \ref{cor:squares2}(2): we need to multiply by $1+\sqrt{3}$ to make $\ord(a)\ne\ord(b)$, which gives
	\[(3+\sqrt{\frac{11}{3}}\sqrt{3})(1+\sqrt{3})=3+\sqrt{33}+\left(3+\sqrt{\frac{11}{3}}\right)\sqrt{3}.\]
	Writing this again in the form $a+b\sqrt{3}$, it leads to $\ord(b)-\ord(a)=2$, $\ord(a)\equiv 1\mod 2$ (so we have to multiply by $2$) and $b/4a+\digit_2(a)\equiv 1\mod 2$ (so we have to change sign). The class is $-2(1+\sqrt{3})$, that is, $t_1=t_2=-2$. This is not a class in the table, so we need to take its pair, $t_1=1$ and $t_2=1$.
	
	It is only left to find $a$ and $b$. The two possible classes have $b=0$, but one has $a=1$ and the other $a=-1$. To find the correct one, we go to the formula in Proposition \ref{prop:exotic}.
	\[\frac{a\alpha\gamma(b+\alpha)}{u^T\Omega_0\hat{u}}=\frac{a\alpha^2\gamma}{u^T\Omega_0\hat{u}}=\frac{3a\sqrt{-2(1+\sqrt{3})}}{u^T\Omega_0\hat{u}}\in\DSq(\Q_2[\sqrt{3}],2(1+\sqrt{3})).\]
	The class of this number, with $a=1$, is $-(1+\sqrt{3})$. As we see in Table \ref{table:3}, the position in row $\gamma^2=-2(1+\sqrt{3})$ and column $-(1+\sqrt{3})$ is unmarked, which means that the number is not in \[\DSq(\Q_2[\sqrt{3}],2(1+\sqrt{3})),\] and we need to take $a=-1$. This finishes the classification of the family of the normal form:
	\[r\begin{pmatrix}
		-1 & 0 & 0 & 0 \\
		0 & 1 & 0 & -3 \\
		0 & 0 & -\frac{1}{3} & 0 \\
		0 & -3 & 0 & 3
	\end{pmatrix}
	+s\begin{pmatrix}
		0 & 0 & 1 & 0 \\
		0 & 3 & 0 & -3 \\
		1 & 0 & 0 & 0 \\
		0 & -3 & 0 & -9
	\end{pmatrix}.\]
	To find the concrete normal form, we need to put $\lambda$ in the form $(r+s\alpha)\gamma$, and the result is
	\[r=\frac{1}{2}\left(\sqrt{\frac{-3+\sqrt{11}+3\sqrt{3}-\sqrt{33}}{2}}+\sqrt{\frac{-3-\sqrt{11}-3\sqrt{3}-\sqrt{33}}{2}}\right)\]
	\[s=\frac{1}{2\sqrt{3}}\left(\sqrt{\frac{-3+\sqrt{11}+3\sqrt{3}-\sqrt{33}}{2}}-\sqrt{\frac{-3-\sqrt{11}-3\sqrt{3}-\sqrt{33}}{2}}\right)\]
	This is a common pattern in this example and the following two: the family of normal forms is very simple, but the concrete normal form is much more complicated.
\end{example}

\begin{example}\label{ex:matrix3}
	Now we classify the same matrix with $p=3$. $11$ is still not a square in $\Q_3$, so we are in case (2) or (3) with $c=-1$. We write $\lambda^2$ as $3+\ii\sqrt{-11}$, where $\sqrt{-11}\in\Q_3$, and we need to check whether this is a square in $\Q_3[\ii]$. We use Proposition \ref{prop:squares3}(1): $\min\{\ord(a),\ord(b)\}\equiv 0\mod 2$ and $a^2+b^2=-2$ is a square in $\Q_3$, so $\lambda^2$ is a square in $\Q_3[\ii]$ and we are in case (2). In this case $c=-1$ is the only parameter we need. This is the focus-focus family
	\[\begin{pmatrix}
		0 & s & 0 & r \\
		s & 0 & -r & 0 \\
		0 & -r & 0 & s \\
		r & 0 & s & 0
	\end{pmatrix}\]
	with
	\[r=\frac{\sqrt{3+\sqrt{11}}-\sqrt{3-\sqrt{11}}}{2\ii},s=\frac{\sqrt{3+\sqrt{11}}+\sqrt{3-\sqrt{11}}}{2}\]
\end{example}

\begin{example}\label{ex:matrix5}
	For the same matrix and $p=5$, we see that $\lambda^2=3+\sqrt{11}\in\Q_5$, so we are in case (1). $c_1$ is given by the class of $\lambda^2$ modulo a square in $\Q_5$.
	\[\lambda^2=3+\sqrt{11}\equiv 3+1=4\mod 5\]
	is a square, so $c_1=1$ (actually, $c_1$ is in the class of $-\lambda^2$, but as $p\equiv 1\mod 4$ it is equivalent to take $\lambda^2$). Analogously
	\[\mu^2=3-\sqrt{11}\equiv 3-1=2\mod 5\]
	so $c_2=2$ (which is $c_0$ for $p=5$) or $c_2=2p^2=50$. To know which one, we go to Proposition \ref{prop:elliptic}:
	\[\frac{2\mu a}{u^T\Omega_0\bar{u}}=\ldots 203033a\in\DSq(\Q_5,2)\]
	which means, by Proposition \ref{prop:images-extra}(2), that $a$ has even order, and since $\ord(ab)=0$, \[\ord(c_2)=\ord(b)-\ord(a)=-2\ord(a)\] is multiple of $4$, so this leaves $c_2=2$. The normal form is
	\[\begin{pmatrix}
		r & 0 & 0 & 0 \\
		0 & r & 0 & 0 \\
		0 & 0 & s & 0 \\
		0 & 0 & 0 & 2s
	\end{pmatrix}\]
	with
	\[r=\sqrt{3+\sqrt{11}},s=\sqrt{\frac{\sqrt{11}-3}{2}}.\]
\end{example}

\begin{example}\label{ex:matrix5-deg}
	Let $M$ be the following matrix:
	\[M=\begin{pmatrix}
		2 & 6 & -2 & -3 \\
		6 & 11 & 1 & -5 \\
		-2 & 1 & -6 & -2 \\
		-3 & -5 & -2 & 3
	\end{pmatrix}\]
	The characteristic polynomial of $A=\Omega_0^{-1}M$ is $(t^2-5)^2$, so this matrix has repeated eigenvalues: \[\{\sqrt{5},\sqrt{5},-\sqrt{5},-\sqrt{5}\}.\]
	
	For $p=5$, $\sqrt{5}\notin\Q_5$, hence it is in case (3) of Theorem \ref{thm:williamson4-deg}. We need to write $\lambda=r\sqrt{c}$ for $c\in Y_p$: this leads to $c=5$ and $r=1$. The value of $a$ must be $1$ or $2$, and from \eqref{eq:prod-proof-deg} we obtain that $a=1$ gives the result in $\DSq(\Q_5,-5)$. So the normal form is
	\[\begin{pmatrix}
		1 & 0 & 0 & 1 \\
		0 & 0 & 5 & 0 \\
		0 & 5 & 1 & 0 \\
		1 & 0 & 0 & 0
	\end{pmatrix}.\]
\end{example}

\begin{example}\label{ex:matrix11-deg}
	The same matrix with $p=11$: now $\sqrt{5}\in\Q_{11}$, so it is in case (2) of Theorem \ref{thm:williamson4-deg}, with $r=\lambda=\sqrt{5}$:
	\[\begin{pmatrix}
		0 & \sqrt{5} & 0 & 0 \\
		\sqrt{5} & 0 & 1 & 0 \\
		0 & 1 & 0 & \sqrt{5} \\
		0 & 0 & \sqrt{5} & 0
	\end{pmatrix}.\]
\end{example}

\begin{remark}
	The matrices in the previous examples are precisely the Hessians of the functions $f_1$ and $f_2$ in \cite[Example 7.1]{CrePel-integrable}, hence the results of that example follow from these examples.
\end{remark}

\end{document}